\documentclass[12pt]{article}
\input epsf.tex

\usepackage{graphicx}
\usepackage{amsmath,amsthm,amsfonts,amscd,amssymb,comment,eucal,latexsym,mathrsfs}
\usepackage{stmaryrd}
\usepackage[all]{xy}

\usepackage{epsfig}

\usepackage[all]{xy}
\xyoption{poly}
\usepackage{fancyhdr}
\usepackage{wrapfig}
\usepackage{epsfig}

\theoremstyle{plain}
\newtheorem{thm}{Theorem}[section]
\newtheorem{prop}[thm]{Proposition}
\newtheorem{lem}[thm]{Lemma}
\newtheorem{cor}[thm]{Corollary}

\theoremstyle{definition}
\newtheorem{defn}{Definition}
\theoremstyle{remark}
\newtheorem{remark}{Remark}
\newtheorem{example}{Example}

\newtheorem{notation}{Notation}

\advance \topmargin by -\headheight
\advance \topmargin by -\headsep
\evensidemargin \oddsidemargin
    \def\E{{\mathbb{E}}} \def\F{{\mathbb{F}}}     \def\K{{\mathbb{K}}}   \def\N{{\mathbb{N}}}  \def\P{{\mathbb{P}}} \def\Q{{\mathbb{Q}}} \def\R{{\mathbb{R}}}        \def\Z{{\mathbb{Z}}}

     \def\barf{{\bar{f}}}             \def\bs{{\bar{s}}} \def\bt{{\bar{t}}}      

      \def\bfG{{\bf{G}}}            \def\bfS{{\bf{S}}}    \def\bfW{{\bf{W}}}   

\def\bfa{{\bf{a}}} \def\bfb{{\bf{b}}} \def\bfc{{\bf{c}}}  \def\bfe{{\bf{e}}}                 \def\bfv{{\bf{v}}}    \def\bfz{{\bf{z}}}

\def\bfxi{\mbox{\boldmath $\xi$}}

\def\bfpi{\mbox{\boldmath $\pi$}}

\def\bfsig{\mbox{\boldmath $\sigma$}} 
\def\bftau{\mbox{\boldmath $\tau$}} 
\def\bfzeta{\mbox{\boldmath $\zeta$}}

\def\cA{{\mathcal{A}}} \def\cB{{\mathcal{B}}}    \def\cF{{\mathcal{F}}}   \def\cI{{\mathcal{I}}}  \def\cK{{\mathcal{K}}}  \def\cM{{\mathcal{M}}}  \def\cO{{\mathcal{O}}} \def\cP{{\mathcal{P}}} \def\cQ{{\mathcal{Q}}}       \def\cX{{\mathcal{X}}} \def\cY{{\mathcal{Y}}} \def\cZ{{\mathcal{Z}}}

               \def\tcP{{\tilde{\mathcal{P}}}}

                       \def\hx{{\hat{x}}}  

                    \def\sU{{\mathscr{U}}}  \def\sW{{\mathscr{W}}}   

  \def\tC{{\tilde{C}}}                       

     \def\tf{{\tilde{f}}}                  \def\tx{{\tilde{x}}}  

               \def\tcP{{\tilde{\cP}}}

       \def\tphi{{\widetilde{\phi}}} \def\tpsi{{\widetilde{\psi}}}

\newcommand{\G}{\Gamma}
\newcommand{\Ga}{\Gamma}

\newcommand{\Si}{\Sigma}
\newcommand{\eps}{\epsilon}

\renewcommand\b{\beta}
\renewcommand\a{\alpha}
\newcommand\g{\gamma}
\newcommand\s{\sigma}
\renewcommand\d{\delta}
\renewcommand\k{\kappa}

\renewcommand\l{\lambda}

\newcommand\Abel{\operatorname{Abel}}

\newcommand\Cluster{\operatorname{Cl}}
\newcommand\cov{\operatorname{cov}}

\newcommand\diam{\operatorname{diam}}

\newcommand\Graphs{{\operatorname{Graphs}}}
\newcommand\Hom{\operatorname{Hom}}

\newcommand\Leb{\operatorname{Leb}}
\newcommand\MAXIND{\operatorname{MAXIND}}

\newcommand\Prob{\operatorname{Prob}}
\newcommand\Proj{\operatorname{Proj}}

\newcommand\Res{\operatorname{Res}}

\newcommand\sym{\operatorname{sym}}

\def\cc{{\curvearrowright}}
\newcommand{\resto}{\upharpoonright}

  \newcommand{\Addresses}{{
  \bigskip
  \footnotesize

  L.~Bowen, \textsc{Department of Mathematics, University of Texas at Austin,
    Austin, Texas 78712}\par\nopagebreak
  \textit{E-mail address}, L.~Bowen: \texttt{lpbowen@math.utexas.edu}

}}
  
\begin{document}
\title{Sofic homological invariants and the Weak Pinsker Property}
\author{Lewis Bowen \\ University of Texas at Austin}
\maketitle

\begin{abstract}
A probability-measure-preserving transformation has the Weak Pinsker Property (WPP) if for every $\eps>0$ it is measurably conjugate to the direct product of a transformation with entropy $<\eps$ and a Bernoulli shift. In a recent breakthrough, Tim Austin proved that every ergodic transformation satisfies this property. Moreover, the natural analog for amenable group actions is also true. By contrast, this paper provides a counterexample in which the group $\G$ is a non-abelian free group and the notion of entropy is sofic entropy. The counterexample is a limit of hardcore models on random regular graphs. In order to prove that it does not have the WPP, this paper introduces new measure conjugacy invariants based on the growth of homology of the model spaces of the action. The main result is obtained by showing that any action with the WPP has subexponential homology growth in dimension $0$, while the counterexample has exponential homology growth in dimension $0$.

\end{abstract}

\noindent
{\bf Keywords}: sofic groups, entropy theory, Weak Pinsker Property\\
{\bf MSC}:37A35\\

\noindent
\tableofcontents

\section{Introduction}

This paper is concerned with the general problem of classifying measure-preserving actions of countable groups on probability spaces. To be precise, fix a countable group $\G$ and let $(X,\mu_X), (Y,\mu_Y)$ be standard probability spaces. Then two actions $\G \cc (X,\mu_X), \G \cc (Y,\mu_Y)$ are {\bf measurably conjugate} or {\bf isomorphic} if there exists a measure-preserving isomorphism $\Phi:(X,\mu_X) \to (Y,\mu_Y)$ that intertwines the actions in the sense that $\Phi(gx)=g\Phi(x)$ for every $g\in \G$ and a.e. $x \in X$. 

Anti-classification theorems convincingly show  it is not possible to classify all actions up to measure-conjugacy \cite{foreman-weiss, foreman-rudolph-weiss2}. In spite of this, there are interesting structural results. To explain these, it is necessary to introduce Bernoulli shifts, which are some of the most fundamental actions. Let $(K,\k)$ be a standard probability space and equip the product space $K^\G=\{x:\G \to K\}$ with the product measure $\k^\G$. The group acts on this space by $(gx)(f) = x(g^{-1}f)$. This action is called the {\bf Bernoulli shift over $\G$ with base $(K,\k)$}. 

A classical example of a general structural result is Sinai's factor Theorem. It states that, when $\G=\Z$, any action with positive entropy factors onto a Bernoulli shift. Moreover, the factor can be chosen so that the relative entropy is zero. These statements have recently been generalized to arbitrary countable groups  by Seward \cite{MR4066472}. 

Another example comes from Pinsker. In 1960, Pinsker conjectured that any ergodic measure-preserving transformation $T:X \to X$ of a standard probability space $(X,\mu)$ is measurably conjugate to a direct product $T \approx S \times U$ such that $S$ has zero entropy and $U$ is a $K$-transformation (which means that every nontrivial factor of $U$ has positive entropy) \cite{MR0152628}. This was falsified by Ornstein \cite{MR0399416, MR0330416}. The study of such systems led Thouvenot to introduce the {\bf Weak Pinsker Property} (WPP) for measure-preserving transformations: $T$ has the WPP if for every $\eps>0$, $T$ is measurably conjugate to a direct product $S \times U$ such that $S$ has entropy $<\eps$ and $U$ is isomorphic to a Bernoulli shift. He asked whether all ergodic transformations have the WPP and proved important structural properties of this class  \cite{MR0453982}. 

In recent breakthrough work, Tim Austin has proven that indeed every ergodic transformation has the WPP  \cite{MR3905465}. Moreover, the analogous statement for measure-preserving actions of amenable groups is also true. 

The purpose of this paper is to give an example of an ergodic action of a non-abelian free group without the WPP.  In this context there are two main entropy notions: sofic and Rokhlin. Sofic entropy was initiated in \cite{bowen-jams-2010} and Rokhlin entropy in \cite{seward-kreiger-1} (see \cite{MR3966832, bowen-survey} for an introduction and survey). This paper uses sofic entropy although the results also apply with Rokhlin entropy because it upper bounds sofic entropy.

The sofic entropy of a measure-preserving action $\bfa:=\G \cc (X,\mu)$ depends apriori  on a choice of sofic approximation $\Si$ to $\G$. So it will be referred to here as {\bf $\Si$-entropy} and denoted by $h_\Si(\bfa)$. The action has the {\bf Weak Pinsker Property (WPP) with respect to $\Si$} if for every $\eps>0$, $\bfa$ is isomorphic to a direct product $\bfb \times \bfc$ such that $\bfc$ has $\Si$-entropy $<\eps$ and $\bfb$ is isomorphic to a Bernoulli shift. For example, if $\bfa$ has nonpositive $\Si$-entropy, then it automatically has the WPP because $\bfb$ is allowed to be the trivial action (which can be thought of as a Bernoulli shift with trivial base space). The main result of this paper is:

\begin{thm}\label{thm:main1}
Let $\F_r$ denote the free group of rank $r$. Then there is an $r_0$ such that for all $r>r_0$, there exists a sofic approximation $\Si$ to $\F_r$ and an ergodic essentially free action $\F_r \cc (X,\mu)$ that does not have the Weak Pinsker Property with respect to $\Si$.
\end{thm}

\begin{remark}
An action $\bfa$ has the {\bf Weak Pinsker Property with respect to Rokhlin entropy} if for every $\eps>0$, $\bfa$ is isomorphic to a direct product $\bfb \times \bfc$ such that $\bfb$ has Rokhlin entropy $<\eps$ and $\bfc$ is isomorphic to a Bernoulli shift. Because Rokhlin entropy upper bounds $\Si$-entropy (for every $\Si$), this property is apriori stronger than the WPP with respect to $\Si$. In particular, the action in Theorem \ref{thm:main1} does not have the WPP with respect to Rokhlin entropy.
\end{remark}

\subsection{Homological measure-conjugacy invariants}
The proof of Theorem \ref{thm:main1} is in two steps, the first of which is a construction of a family of new measure-conjugacy invariants based on the asymptotic homology of model spaces. Here is a brief sketch in the special case that $\mu$ is a shift-invariant measure on $\cX^\G$ where $\cX$ is a finite alphabet. In this case, the sofic approximation $\Si$ is a sequence $\Si=\{\s_n\}_{n \in \N}$ of maps $\s_n:\G \to \sym(V_n)$ where each $V_n$ is a finite set and $\sym(V_n)$ is the symmetric group of $V_n$. For every open neighborhood $\cO$ of $\mu$ in the space of probability measures on $\cX^\G$ there is a subset $\Omega(\cO,\s_n) \subset \cX^{V_n}$ consisting of vertex-labelings whose ``empirical measure'' is in $\cO$. The sets $\Omega(\cO,\s_n)$ equipped with the normalized Hamming metric are called {\bf model spaces}. The $\Si$-entropy is the exponential rate of growth of the cardinalities of these model spaces. 

Given a bound $\k>0$, each model space $\Omega(\cO,\s_n)$ is the vertex set of a simplicial complex whose $d$-simplices consist of subsets $S \subset \Omega(\cO,\s_n)$ of cardinality $d+1$ such that the distance between any two elements of $S$ is bounded by $\k$. Homology is usually defined as cycles mod boundaries. That is also true here with the caveat that the boundaries are defined using parameters $\k' \ge \k$ and $\cO' \supset \cO$ in place of $\k,\cO$. So the homology group of the $n$-th model space depends on four parameters $\k,\k',\cO,\cO'$ in addition to $\s_n$. The asymptotic behavior of these homology groups provide new invariants. This idea was inspired by Tim Austin's paper \cite{MR3543677} which gave an asymptotic notion of connectedness for model spaces. That notion is equivalent to the asymptotic triviality of the $0$-dimensional homology groups.

One of the new invariants, denoted $b_{\Si,0}(\bfa)$, is the exponential growth rate of the $0$-th betti numbers of the model spaces. Intuitively, it estimates the growth rate of the number of ``clusters'' of good models. If an action $\bfa$ has the Weak Pinsker Property with respect to $\Si$ then $b_{\Si,0}(\bfa)=0$. This is because the model spaces for a direct product of the form $\bfb \times \bfc$ where $\bfb$ is Bernoulli contract (in a coarse sense) 
to model spaces for $\bfc$ and $b_{\Si,0}(\bfc)$ is bounded by the $\Si$-entropy of $\bfc$.

\subsection{An action with positive zero-dimensional homology growth}\label{sec:intro-markov}
To finish the proof of Theorem \ref{thm:main1}, the next result suffices.

\begin{thm}\label{thm:indep0}
There exists $r_0$ such that if $r>r_0$ then there exists a sofic approximation $\Si$ to $\F_r$ and an invariant measure $\mu$ on the shift space $\{0,1\}^{\F_r}$ such that $b_{\Si,0}(\F_r\cc (\{0,1\}^{\F_r},\mu)) >0$. In particular, $\G \cc (\{0,1\}^\G,\mu)$ does not have the Weak Pinsker Property.
\end{thm}

The example is based on the geometry of the space of independent subsets of random regular graphs. To be precise, let $G=(V,E)$ be a graph. A subset $W \subset V$ is {\bf independent} if there does not exist an edge between any two vertices of $W$. The {\bf density} of $W$ is $\#W/\#V$. The maximum density of an independent set is denoted $\a(G)$.

Fix an even integer $d \ge 3$ and consider choosing a $d$-regular graph $\bfG_{d,n}$ on $n$ vertices uniformly at random (amongst all $d$-regular graphs on $n$ vertices). The first moment method shows that $\a(\bfG_{d,n})$ is bounded above by $2\log(d)/d + o(\log(d)/d)$ with high probability as $n\to\infty$ \cite{MR1864966}. By a non-constructive argument using Azuma's inequality, Frieze-\L uczak obtained a matching lower bound \cite{MR1142268}. More recently, it was shown in \cite{MR3161470} that the limit $\lim_{n\to\infty} \E[\a(\bfG_{d,n})]$ exists and an explicit formula was obtained in \cite{MR3689942} by a deep study of the structure of high density independent sets.

There are no known polynomial-time algorithms for constructing independent subsets of $\bfG_{d,n}$ with density larger than $\log(d)/d$. It is argued in \cite{MR3385742} that a reason for this is that there are many independent subsets $I$ with density between $\log(d)/d$ and $2\log(d)/d$ that are maximal in the sense that they are not properly contained in any other independent subsets. Moreover, it is often the case that there does not exist a subset $I' \subset I$ with density larger than $\log(d)/d +\eps$ which is contained in an independent subset with density larger than the density of $I$. So local perturbations cannot be used to increase the density of a given independent subset. To be precise, the paper \cite{MR3385742} studies Erd\"os-Renyi style sparse graphs. However, the same ideas  can be adapted to regular graphs.


Another feature established in \cite{MR3385742} is that the space of independent sets with a fixed density in between $\log(d)/d$ and $2\log(d)/d$ ``shatters'' into exponentially many clusters separated by macroscopic gulfs. A similar phenomenon is used in  \cite{MR3359490} to show that no `local' algorithm can produce independent subsets of $\bfG_{d,n}$ with density larger than $(1 + \frac{1}{\sqrt{2}})\log(d)/d + \eps$. This was improved to $\log(d)/d + \eps$ in \cite{MR3650409}. Shattering is used here to obtain an action of the free group with positive $b_{\Si,0}$.

In order to explain how to utilize these results to obtain Theorem \ref{thm:main1}, let $\F_r = \langle a_1,\ldots, a_r\rangle$ be the free group of rank $r\ge 2$. Given a homomorphism $\s: \F_r \to \sym(n)$,  let $G(\s)$ be the multi-graph with vertex set $[n]$ and edges $\{v, \s(a_i)v\}$ (over $v\in [n], 1\le i \le r$). The {\bf permutation model} is the random graph $G(\bfsig_n)$ where $\bfsig_n$ is a uniformly random homomorphism from $\F_r$ to $\sym(n)$. By \cite{MR1909503} the permutation model and the configuration model used in \cite{MR1142268, MR3359490} to study $\bfG_{2r,n}$ are contiguous. This allows results about $\bfG_{2r,n}$ to be transferred to $G(\bfsig_n)$. 

A result of Bollob\'as  \cite{MR595929} implies that, with high probability, $\bfG_{2r,n}$ has few short cycles. Together with the contiguity theorem, this shows the existence of a sofic approximation $\Si=\{\s_n\}_{n=1}^\infty$ to $\F_r$ such that the deterministic graph $G(\s_n)$ and the random graph $G(\bfsig_n)$ have (with high probability) approximately the same number of independent sets (of some fixed density). Moreover, the space of independent subsets of $G(\s_n)$ at a certain fixed density shatters.

An action of the free group is obtained using a non-constructive compactness argument whose proof is related to the proof of the Variational Principle in \cite{kerr-li-variational}. The end result is an invariant measure $\mu$ on the shift space $\{0,1\}^{\F_r}$ such that a significant fraction of independent subsets at a certain fixed density of $G(\s_n)$ are good models for $\mu$. From this, we conclude $b_{\Si,0}(\mu)>0$.



\subsection{A brief guide to the paper}

\begin{itemize}
\item \S \ref{sec:notation} explains notational conventions.
\item \S \ref{sec:entropy} reviews sofic entropy and fixes notation used throughout the paper.
\item \S \ref{sec:homology} defines the new homological invariants.
\item \S \ref{sec:prelim}-\ref{sec:proof1} contain the proof that the new invariants are in fact invariant.
\item \S \ref{sec:computations} contains proofs that the new invariants trivialize when the group is amenable or the action is Bernoulli. Also in this section is a proof that if $\bfa$ has the WPP with respect to $\Si$ then $b_{\Si,0}(\bfa)=0$.
\item \S \ref{sec:markov} proves Theorem \ref{thm:indep0}.
\item \S \ref{sec:questions} is a list of open problems related to the new invariants and the Weak Pinsker Property.

\end{itemize}

{\bf Acknowledgments}. The homological invariants introduced in this paper are inspired by \cite{MR3543677}. The techniques for proving that they are measure-conjugacy invariants are simplified versions of techniques introduced in \cite{MR3542515}. Brandon Seward suggested that it might be possible to use the shattering property to give a counterexample to the WPP. I would like to thank Dylan Airey, Tim Austin and Brandon Seward for many conversations related to this paper. Also thanks to IPAM, the Institute for Pure and Applied Mathematics and the UCLA mathematics dept.  A significant part of this paper was written while I was visiting during the Quantitative Linear Algebra Program in the Spring of 2018. Also thanks to Tim Austin and Chris Shriver for catching errors in previous versions. The author was supported in part by NSF grant DMS-1500389 and a fellowship from the Simons Foundation.

\section{Notation and conventions}\label{sec:notation}
In general, if $A,B$ are sets then $A^B$ denotes the set of all functions $x:B \to A$. If $x\in A^B$ and $b\in B$ then the notations $x(b)$ and $x_b$ express the same element of $A$.

All maps and subsets are measurable unless explicitly stated otherwise. As a rule, all measure zero phenomena are ignored.

Given a topological space $X$, let $\Prob(X)$ denote the space of all Borel probability measures on $X$ endowed with the weak* topology. This is the smallest topology such that for every continuous compactly supported function $f$ on $X$ the map $\mu \mapsto \int f~d\mu$ is continuous (for $\mu \in \Prob(X)$). If $X$ is compact then the Banach-Alaoglu Theorem implies $\Prob(X)$ is compact. If $\G \cc X$ is a continuous action by $\G$ then let $\Prob_\G(X)$ denote the subspace of $\G$-invariant Borel probability measures. This is a closed subspace of $\Prob(X)$.

We write $f(n)=o_n(1)$ to mean $\lim_{n\to\infty} f(n) = 0$. Similarly, $f(r)=o_r(\log^2(r)/r)$ means $\lim_{r\to\infty} f(r) \left( \log^2(r)/r \right)^{-1} = 0$.

\section{A review of sofic entropy}\label{sec:entropy}

We will use the symbolic approach to sofic entropy with notational conventions similar to Tim Austin's from \cite{MR3543677, MR3542515}.

\subsection{Sofic approximations}


Suppose $\sigma:\Ga \to \sym(V)$ is a map where $V$ is a finite set and $\sym(V)$ is the group of permutations of $V$. It is not required that $\sigma$ is a homomorphism. Let $D \Subset \G$ be finite and $\d>0$. Then $\s$ is 
\begin{itemize}
\item {\bf $(D,\d)$-multiplicative} if 
$$\#\{v\in V:~ \sigma_i(gh)v = \sigma_i(g)\sigma_i(h)v ~ \forall g,h \in D\} > (1- \delta )|V|,$$
\item {\bf $(D,\d)$-trace preserving} if 
$$\#\{v\in V:~ \sigma_i(f)v\ne v~\forall f\in D\setminus\{1_\G\} \} >(1- \delta) |V|,$$
\item {\bf $(D,\d)$-sofic} if it is both $(D,\d)$-multiplicative and $(D,\d)$-trace preserving.
\end{itemize}
A {\bf sofic approximation} to $\Ga$ consists of a sequence $\Si = \{\sigma_i\}_{i\in \N}$ of maps $\sigma_i:\Ga \to \sym(V_i)$ such that for all finite $D\subset \G$, $\d>0$ and all but finitely many $i$, $\sigma_i$ is $(D,\d)$-sofic. A group is {\bf sofic} it admits a sofic approximation.

\subsection{Sofic entropy}\label{sec:sofic-entropy}

Throughout, $(\cX,d_\cX)$ and $(\cY,d_\cY)$ denote compact metric spaces. Given a finite set $V$, let $d_\cX^V$ be the normalized $\ell^1$-metric on $\cX^V$ defined by
$$d^V_\cX(x,y) := |V|^{-1}\sum_{v\in V} d_\cX(x_v,y_v).$$


For any finite set $V$, map $\sigma:\Ga \to \sym(V)$, $x \in \cX^{V}$ and $v\in V$ the {\bf pullback name of $x$ with respect to $(\sigma,v)$} is the element $\Pi^{\sigma}_v(x) \in \cX^\Ga$ defined by
$$\Pi^{\sigma}_v(x)(g) := x(\sigma(g)^{-1}v).$$
For example, if $\s$ is a homomorphism then $h \Pi^\s_v(x) = \Pi^\s_{\s(h)v}(x)$ so that the map $v \mapsto \Pi^\s_v(x)$ is $\G$-equivariant. 

The {\bf empirical measure of $x$} is the probability measure $P^{\sigma}_x$ on $\cX^\Ga$ defined by
$$P^{\sigma}_x := |V|^{-1} \sum_{v\in V} \delta_{\Pi^\sigma_v(x)}.$$
For example, if $\s$ is a homomorphism and $\s(\G)$ acts transitively on $V$ then $\{\Pi^\sigma_v(x):~v\in V\}$ is a single $\G$-orbit in which case $P^\s_x$ is the uniform measure on a finite $\G$-orbit. 

Given $\cO \subset \Prob(\cX^\Ga)$, an element $x\in \cX^V$ is a {\bf $(\cO,\s)$-microstate} if $P^\s_x \in \cO$. Let $\Omega(\cO,\sigma)$ be the set of all $(\cO,\s)$-microstates. The metric space $(\Omega(\cO,\sigma), d^V_\cX)$ is a {\bf model space} for the action $\G \cc (\cX^\G,\mu)$ for any $\mu \in \cO$. A major idea introduced in \cite{MR3542515, MR3543677} is to derive measure-conjugacy invariants from the asymptotic geometric features of these model spaces. 

Recall that a subset $Y$ of a metric space $(X,d_X)$ is {\bf $\epsilon$-covering} if $X$ is the open $\eps$-neighborhood of $Y$. Let $\cov_\eps(X,d_X)$ denote the minimum cardinality of an $\eps$-covering subset of $X$.

Let $\Si=\{\sigma_i\}_{i\in \N}$ be a sofic approximation to $\G$. The {\bf $\Si$-entropy of $\G \cc (\cX^\G,\mu)$}  is defined by
$$h_\Si(\mu):=\sup_{\eps>0} \inf_{\cO\ni \mu} \limsup_{i \to \infty} |V_i|^{-1}\log \cov_\eps(\Omega(\cO,\s_i), d_\cX^{V_i}).$$
See \cite{MR3542515} for a proof that this definition is equivalent to previous formulations of sofic entropy given in \cite{bowen-jams-2010} or \cite{kerr-li-variational} for example. For general discussions or when $\Si$ is left implicit, the $\Si$-entropy is called the {\bf sofic entropy}.

The basic facts about sofic entropy are: it is a measure-conjugacy invariant, it agrees with classical entropy when $\G$ is amenable, it can depend on the choice of sofic approximation, it can increase under factor maps, the sofic entropy of a Bernoulli shift is the Shannon entropy of the base. See \cite{MR3966832} for an introduction.


\begin{remark}
In the special case in which $\cX$ is finite, the definition above  reduces to
$$h_\Si(\mu):=\inf_{\cO\ni \mu} \limsup_{i \to \infty} |V_i|^{-1}\log \#\Omega(\cO,\s_i).$$
\end{remark}


\section{Sofic homology}\label{sec:homology}

\subsection{Homology theory on the Hamming cube}
Fix a finite set $V$ and compact metric space $(\cX,d_\cX)$. For an integer $d\ge 0$, let $C_d(\cX^V)$ be the abelian group generated by all symbols of the form $[x_0,\ldots, x_d]$ (with $x_0,\ldots, x_d \in \cX^V$) subject to the relations:
$$[x_{\pi(0)},\ldots, x_{\pi(d)}] = (-1)^{\textrm{sign}(\pi)}[x_0,x_1,\ldots, x_d]$$
over all $\pi \in \sym(d+1)$. An element of the form $[x_0,\ldots, x_d]$ is an {\bf oriented $d$-simplex} of $\cX^V$ and an element of $C_d(\cX^V)$ is called a {\bf $d$-chain}. 

Let $\partial_d: C_d(\cX^V) \to C_{d-1}(\cX^V)$ denote the boundary map
$$\partial_d( [ x_0,\ldots, x_d] ) = \sum_{i=0}^d (-1)^i [x_0,\ldots, \hx_i, \ldots, x_d]$$
where $\hx_i$ indicates that $x_i$ is omitted.

There is not much interesting that we can say about the homology of the Hamming cube $\cX^V$. Instead we will focus on special subgroups of $C_d(\cX^V)$ defined in terms of sofic approximation data as explained next.

\subsection{Special subgroups defined by a sofic approximation}

Let $\s:\G \to \sym(V)$ be a map. Given an open subset $\cO \subset \Prob(\cX^\G)$ and $\k>0$, let $C_d(\cO,\k,\s)$ be the subgroup of $C_d(\cX^V)$ generated by all chains of the form $[x_0,\ldots, x_d]$ such that each $x_i$ is a $(\cO,\s)$-microstate (that is $P_{x_i}^\s \in \cO$) and $d_\cX^V(x_i,x_j)<\k$ for all $i,j$. Let 
$$Z_d(\cO,\k,\s)= \ker(\partial_d) \cap C_d(\cO,\k,\s)$$
$$B_d(\cO,\k,\s) = \partial_{d+1}( C_{d+1}(\cO,\k,\s))$$
be the $(\cO,\k,\s)$-cycles and boundaries respectively. 

The {\bf length} of a $d$-chain $z \in C_d(\cX^V)$ is the smallest number of oriented simplices needed to represent $z$. So if
$$z = \sum_{i=1}^k c_i s_i$$
where $c_i \in \Z$ are coefficients and $s_i = [x^i_0,\ldots, x^i_d]$ is an oriented simplex then the length of $z$ is at most $k$. For $L>0$, let $Z_d^L(\cO,\k,\s)$ be the subgroup of $Z_d(\cO,\k,\s)$ generated by $(\cO,\k,\s)$-cycles of length $\le L$. To be precise, $z \in Z_d^L(\cO,\k,\s)$ if it is possible to write $z=\sum_{i=1}^k c_i z_i$ for some coefficients $c_i \in \Z$ and cycles $z_i \in Z_d(\cO,\k,\s)$ such that each $z_i$ has length $\le L$. 

Given nested open subsets $\cO_1\subset \cO_2 \subset \Prob(\cX^\G)$, constants $0<\k_1\le \k_2$ and $L>0$, define the homology group
$$H^L_d(\cO_1,\cO_2,\k_1,\k_2,\s):= \frac{ Z^L_d(\cO_1,\k_1,\s) }{ Z^L_d(\cO_1,\k_1,\s) \cap B_d(\cO_2,\k_2,\s)}.$$

\subsection{Main results}

Before stating the main theorem, we mention the following corollary which gives the flavor of the main result without as many quantifiers.

\begin{defn}
A group $H$ is a {\bf QS-group} of a group $G$ if $H$ is isomorphic to a quotient of a subgroup of $G$. Let $\Abel$ denote the class of abelian groups. A function $F:\Abel \to \R$ is {\bf monotone} if whenever $H$ is an QS-group of $G$, $F(H) \le F(G)$.
\end{defn}

The next corollary follows immediately from Theorem \ref{thm:mc} which is stated below.
\begin{cor}\label{cor:monotone}
Let $F=\{F_i\}_{i\in\N}$ be a sequence of monotone functions $F_i:\Abel \to \R$. Given an invariant measure $\mu \in \Prob_\G(\cX^\G)$, a sofic approximation $\Si$ and $L \in [1,\infty]$, define 
$$F_{d,\Si}(\mu):= \sup_{\cO_2 \ni \mu} \sup_{\kappa_2>0} \inf_{\mu\in\cO_1\subset \cO_2}\inf_{\kappa_1>0} \sup_{0<L<\infty} \limsup_{i\to\infty} F_i(H^L_d(\cO_1,\cO_2,\k_1,\k_2,\s_i)).$$
If $\cX$ is totally disconnected then $F_{d,\Si}$ is a measure-conjugacy invariant. In other words, if $(\cY,d_\cY)$ is another totally disconnected compact metric space, $\nu \in \Prob_\G(\cY^\G)$ and the actions $\G \cc (\cX^\G,\mu), \G \cc (\cY^\G,\nu)$ are measurably conjugate then $F_{d,\Si}(\mu)=F_{d,\Si}(\nu)$.
\end{cor}

The next result follows by setting $F_i(G):= |V_i|^{-1} \log \dim_\Q(G\otimes_\Z \Q)$ in the previous corollary.
\begin{cor}\label{cor:betti}
Given $\mu \in \Prob_\G(\cX^\G)$, define the {\bf $d$-th betti number of $\mu$ with respect to $\Si$} by
$$b_{d,\Si}(\mu):= \sup_{\cO_2 \ni \mu} \sup_{\kappa_2>0} \inf_{\mu\in\cO_1\subset \cO_2}\inf_{\kappa_1>0} \sup_{0<L<\infty} \limsup_{i\to\infty}  |V_i|^{-1} \log \dim_\Q(H^L_d(\cO_1,\cO_2,\k_1,\k_2,\s_i) \otimes_\Z \Q).$$
If $\cX$ is totally disconnected then $b_{d,\Si}(\mu)$ is a measure-conjugacy invariant.
\end{cor}

The main definition is:
\begin{defn}
Let $\mu \in \Prob_\G(\cX^\G), \nu \in \Prob_\G(\cY^\G)$, $L,d \ge 0$. Then the {\bf $d$-dimensional sofic homology of $\nu$ is less than or equal to the $d$-dimensional sofic homology of $\mu$}  if for every open neighborhood $\cO_{2,\nu} \ni \nu$, every $\k_{2,\nu}>0$ there exist an open neighborhood $\cO_{2,\mu} \ni \mu$ and $\k_{2,\mu}>0$ such that for every open $\cO_{1,\mu}$ with $\mu \in \cO_{1,\mu} \subset \cO_{2,\mu}$ and every $\k_{1,\mu}$ with $0<\k_{1,\mu}\le \k_{2,\mu}$ there exist an open neighborhood $\cO_{1,\nu}$ with $\nu \in \cO_{1,\nu} \subset \cO_{2,\nu}$ and $\k_{1,\nu}$ with $0<\k_{1,\nu} \le \k_{2,\nu}$ such that for every $0<L<\infty$ and all but finitely many $n$, $H^L_d(\cO_{1,\nu},\cO_{2,\nu},\k_{1,\nu},\k_{2,\nu},\s_n)$ is a QS-group of $H^L_d(\cO_{1,\mu},\cO_{2,\mu},\k_{1,\mu},\k_{2,\mu},\s_n)$. The {\bf $d$-dimensional sofic homology theories of $\mu$ and $\nu$ are equivalent} if the $d$-dimensional sofic homology of $\mu$ is less than or equal to the $d$-dimensional sofic homology of $\nu$ and vice versa. 
\end{defn}

The main theorem is:
\begin{thm}\label{thm:mc}
The homology groups defined above yield a measure-conjugacy invariant as follows. Suppose $\cX,\cY$ are totally disconnected compact metric spaces, $\mu \in \Prob_\G(\cX^\G), \nu \in \Prob_\G(\cY^\G)$ and $\G \cc (\cX^\G,\nu)$ is measurably conjugate to $\G \cc (\cY^\G,\nu)$. Then $\mu$ and $\nu$ have equivalent $d$-dimensional sofic homology theories with respect to every approximation $\Si$ and for every dimension $d$.
\end{thm}

\begin{remark}
All of the definitions could be changed by setting $L=\infty$ throughout. The analog of Theorem \ref{thm:mc} still holds under this change with essentially the same proof. However, we do not know how to compute this homology except in degenerate cases.
\end{remark}



\section{Preliminaries to the proof of Theorem \ref{thm:mc}}\label{sec:prelim}

\subsection{Almost Lipschitz maps}

\begin{defn}

Let $(X,d_X)$ and $(Y,d_Y)$ be metric spaces, let $\eps > 0$, and let $L < \infty$.  A map $\phi:X \to Y$ is \textbf{$\eps$-almost $L$-Lipschitz} if
\[d_Y(\phi(x),\phi(x')) \leq \eps + Ld_X(x,x') \quad \forall x,x' \in X.\]
A map is \textbf{$\eps$-almost Lipschitz} if it is so for some $L$.
\end{defn}

\begin{lem}\label{lem:uc-al}
A uniformly continuous map from a bounded metric space to another bounded metric space is $\eta$-almost Lipschitz for every $\eta>0$.
\end{lem}

\begin{proof}
Let $\phi: \cX \to \cY$ be a uniformly continuous map from a bounded space $(\cX,d_\cX)$ to a bounded metric space $(\cY,d_\cY)$ and let $\eta>0$. Let $\eps>0$ be small enough so that if $d_\cX(x,y)<\eps$ then $d_\cY(\phi x, \phi y)< \eta$. 

Now let $x,y \in X$ be arbitrary. If $d_\cX(x,y)\ge \eps$ then 
$$d_\cY(\phi x, \phi y) \le \diam(\cY,d_\cY) \le \eta + \frac{\diam(\cY,d_{\cY})}{\eps} d_\cX(x,y).$$
So $\phi$ is $\eta$-almost $ \frac{\diam(\cY,d_{\cY})}{\eps}$-Lipschitz.
\end{proof}


\subsection{Equivariant maps and their approximations}


\begin{notation}
If $x \in \cX^\Ga$ and $g\in \G$, then $S^gx = gx \in \cX^\G$ is defined by $S^gx(f)=x(g^{-1}f)$. We will also write $S^gx$ if $x \in \cY^\G$. So $S^g$ is the {\bf shift} by $g$. 
\end{notation}

\begin{defn}\label{D:equivariant}
A map $\Phi:\cX^\Ga \to \cY^\Ga$ is {\bf equivariant} if $\Phi(gx) = g \Phi(x)$ for a.e. $x\in \cX^\Ga$ and every $g\in \Ga$. Given a map $\psi:\cX^\Ga \to \cY$ we define an equivariant map $\psi^\Ga:\cX^\Ga \to \cY^\Ga$ by $\psi^\Ga (x)(h) = \psi(S^{h^{-1}}x)$. For example, if $\Phi:\cX^\Ga \to \cY^\Ga$ is equivariant and $\phi:\cX^\Ga \to \cY$ is defined by $\phi(x)=\Phi(x)(1_\Ga)$ then $\Phi = \phi^\Ga$. 

\end{defn}

\begin{defn}\label{dfn:local-fn}
For a subset $D \subset \Ga$, let $\Res^D: \cX^\Ga \to \cX^D$ denote the restriction map.  If $\phi:\cX^\Ga\to \cY$ and $D \subseteq \Ga$ is finite, then $\phi$ is \textbf{$D$-local} if it is measurable with respect to $\Res^D$.  A function is \textbf{local} if it is \textbf{$D$-local} for some $D$.

\end{defn}

\begin{defn}\label{defn:UC}
As above, we let $(\cX,d_\cX)$ and $(\cY,d_\cY)$ be bounded Polish spaces. Also let $\mu \in \Prob_\Ga(\cX^\Ga)$, $\phi: \cX^\Ga\to \cY$ be a measurable function, and $\eta > 0$.  An {\bf $\eta$-uniformly continuous} (or {\bf $\eta$-UC}) approximation to $\phi$ rel $(\mu,d_\cX,d_\cY)$ is a measurable map $\tphi:\cX^\Ga \to \cY$ with the following properties.
\begin{itemize}
\item[i)] The map $\tphi$ approximates $\phi$ in the sense that
\begin{eqnarray}\label{eq:int-approx}
\int d_\cY(\tphi(x),\phi(x))\, d\mu(x) < \eta.
\end{eqnarray}
\item[ii)] There is a finite $D \subseteq \Ga$ such that $\tphi$ is $D$-local. 
\item[iii)] Regarded as a map from $\cX^D$ to $\cY$, $\tphi$ is uniformly continuous with respect to $d_\cX^D$ and $d_\cY$ (where $d_\cX^D$ is the normalized $\ell^1$-metric on $\cX^D$ as defined in the beginning of \S \ref{sec:sofic-entropy}). 
\end{itemize}
\end{defn}

\begin{lem}\label{lem:approx-by-Lip}
Suppose that $\cX,\cY,\mu,\phi$ are as in Definition \ref{defn:UC}, $\cX$ is totally disconnected and both $\cX,\cY$ are compact. Then there exist $\eta$-UC approximations to $\phi$ for all $\eta > 0$.
\end{lem}

\begin{proof}
After rescaling if necessary, we may assume that the diameter of $\cY$ is bounded by 1. Because $\cY$ is compact, there exists a finite open cover $\cO=\{O_1,\ldots, O_n\}$ of $\cY$ by sets of diameter $<\eta/3$.


A subset $X \subset \cX^\G$ is {\bf $D$-local} if its characteristic function $1_X:\cX^\G \to \R$ is $D$-local. Because $\cX$ is totally disconnected, for every $1\le i \le n$, there exist a finite subset $D_i \subset \G$ and a  $D_i$-local clopen subset $\tC_i \subset \cX^\G$ such that
$$\mu(\tC_i \vartriangle \phi^{-1}(O_i)) < \frac{\eta}{3n}.$$
For $1\le i \le n$, let 
$$C_i := \tC_i \setminus \cup_{j=i+1}^n \tC_j.$$
Also let $C_0  = \cX^\G \setminus \cup_{i=1}^n C_i$. Then $\{C_i\}_{i=0}^n$ is a clopen partition of $\cX^\G$. Setting $D=\cup_i D_i$, we see that $C_i$ is $D$-local for every $0\le i \le n$.

Choose a point $p_i \in O_i$ for all $1\le i \le n$ and also let $p_0 \in \cY$ be an arbitrary point. Define $\tphi:\cX^\G \to \cY$ by $\tphi(x)=p_i$ if $x \in C_i$. By construction, $\tphi$ is $D$-local. It is uniformly continuous because it is continuous and $\cX^\G$ is compact. To finish the proof, it suffices to estimate the error in the approximation to $\phi$:
\begin{eqnarray*}
\int d_\cY(\phi(x),\tphi(x))~d\mu(x) &=& \sum_{i=0}^n \int_{C_i} d_\cY(\phi(x),\tphi(x))~d\mu(x).
\end{eqnarray*}
Since $C_0 = \cX^\G \setminus \cup_{i=1}^n \tC_i$,
$$\mu(C_0) \le  \sum_{i=1}^n \mu(\tC_i \vartriangle \phi^{-1}(O_i)) \le \eta/3.$$
Since the diameter of $(\cY,d_{\cY})$ is bounded by 1, $ \int_{C_0} d_\cY(\phi(x),\tphi(x))~d\mu(x) \le \eta/3$.

For any $1\le i \le n$, $C_i \subset \tC_i$. Therefore, $C_i \subset \phi^{-1}(O_i) \cup (\tC_i \vartriangle \phi^{-1}(O_i) )$. If $x \in C_i \cap \phi^{-1}(O_i) $ then $d_\cY(\phi x, \tphi x)<\eta/3$ since $O_i$ has diameter $<\eta/3$. So
\begin{eqnarray*}
\int_{C_i} d_\cY(\phi(x),\tphi(x))~d\mu(x) &=& \int_{C_i \cap  \phi^{-1}(O_i)} d_\cY(\phi(x),\tphi(x))~d\mu(x) + \int_{C_i \setminus \phi^{-1}(O_i)} d_\cY(\phi(x),\tphi(x))~d\mu(x) \\
&\le& \mu(C_i)\eta/23 + \mu(\tC_i \vartriangle \phi^{-1}(O_i)) \le \mu(C_i)\eta/3  + \frac{\eta}{3n}.
\end{eqnarray*}
Since $\sum_{i=1}^n \mu(C_i)\eta/3  + \frac{\eta}{3n} \le 2\eta/3$,
\begin{eqnarray*}
\int d_\cY(\phi(x),\tphi(x))~d\mu(x) &\le & \eta.
\end{eqnarray*}
Since $\eta$ is arbitrary, this implies the lemma.

\end{proof}

\begin{defn}
Let $F \subset \G$ be finite and $\phi:\cX^\G \to \cY$. Then $\phi^F: \cX^\G \to \cY^F$ is defined by $\phi^F = \Res^F \circ \phi^\G$. So for any $f\in F$,
$$\phi^F(x)(f) = \phi^\G(x)(f) = \phi(S^{f^{-1}}x).$$ 
\end{defn}

\begin{lem}\label{lem:good-approx-Ham-sum}
Suppose that $\cX,\cY,\mu,\phi$ are as in Definition \ref{defn:UC} and $(\cY,d_\cY)$ has diameter at most $1$.  If $\tphi$ is an $\eta$-UC approximation to $\phi$ rel $(\mu,d_\cX,d_\cY)$ for some $\eta \in (0,1)$, then $\tphi^F$ is an $\eta$-UC approximation to $\phi^F:\cX^\G\to \cY^F$ rel $(\mu,d_\cX,d_\cY^{F})$ for every finite $F\subseteq \G$.
\end{lem}

\begin{proof}
This lemma is similar to \cite[Lemma 4.4]{MR3542515} but it is easier since we work with UC maps.  

Firstly, the shift-invariance of $\mu$ and inequality~(\ref{eq:int-approx}) imply that
\begin{equation}\label{eq:int-approx2}
\int d_\cY^{F}\big(\phi^F(x),\tphi^F(x)\big)\,d\mu(x) = \frac{1}{|F|} \sum_{g \in F^{-1}}\int d_\cY(\phi(gx),\tphi(gx))\,d\mu(x) < \eta.
\end{equation}

Let $\tphi$ be $D$-local for some finite $D \subset \G$. Then $\tphi^F$ is $FD$-local since for any $f \in F$, $\tphi^F(x)_f = \tphi(S^{f^{-1}}x)$ depends only on the restriction of $S^{f^{-1}}x$ to $D$. However, for $d\in D$, $S^{f^{-1}}x(d)=x(fd)$. So $\tphi^F(x)$ depends only on the restriction of $x$ to $FD$.

 Lastly, we claim $\tphi^F$ is uniformly continuous as a map from $\cX^{FD}$ to $\cY^F$. To see this, let $\eps>0$. Since $\tphi$ is uniformly continuous as a map from $\cX^D$ to $\cY$, there is a $\d>0$ such that if $x,y \in \cX^D$ satisfy $d_\cX^D(x,y)<\sqrt{\d} |D|$ then $d_\cY(\tphi(x),\tphi(y))<\eps/2$. By choosing $\delta$ smaller if necessary we may assume $\sqrt{\d}<\eps/2$.

For every $g\in FD$ the number of pairs $(f,d) \in F\times D$ such that $fd=g$ is at most $|D|$. Therefore,
\begin{eqnarray*}
d_{\cX}^{FD}(x,y) &=& |FD|^{-1} \sum_{g\in FD} d_\cX(x_g,y_g)\\
 &\ge & |FD|^{-1} |D|^{-1} \sum_{f\in F} \sum_{d\in D} d_\cX(x_{fd}, y_{fd}) \\
 &=&  |FD|^{-1} \sum_{f\in F} d^D_\cX( S^{f^{-1}}x, S^{f^{-1}}y).
 \end{eqnarray*}
  Suppose $x,y \in \cX^{FD}$ satisfy $d_\cX^{FD}(x,y)< \d$. By the previous inequality,
  $$|F|^{-1}  \sum_{f\in F} d^D_\cX( S^{f^{-1}}x, S^{f^{-1}}y)\le \frac{|FD|}{|F|}d_{\cX}^{FD}(x,y) < \frac{|FD|}{|F|}\d \le |D| \d.$$
  
 By Markov's inequality, there exists a subset $F' \subset F$ such that $|F'| \ge (1-\sqrt{\d})|F|$ and $d_\cX^D(S^{g^{-1}} x, S^{g^{-1}} y) < \sqrt{\d} |D|$ for all $g \in F'$. By choice of $\d$, if $f\in F'$ then $d_\cY( \tphi(S^{f^{-1}} x), \tphi(S^{f^{-1}} y))\le \eps/2$. Because the diameter of $(\cY,\d_\cY)$ is at most 1,
 $$d_\cY^F( \tphi^F x, \tphi^F y) = |F|^{-1} \sum_{g\in F^{-1}} d_\cY( \tphi(S^{g} x), \tphi(S^{g} y)) \le \frac{(\eps/2) |F'| + |F \setminus F'| }{|F|} \le \eps/2 + \sqrt{\d} \le \eps.$$ 
This shows $\tphi^F$ is uniformly continuous as a map from $(\cX^{FD}, d_\cX^{FD})$ to $(\cY^F, d_\cY^F)$. 

\end{proof}

\begin{lem}\label{lem:composition}
Suppose that $\cX,\cY,\mu,\phi$ are as in Definition \ref{defn:UC} and $(\cY,d_\cY)$ has diameter at most $1$. Let $\nu  = \phi^\G_*\mu \in \Prob_\G(\cY^\G)$. Suppose $(\cZ,d_\cZ)$ is also a bounded Polish space with diameter 1 and $\psi:\cY^\Ga \to \cZ$ is measurable.  Let $\tpsi$ be an $\eta_\psi$-UC approximation to $\psi$ and $\tphi$ an $\eta_\phi$-UC approximation to $\phi$. Then $\tpsi \circ \tphi^\Ga$ is an $\eta$-UC-approximation to $\psi \circ \phi^\Ga$ where $\eta=\eta(\tphi,\tpsi)$ tends to $2\sqrt{\eta_\psi}+\eta_\psi$ as $\eta_\phi$ tends to zero with $\tpsi$ fixed.
\end{lem}

\begin{proof}
By definition there exists a finite subset $D_\psi \subset \Ga$ such that $\tpsi$ is $D_\psi$-local and $\tpsi$ regarded as a map from $\cY^{D_\psi}$ to $\cZ$ is uniformly continuous.  Moreover,  
$$\int d_\cZ(\psi y, \tpsi y)~d\nu(y) < \eta_\psi.$$
As mentioned in Lemma \ref{lem:uc-al}, because $\tpsi$ is uniformly continuous, it is $\eta_\psi$-almost $L_\psi$-Lipschitz for some constant $L_\psi$.

Suppose $\tphi$ is an $\eta_\phi$-UC approximation to $\phi$. By Lemma \ref{lem:good-approx-Ham-sum}, $\tphi^{D_\psi}$ is a $\eta_\phi$-UC approximation to $\phi^{D_\psi}$. So there exists a finite subset $D_\phi \subset \Ga$ such that $\tphi^{D_\psi}$ is $D_\phi$-local, $\tphi^{D_\psi}$ regarded as a map from $\cX^{D_\phi}$ to $\cY^{D_\psi}$ is uniformly continuous and 
$$\int d^{D_\psi}_\cY(\phi^{D_\psi} x, \tphi^{D_\psi} x)~d\mu(x) < \eta_\phi.$$

It is immediate that $\tpsi \circ \tphi^\Ga$ is $D_\phi D_\psi$-local and when regarded as a map from $\cX^{D_\phi D_\psi}$ to $\cZ$, it is uniformly continuous. Let
$$G_\psi = \{y \in \cY^\Ga:~ d_\cZ(\tpsi y, \psi y ) < \sqrt{\eta_\psi}\}$$
$$G_\phi = \{x\in \cX^\Ga:~ d_\cY^{D_\psi}(\tphi^{D_\psi} x, \phi^{D_\psi} x) < \sqrt{\eta_\phi} \}.$$
Suppose $x \in G_\phi \cap (\phi^\Ga)^{-1}(G_\psi)$. 

\begin{eqnarray*}
d_\cZ( \tpsi \tphi^\Ga x, \psi \phi^\Ga x) &\le& d_\cZ( \tpsi \tphi^\Ga x, \tpsi \phi^\Ga x) + d_\cZ( \tpsi \phi^\Ga x, \psi \phi^\Ga x)\\
&\le& (\eta_\psi + L_\psi \sqrt{\eta_\phi} ) + \sqrt{\eta_\psi}.
\end{eqnarray*}
The first term above occurs because $\tpsi$ is $\eta_\psi$-almost $L_\psi$-Lipschitz as a map from $\cY^{D_\psi}$ to $\cZ$ and $d^{D_\psi}_\cY(\tphi^{D_\psi} x, \phi^{D_\psi} x)<\sqrt{\eta_\phi}$. The second term occurs because $\phi^\Ga(x) \in G_\psi$. 

It follows that
\begin{eqnarray*}
\int d_\cZ( \tpsi \tphi^\Ga x, \psi \phi^\Ga x) ~d\mu(x) &\le& (1-\mu(G_\phi \cap (\phi^\Ga)^{-1}(G_\psi)))\diam(\cZ) + \eta_\psi + L_\psi \sqrt{\eta_\phi} + \sqrt{\eta_\psi}.
\end{eqnarray*}
By Markov's inequality, $\mu(G_\phi) > 1 - \sqrt{\eta_\phi}$ and $\nu(G_\psi) > 1-  \sqrt{\eta_\psi}$. Because $\phi^\Ga_*\mu = \nu$ it follows that
$$1-\mu(G_\phi \cap (\phi^\Ga)^{-1}(G_\psi)) <  \sqrt{\eta_\phi} + \sqrt{\eta_\psi}.$$
Since $\diam(\cZ)=1$,
\begin{eqnarray*}
\int d_\cZ( \tpsi \tphi^\Ga x, \psi \phi^\Ga x) ~d\mu(x) &\le&  (L_\psi +1)\sqrt{\eta_\phi} + 2\sqrt{\eta_\psi} + \eta_\psi.
\end{eqnarray*}

\end{proof}

\subsection{Sofic models}


Recall that the pullback name of $x \in \cX^V$ with respect to $\sigma:\Ga \to \sym(V)$ and $v\in V$ is 
$$\Pi^{\sigma}_v(x)(g) := x(\sigma(g)^{-1}v).$$
Given a map $\sigma:\Ga \to \sym(V)$, where $V$ is a finite set and a map $\phi:\cX^\Ga \to \cY$, define $\phi^\sigma:\cX^V \to \cY^V$ by
$$\phi^\sigma(x)_v = \phi( \Pi^\sigma_v(x)).$$


\begin{lem}\label{lem:sofic-UC}
Suppose $\phi:\cX^\Ga \to \cY$ is $D$-local for some finite set $D \subset \Ga$ and regarded as a map from $\cX^D \to \cY$ is $\eta$-almost $L$-Lipschitz. Then $\phi^\sigma$ is $\eta$-almost $L$-Lipschitz regarded as map from $(\cX^V,d_\cX^V)$ to $(\cY^V,d_\cY^V)$.
\end{lem}

\begin{proof}
Let $x,y \in \cX^V$. Then
\begin{eqnarray*}
d_{\cY}^V( \phi^\s x, \phi^\s y) &=& |V|^{-1} \sum_{v\in V} d_{\cY}( (\phi^\s x)_v, (\phi^\s y)_v ) \\
&\le & |V|^{-1} \sum_{v\in V} \eta + L d^D_{\cX}( \Res^D \Pi^\s_v(x), \Res^D \Pi^\s_v(y) ) \\
&= & \eta +L |V|^{-1} \sum_{v\in V} |D|^{-1} \sum_{g\in D} d_{\cX}( \Pi^\s_v(x)_g, \Pi^\s_v(y)_g ) \\
&= & \eta +L |V|^{-1} \sum_{v\in V} |D|^{-1} \sum_{g\in D} d_{\cX}( x(\s(g^{-1})v), y(\s(g^{-1})v)) \\
&= & \eta +L d_\cX^V(x,y)
\end{eqnarray*}
where the last equality holds because for each $v\in V$ the number of pairs $(g,w) \in D \times V$ such that $\s(g^{-1})w=v$ equals $|D|$. Because $x,y$ are arbitrary, this implies the lemma.
\end{proof}

\begin{lem}\label{lem:sofic-2}
Suppose $\phi:\cX^\Ga \to \cY$ is $D_\phi$-local for some finite set $D_\phi \subset \Ga$ and $\psi:\cY^\Ga \to \cZ$ is $D_\psi$-local for some finite set $D_\psi$. Then for all $x\in \cX^V$,
$$ \{v \in V:~ (\psi \phi^\Ga)^\sigma(x)_v \ne \psi^\sigma \phi^\sigma(x)_v\} \subset \{v\in V:~ \exists h\in D_\phi, g\in D_\psi, ~\sigma(h^{-1})\sigma(g^{-1})v \ne \sigma(h^{-1}g^{-1})v\}.$$
In particular, if $1_\Ga \in D_\phi \cap D_\psi$ and $\sigma$ is a $(D^{-1}_\phi D^{-1}_\psi,\delta)$-sofic approximation to $\Ga$ then 
$$ \#\{v \in V:~ (\psi \phi^\Ga)^\sigma(x)_v \ne \psi^\sigma \phi^\sigma(x)_v\} \le \delta |V|.$$
\end{lem}

\begin{proof}
Fix $v\in V$. Suppose $\sigma(h^{-1})\sigma(g^{-1})v = \sigma(h^{-1}g^{-1})v$ for all $h\in D_\phi$ and $g \in D_\psi$. It suffices to show $(\psi\phi^\G)^\sigma(x)_v = \psi^\sigma \phi^\sigma(x)_v$. Observe that
$$(\psi\phi^\G)^\sigma(x)_v = \psi (\phi^\G( \Pi^\s_v x))$$
$$\psi^\sigma \phi^\sigma(x)_v = \psi( \Pi^\s_v(\phi^\s(x))).$$
Since $\psi$ is $D_\psi$-local, it suffices to show that for every $g\in D_\psi$,
$$\phi^\G( \Pi^\s_v x)_g = \Pi^\s_v(\phi^\s(x))_g.$$
Observe that
$$\phi^\G( \Pi^\s_v x)_g = \phi(S^{g^{-1}} \Pi^\s_v (x))$$
$$ \Pi^\s_v(\phi^\s(x))_g  = \phi^\s(x)(\s(g^{-1})v) = \phi( \Pi^\s_{\s(g^{-1})v}(x)).$$
Since $\phi$ is $D_\phi$-local it suffices to show that for every $h \in D_\phi$,
$$S^{g^{-1}} \Pi^\s_v (x)_h = \Pi^\s_{\s(g^{-1})v}(x)_h.$$
The left-hand side simplifies as follows:
$$S^{g^{-1}} \Pi^\s_v (x)_h = \Pi^\s_v (x)_{gh} = x(\s( (gh)^{-1})v) = x(\s(h^{-1}g^{-1})v).$$
The right-hand side simplifies to
$$ \Pi^\s_{\s(g^{-1})v}(x)_h = x(\s(h^{-1})\s(g^{-1})v).$$
Therefore if $\s(h^{-1}g^{-1})v = \s(h^{-1})\s(g^{-1})v$ for every $h \in D_\phi$ and $g\in D_\psi$ then $(\psi\phi^\G)^\sigma(x)_v = \psi^\sigma \phi^\sigma(x)_v$.




\end{proof}

\begin{defn}\label{D:TV}
The {\bf total variation distance} between two measures $\mu$ and $\nu$ on the same $\s$-algebra $\cF$ is 
$$d_{\textrm{TV}}(\mu,\nu) = \sup_{A \in \cF} |\mu(A) - \nu(A)|.$$
\end{defn}


Roughly speaking, the next lemma shows that closeness in total variation distance of restricted measures implies closeness in the weak* topology. 

\begin{lem}\label{L:TV}
For any $\mu \in \Prob(\cX^\G)$ and any weak* open set $O \subset \Prob(\cX^\G)$ with $\mu \in O$, there exists a finite set $E \subset \G$ and $\d>0$ such that
$$\left\{ \nu \in \Prob(\cX^\G):~ d_{\textrm{TV}}(\Res^E_*\mu, \Res^E_*\nu)<\d\right\} \subset O.$$
\end{lem}

\begin{proof}
By definition of the weak* topology, there are continuous functions $f_1,\ldots, f_k$ on $\cX^\G$ and $\eps>0$ such that
 $$\left\{ \nu \in \Prob(\cX^\G):~ \left| \int f_i~d\nu - \int f_i ~d\mu\right|<\eps ~\forall 1\le i \le k \right\} \subset O.$$
Because each $\cX^\G$ is the inverse limit of the compact spaces $\cX^D$ over finite $D \subset \G$, there exist a finite $E \subset \G$ and $E$-local continuous functions $f'_1,\ldots, f'_k$ on $\cX^\G$ such that $|f_i-f'_i|< \eps/3$ for all $i$. By the triangle inequality,
$$\left\{ \nu \in \Prob(\cX^\G):~ \left| \int f'_i~d\nu - \int f'_i ~d\mu\right|\le \eps/3 ~\forall 1\le i \le k \right\} \subset O.$$
Now suppose $\nu \in \Prob(\cX^\G)$ and $d_{\textrm{TV}}(\Res^E_*\mu, \Res^E_*\nu)<\frac{\eps}{6M}$ where $M=\max_{1\le i\le k} \|f'_i\|_{\textrm{sup}}$. By abuse of notation, we may consider each $f'_i$ to be a continuous function on $\cX^E$. So
\begin{eqnarray*}
\left| \int f'_i~d\nu - \int f'_i ~d\mu\right| &=& \left| \int f'_i~d\Res^E_*\nu - \int f'_i ~d\Res^E_*\mu\right| \\
&\le& 2\|f'_i\|_{\textrm{sup}} d_{\textrm{\textrm{TV}}}(\Res^E_*\mu, \Res^E_*\nu) \le \eps/3.
\end{eqnarray*}
By the previous inclusion, this implies $\nu \in O$ and completes the lemma (with $\delta = \frac{\eps}{6M}$). 
\end{proof}

\begin{lem}\label{lem:sofic-3}
Suppose $\phi:\cX^\Ga \to \cY$ is  $D$-local for some finite set $D \subset \Ga$. Also let $E \subset \G$ be finite and $\d>0$. Then there exist a finite set $F \subset \G$ and $\eps>0$ (depending only on $D,E,\d$) such that if $\sigma$ is an $(F,\eps)$-sofic approximation to $\Ga$, then for all $x\in \cX^V$,
$$\# \{v \in V:~ (\phi^\Ga)(\Pi^\sigma_v(x))(g) \ne \Pi^\sigma_v( \phi^\sigma(x))(g)~ \forall g \in E \}  \le \delta |V|.$$
In particular, the total variation distance between the restricted empirical measures $\Res^E_* \phi^\G_*P_x^\s$ and $\Res^E_* P^\s_{\phi^\s(x)}$ is bounded by $\d$. 
\end{lem}

\begin{proof}
Let $F \subset \G$ and $\eps>0$ be such that if $\s$ is $(F,\eps)$-sofic then 
$$\# \{v \in V:~\s(gh)^{-1}v = \s(h)^{-1}\s(g)^{-1}v ~\forall h \in D, \forall g\in E\} \ge (1-\d)\#V.$$
Suppose $\s$ is  $(F,\eps)$-sofic and fix $v\in V$. By Definition \ref{D:equivariant},
$$(\phi^\Ga)(\Pi^\sigma_v(x))(g) = \phi(S^{g^{-1}}\Pi^\s_v(x)).$$
By the definitions of pullback and $\phi^\s$, 
$$\Pi^\sigma_v( \phi^\sigma(x))(g) = \phi^\s(x)(\s(g)^{-1}v) = \phi(\Pi^\s_{\s(g)^{-1}v}(x)).$$
Because $\phi$ is $D$-local, if 
$$(S^{g^{-1}}\Pi^\s_v(x))(h) = \Pi^\s_{\s(g)^{-1}v}(x)(h)$$
for all $h \in D$ then $(\phi^\Ga)(\Pi^\sigma_v(x))(g) = \Pi^\sigma_v( \phi^\sigma(x))(g)$. 

We compute
$$(S^{g^{-1}}\Pi^\s_v(x))(h) = \Pi^\s_v(x)(gh)=x(\s(gh)^{-1}v)$$
and
$$\Pi^\s_{\s(g)^{-1}v}(x)(h) = x(\s(h)^{-1}\s(g)^{-1}v).$$
Because $\s$ is $(F,\eps)$-sofic, there is a $(1-\d)$-fraction of vertices $v$ such that $(S^{g^{-1}}\Pi^\s_v(x))(h) = \Pi^\s_{\s(g)^{-1}v}(x)(h)$ for all $g \in E, h \in D$. Again, since $\phi$ is $D$-local, this condition implies $(\phi^\Ga)(\Pi^\sigma_v(x))(g) = \Pi^\sigma_v( \phi^\sigma(x))(g)$ for all $g \in E$. This proves the first claim. 

The last claim is implied by the first. In fact, $\Res^E_* \phi^\G_*P_x^\s$ is the law of $\Res^E(\phi^\Ga(\Pi^\sigma_v(x)))$ where $v \in V$ is chosen uniformly at random, while $\Res^E_* P^\s_{\phi^\s(x)}$ is the law of $\Res^E(\Pi^\sigma_v( \phi^\sigma(x)))$ where $v \in V$ is chosen uniformly at random. Since $\Res^E(\phi^\Ga(\Pi^\sigma_v(x)))= \Res^E(\Pi^\sigma_v( \phi^\sigma(x)))$ for all but a $\d$-fraction of vertices $v$, the total variation distance between the restricted empirical measures $\Res^E_* \phi^\G_*P_x^\s$ and $\Res^E_* P^\s_{\phi^\s(x)}$ is bounded by $\d$. 

\end{proof}

\section{Proof of Theorem \ref{thm:mc}}\label{sec:proof1}

We need one more lemma before the proof of the main theorem. Let $(\cY,d_\cY)$ be a compact metric space. Given a map $\lambda:\cY^V \to \cY^V$ let $\lambda_*:C_d(\cY^V) \to C_d(\cY^V)$ be the corresponding homomorphism of chain groups. Note $\lambda_*$ commutes with all boundary maps $\partial_d:C_d(\cY^V) \to C_{d-1}(\cY^V)$. 

\begin{lem}\label{lem:cylinder}
Let $\lambda:\cY^V \to \cY^V$ be given and suppose there is a constant $\k'>0$ such that $d^V_\cY(y,\lambda(y))<\k'$ for all $y\in \cY^V$. If $z \in Z_d(\cO_1,\k,\s)$ and $\lambda(\Omega(\cO_1,\s)) \subset \Omega(\cO_2,\s)$ (for some $\cO_1, \cO_2, \k,\s$)  then 
$$z - \lambda_*z \in B_d(\cO_1\cup\cO_2,\k+2\k',\s).$$
\end{lem}


\begin{proof}







It will be convenient to work with ordered simplices rather than ordinary simplices. An {\bf ordered $k$-simplex} is an ordered $(k+1)$-tuple $(x_0,\ldots, x_k)$ with $x_i \in \cY^V$.  Given an ordered $k$-simplex $(x_0,\ldots, x_k)$ define $P(x_0,\ldots, x_k) \in  C_{k+1}(\cY^V)$ by
$$P(x_0,\ldots, x_k) := \sum_{i=0}^k (-1)^i [x_0, \ldots, x_i, \l(x_i), \ldots, \l(x_k)].$$
Given an oriented simplex $[x_0,\ldots, x_k] \in C_k(\cY^V)$ define $P([x_0,\ldots, x_k]) \in C_{k+1}(\cY^V)$ by
$$P([x_0,\ldots, x_k]) := |\sym(k+1)|^{-1} \sum_{\pi \in \sym(k+1)} \textrm{sign}(\pi)P(x_{\pi(0)},\ldots, x_{\pi(k)} ).$$

We extend $P$ linearly so that it is well-defined as a homomorphism from $C_k(\cY^V)$ to $C_{k+1}(\cY^V)$ (for every $k$ too). 

\noindent {\bf Claim}. $\partial P = \l_* - I + P\partial$ where $I$ denotes the identity map.

\begin{proof}
It suffices to show that for any $x_0,\ldots, x_k$,
$$\partial P([x_0,\ldots, x_k]) = [\l(x_0), \l(x_1),\ldots, \l(x_k)] - [x_0,\ldots, x_k] - P (\partial [x_0,\ldots, x_k]).$$
The proof is by direct inspection of the coefficients. Details are in the proof of \cite[Theorem 2.10]{MR1867354}.
\end{proof}

If $z \in Z_d(\cO_1,\k,\s)$ then $Pz \in C_d(\cO_1\cup\cO_2,\k + 2\k',\s)$ since $\lambda(\Omega(\cO_1,\s)) \subset \Omega(\cO_2,\s)$ and $d^V_\cY(y,\lambda(y))<\k'$ for all $y$. The claim implies
$$\partial P(z) = \l_*z - z - P\partial z = \l_*z - z.$$ 
Therefore $\l_*z -z \in B_d(\cO_1\cup\cO_2,\k+2\k',\s).$





\end{proof}

\begin{proof}[Proof of Theorem \ref{thm:mc}]
Suppose $\Phi:(\cX^\Ga,\mu) \to (\cY^\Ga,\nu)$ is a measure-conjugacy. We may assume without loss of generality that $(\cX,d_\cX)$ and $(\cY,d_\cY)$ have diameter 1. Let $\cO_{2,\nu}$ and $\kappa_{2,\nu}$ be given. 


We need to choose $\cO_{2,\mu}$ and $\kappa_{2,\mu}$. Before doing this, define $\phi:\cX^\Ga \to \cY$ by $\phi(x):=\Phi(x)_e$ (where $e \in \G$ is the identity element). Note $\phi^\Ga = \Phi$. We choose a UC-approximation $\tphi$ to $\phi$ as follows. Choose $0<\eta_\phi<1$ small enough  so that $3 \sqrt{\eta_\phi}< \frac{\kappa_{2,\nu}}{8}$.  By Lemma \ref{lem:approx-by-Lip} there exists an $\eta_\phi$-UC-approximation $\tphi$ to $\phi$. Because $\tphi$ is continuous, there exists an open neighborhood $\cO_{2,\mu}$ of $\mu$ such that the closure of $\tphi^\Ga_*\cO_{2,\mu}$ is contained in $\cO_{2,\nu}$.  By definition of UC-approximation, there is a finite set $D_\phi \subset \G$ such that $\tphi$ is $D_\phi$-local. By Lemma \ref{lem:uc-al}, $\tphi$ is $\eta_\phi$-almost $L_\phi$-Lipschitz for some $L_\phi$ (when regarded as a map from $\cX^{D_\phi}$ to $\cY$).  Now choose $\kappa_{2,\mu}>0$ so that $\eta_\phi + L_\phi\kappa_{2,\mu} < \frac{\kappa_{2,\nu}}{4}$.  

Next we let $\cO_{1,\mu}$ be an arbitrary open set with $\mu \in \cO_{1,\mu} \subset \cO_{2,\mu}$ and let $\k_{1,\mu}$ be an arbitrary constant with $0<\k_{1,\mu} \le \k_{2,\mu}$.

We need to choose $\cO_{1,\nu}$ and $\kappa_{1,\nu}$. Before doing this, define $\psi:\cY^\Ga \to \cX$ by $\psi(y):=\Phi^{-1}(y)_e$. Note $\psi^\Ga = \Phi^{-1}$. We choose a UC-approximation $\tpsi$ to $\psi$ as follows. Choose $0<\eta_\psi<1$ so that $3 \sqrt{\eta_\psi}< \frac{\kappa_{1,\mu}}{2}$. By Lemma \ref{lem:composition} (and using $3 \sqrt{\eta_\phi}< \frac{\kappa_{2,\nu}}{8}$) we can choose $\eta_\psi$ smaller if necessary so that for any $\eta_\psi$-UC approximation $\tpsi$ to $\psi$, the composition $\tphi \tpsi^\G$ is a $\frac{\kappa_{2,\nu}}{8}$-UC approximation to $\phi \psi^\G$ which is the identity-coordinate projection. Fix such a UC-approximation $\tpsi$. Because $\tpsi$ and therefore $\tpsi^\G$ are continuous, there is an open neighborhood $\cO_{1,\nu}$ of $\nu$ such that $\cO_{1,\nu} \subset \cO_{2,\nu}$ and the closure of $\tpsi^\Ga_*\cO_{1,\nu}$ is contained in $\cO_{1,\mu}$. By choosing $\cO_{1,\nu}$ smaller if necessary we may assume that
$$\int d(\tphi \tpsi^\G z,z_e)~d\nu'(z) < \kappa_{2,\nu}/8$$
for every $\nu' \in \cO_{1,\nu}$. This is because the inequality holds if $\nu'=\nu$ (since $\tphi\tpsi^\G$ is a $\kappa_{2,\nu}/8$-UC approximation to the identity coordinate-projection) and the map $z \mapsto d(\tphi \psi^\G z,z_e)$ is continuous.

 By definition of UC-approximation, there is a finite set $D_\psi \subset \G$ such that $\tpsi$ is $D_\psi$-local. By Lemma \ref{lem:uc-al}, $\tpsi$ is $\eta_\psi$-almost $L_\psi$-Lipschitz for some $L_\psi$ (when regarded as a map from $\cY^{D_\psi}$ to $\cX$). Finally, choose $\kappa_{1,\nu}>0$ so that $\eta_\psi + L_\psi\kappa_{1,\nu} < \kappa_{1,\mu}$ and $\kappa_{1,\nu} < \frac{\kappa_{2,\nu}}{2}$.



To simplify notation, let 
$$Z(n,\nu):=Z^L_d(\cO_{1,\nu}, \k_{1,\nu},{\s_n}), \quad Z(n,\mu):=Z^L_d(\cO_{1,\mu}, \k_{1,\mu},{\s_n})$$
$$B(n,\nu):=B_d(\cO_{2,\nu},\k_{2,\nu},{\s_n}), \quad B(n,\mu):=B_d(\cO_{2,\mu},\k_{2,\mu},{\s_n})$$
$$H(n,\nu):= H^L_d(\cO_{1,\nu},\cO_{2,\nu},\k_{1,\nu},\k_{2,\nu},\s_n), \quad H(n,\mu):= H^L_d(\cO_{1,\mu},\cO_{2,\mu},\k_{1,\mu},\k_{2,\mu},\s_n).$$
Also let 
$$q_{n,\nu}:Z(n,\nu) \to H(n,\nu), \quad q_{n,\mu}: Z(n,\mu) \to H(n,\mu)$$
be the quotient maps.

Next we must verify that for all but finitely many $n$, $H(n,\nu)$ is a QS-group of $H(n,\mu)$. Our strategy is as follows. First we show that $\tpsi^{\s_n}_*(Z(n,\nu)) \subset Z(n,\mu)$. Let $S = q_{n,\mu}\circ \tpsi^{\s_n}_*(Z(n,\nu)) \le H(n,\mu)$. Next we show that $\tphi^{\s_n}$ induces a map, denoted by $\tphi^H$, from $S$ back to $H(n,\nu)$ that is surjective. This shows that $H(n,\nu)$ is a quotient of the subgroup $S \le H(n,\mu)$ and thereby completes the proof. 

It is convenient to first show that $\tpsi^{\s_n}$ and $\tphi^{\s_n}$ behave well with respect to the Hamming metrics and empirical distributions.





 \noindent {\bf Claim 1}.  For all but finitely many $n \in \N$ the following holds. For any $x,x' \in \Omega(\cO_{2,\mu},\s_n)$ and $y,y' \in  \Omega(\cO_{1,\nu},\s_n)$,
 \begin{enumerate}
 \item[(1)] $\tphi^{\s_n}(x) \in \Omega(\cO_{2,\nu},\s_n)$,
 \item[(2)] if $d^{V_n}_\cX(x,x') < \k_{2,\mu}$ then $d^{V_n}_\cY( \tphi^{\s_n}(x), \tphi^{\s_n}(x')) < \frac{\kappa_{2,\nu}}{4}$. 
 \item[(3)] $\tpsi^{\s_n}(y) \in \Omega(\cO_{1,\mu},\s_n)$,
 \item[(4)] If $d^{V_n}_\cY(y,y') < \k_{1,\nu}$ then $d^{V_n}_\cX( \tpsi^{\s_n}(y), \tpsi^{\s_n}(y')) < \k_{1,\mu}$. 
 \end{enumerate}
 
  \begin{proof}[Proof of Claim 1]
  
We chose $\cO_{2,\mu}$ so that the {\em closure} of $\tphi^\Ga_*\cO_{2,\mu}$ is contained in $\cO_{2,\nu}$. By compactness and Lemma \ref{L:TV}, there exists a finite set $E \subset \G$ and $\d>0$ such that for any $\a\in \tphi^\Ga_*\cO_{2,\mu}$, if $\b \in \Prob(\cY^\G)$ is such that the restricted measures $\Res^E_*\a, \Res^E_*\b$ have total variation distance $<\d$ then $\b \in \cO_{2,\nu}$.   
  
   Lemma \ref{lem:sofic-3} implies that if $n$ is sufficiently large (independent of $x$) then  the total variation distance of the restricted measures $\Res^E_*P^{\s_n}_{\tphi^{\s_n}(x)}$ and $\Res^E_*\tphi^\G_* P^{\s_n}_x \in \cO_{2,\nu}$ is bounded by $\d>0$.   Since $P^{\s_n}_{x} \in \cO_{2,\mu}$, $\tphi^\G_* P^{\s_n}_x \in \tphi^\Ga_*\cO_{2,\mu}$. So this implies $P^{\s_n}_{\tphi^{\s_n}(x)}\in \cO_{2,\nu}$. This proves (1). 
 
 Since $d^{V_n}_\cX(x,x')<\k_{2,\mu}$ and $\tphi$ is $\eta_\phi$-almost $L_\phi$-Lipschitz, Lemma \ref{lem:sofic-UC} implies
$$d^{V_n}_\cY( \tphi^{\s_n}(x), \tphi^{\s_n}(x')) \le L_\phi d^{V_n}_\cX(x,x') + \eta_\phi < L_\phi \k_{2,\mu} + \eta_\phi < \frac{\kappa_{2,\nu}}{4}.$$
This proves (2). The other statements are proven similarly.
 \end{proof}

  \noindent {\bf Claim 2}.  For all but finitely many $n \in \N$, for every $z \in Z(n,\nu)$,
  $$\tphi^{\s_n}_*\tpsi^{\s_n}_*(z)  - z \in B(n,\nu).$$
  
  \begin{proof}[Proof of Claim 2]
  We will verify the hypotheses of Lemma  \ref{lem:cylinder} with $\lambda=\tphi^{\s_n}\tpsi^{\s_n}$. So let $y \in \Omega(\cO_{1,\nu},\s_n)$. 
  
  By Claim 1 (3), $\tpsi^{\s_n}(y) \in \Omega(\cO_{1,\mu},\s_n)$. Since $\cO_{1,\mu} \subset \cO_{2,\mu}$, this implies $\tpsi^{\s_n}(y) \in \Omega(\cO_{2,\mu},\s_n)$. So Claim 1 (1) implies   $\tphi^{\s_n}\tpsi^{\s_n}(y) \in \Omega(\cO_{2,\nu},\s_n)$.
 
By Lemma \ref{lem:sofic-2} for all but finitely many $n$, 
$$ \#\{v \in V_n:~ (\tphi \tpsi^\Ga)^{\sigma_n}(y)_v \ne \tphi^{\sigma_n} \tpsi^{\sigma_n}(y)_v\} \le (\kappa_{2,\nu}/8) |V_n|.$$
Because $\cY$ has diameter 1,
$$d^{V_n}_\cY(\tphi^{\s_n}\tpsi^{\s_n}(y), y)\le (\kappa_{2,\nu}/8)+ d^{V_n}_\cY( (\tphi \tpsi^\G)^{\s_n}(y), y).$$
 Note
 $$ d^{V_n}_\cY( (\tphi \tpsi^\G)^{\s_n}(y), y) = \int d_\cY( (\tphi \tpsi^\G)(z), z_e)~dP^{\s_n}_y(z).$$
Since $P^{\s_n}_y \in \cO_{1,\nu}$, the choice of $\cO_{1,\nu}$ implies
  $$ \int d_\cY( (\tphi \tpsi^\G)(z), z_e)~dP^{\s_n}_y(z)<\kappa_{2,\nu}/8.$$
 Thus,
 $$d^{V_n}_\cY(\tphi^{\s_n}\tpsi^{\s_n}(y), y)\le \kappa_{2,\nu}/4$$
 for every $y\in \Omega(\cO_{1,\nu},\s_n)$. Claim 2 now follows from Lemma \ref{lem:cylinder} with $\kappa'=\kappa_{2,\nu}/4$ and $\kappa=\kappa_{1,\nu}<\kappa_{2,\nu}/2$. 
 \end{proof}
 
   \noindent {\bf Claim 3}.  For all but finitely many $n \in \N$,
   $$(\tpsi^{\s_n}_*)^{-1}(B(n,\mu))  \cap Z(n,\nu) \subset B(n,\nu).$$

  \begin{proof}[Proof of Claim 3]
 Let $z \in Z(n,\nu)$ and suppose $\tpsi^{\s_n}_*(z) \in B(n,\mu)$. By Claim 1 (1,2), $\tphi^{\s_n}_*\tpsi^{\s_n}_*(z) \in B(n,\nu)$. By Claim 2, $z-\tphi^{\s_n}_*\tpsi^{\s_n}_*(z) \in B(n,\nu)$. Therefore, $z \in B(n,\nu)$ as required.
 \end{proof}

 Since $\tpsi^{\s_n}$ commutes with the boundary map $\partial_d$, Claim 1 (3,4) implies that $\tpsi^{\s_n}_*(Z(n,\nu)) \subset Z(n,\mu)$. Therefore $S:= q_{n,\mu}\circ \tpsi^{\s_n}_*(Z(n,\nu))$ is a well-defined subgroup of  $H(n,\mu)$. Define $\tphi^H:S \to H(n,\nu)$ as follows. Given $w \in S$, let $z \in Z(n,\nu)$ be such that $q_{n,\mu}\circ \tpsi^{\s_n}_*(z)=w$. Then define $\tphi^H(w):= z + B(n,\nu) \cap Z(n,\nu)$. 
 
 To see that $\tphi^H$ is well-defined, suppose that $z' \in Z(n,\nu)$ also satisfies $q_{n,\mu}\circ \tpsi^{\s_n}_*(z')=w$. Then 
 $$\tpsi^{\s_n}_*(z) - \tpsi^{\s_n}_*(z') =  \tpsi^{\s_n}_*(z-z')  \in B(n,\mu).$$
 By Claim 3, this implies $z-z' \in B(n,\nu)$. This implies $\tphi^H$ is well-defined.
 
 To check that $\tphi^H$ is surjective, let $z \in Z(n,\nu)$. If $w =q_{n,\mu}\tpsi^{\s_n}_*(z) \in S$ then $\tphi^H (w) = z + B(n,\mu) \cap Z(n,\nu)$. This shows $\tphi^H$ is surjective. So $H(n,\nu)$ is a QS-group of $H(n,\mu)$ as required.
 
\end{proof}

\section{Homology computations}\label{sec:computations}

\subsection{Contractible model spaces}
Let $(\cX,d_\cX)$ be  compact totally disconnected metric space, $\mu \in \Prob_\G( \cX^\G)$, and $\Si$ be a sofic approximation to $\G$.

\begin{defn}[Contractible model spaces]
We say $\mu$ {\bf has contractible model spaces} with respect to $\Si$ if for every open neighborhood $\cO_2$ of $\mu$ in $\Prob( \cX^\G)$, and every $\d>0$  there exists an open neighborhood $\cO_1 \subset \Prob(\cX^\G)$ with $\mu \in \cO_1 \subset \cO_2$ such that for every $0<K<\infty$ and all but finitely many $n$, if $x_1,\ldots, x_K \in \Omega(\cO_1, \s_n)$ then there exist $x_i^{(j)}$ (for $j\ge 0$) such that 
\begin{enumerate}
\item $x_i^{(j)} \in \Omega(\cO_2,\s_n)$ for all $i,j$, 
\item $x_i^{(0)}=x_i$ for all $i$,
\item $d^{V_n}_\cX(x_i^{(j+1)}, x_k^{(j+1)}) \le d^{V_n}_\cX(x_i^{(j)}, x_k^{(j)})$ for all $i,j,k$,
\item $d^{V_n}_{\cX} (x_i^{(j)}, x_i^{(j+1)}) < \d$ for all $i,j$
\item there exists $M$ such that $x_1^{(M)} = \cdots = x_K^{(M)}$. This $M$ may depend on $x_1,\ldots, x_K$. In particular, $M$ may depend on $n$. 
\end{enumerate}
\end{defn}

\begin{defn}[Vanishing homology]
The measure $\mu$ is said to have {\bf vanishing reduced homology in dimension $d$ with respect to $\Si$} if $\forall$ open neighborhoods $\cO_2 \ni \mu$, $\forall \k_2>0$ there exist an open neighborhood $\cO_1$ with $\mu \in \cO_1 \subset \cO_2$ and $\kappa_1>0$ such that for every $L \in \N$, either $d>0$ and $H^L_d(\cO_{1},\cO_{2},\k_{1},\k_{2},\s_n)=0$ for all but finitely many $n$ or $d=0$ and  $H^L_d(\cO_{1},\cO_{2},\k_{1},\k_{2},\s_n)\cong \Z$ for all but finitely many $n$. By Theorem \ref{thm:mc}, this notion is a measure-conjugacy invariant. Hence it can also be applied to measure-preserving systems of the form $\G \cc (X,\mu)$ in which $(X,\mu)$ is a standard probability space without any additional structure. 
\end{defn}

\begin{prop}\label{prop:contractible}
If $\mu$ has contractible model spaces with respect to $\Si$ then $\mu$ has vanishing reduced homology in every dimension with respect to $\Si$.
\end{prop}

\begin{remark}
The proof of this Proposition is the only place in this paper where the finiteness of the parameter $L$ in the definition of the homology groups $H^L_d(\cdot)$ is used directly.  
\end{remark}

\begin{proof}
Let $\cO_2$ be an open neighborhood of $\mu$ and $\k_2>0$. Choose $0<\d<\k_2/3$ and let $\k_1:=\d$. Let $\cO_1$ be as in the definition of contractibility. Let $d \in \N$ be a dimension and $L>0$. Let $K\ge (d+1)L$. 
Let
$$z_n = \sum_{i=1}^k c_i s_i \in Z^L_d(\cO_1,\k_1,\s_n)$$
be a cycle of length $k \le L$. If $d>0$ then it suffices to show that $z_n \in B_d(\cO_2,\k_2,\s_n)$. 

Let $x_0,\ldots, x_m \in \Omega(\cO_1,\s_n)$ be an enumeration of the vertices contained in the oriented simplices $s_1,\ldots, s_k$. Note $m \le (d+1)L \le K$ is bounded independently of  $n$. 

Let $x_i^{(j)}$ be as in the definition of contractible. The map $\psi^{(j)}:\{x_0,\ldots, x_m\} \to \cX^{V_n}$ defined by $\psi^{(j)}(x_i)=x_i^{(j)}$ is well-defined. We can extend it to a map on all of $\cX^{V_n}$ by setting $\psi^{(j)}(x)=x$ for all $x \notin \{x_0,\ldots, x_m\}$. In particular, $\psi^{(j)}_*(s_i)$ is well-defined for every simplex $s_i$ and therefore $\psi^{(j)}_*(z_n)$ is also well-defined as an element of $C_d(\cX^{V_n})$. 

We claim that $\psi^{(j)}_*(z_n) \in Z_d(\cO_2,\k_1,\s_n)$ for all $j$. For $j=0$ this is true since $\psi^{(0)}_*(z_n)=z_n$. Assuming it is true for some $j \ge 0$, observe that by Lemma \ref{lem:cylinder},
$$\psi^{(j+1)}_*(z_n) - \psi^{(j)}_*(z_n)  \in B_d(\cO_2,\k_2,\s_n)$$
(this uses $\k_1 < \k_2/3$). Moreover $\psi^{(j+1)}_*(s_i) \in C_d(\cO_2,\k_1,\s_n)$ by property (3) in the definition of contractibility. Thus $\psi^{(j+1)}_*(z_n) \in Z_d(\cO_2,\k_1,\s_n)$. This completes the induction. 

Moreover, we showed that $\psi^{(j)}_*(z_n) - z_n \in B_d(\cO_2,\k_2,\s_n)$ for all $j$. If $d>0$ then $\psi^{(M)}_*(z_n)$ is trivial (by property (4)) and so $z_n \in B_d(\cO_2,\k_2,\s_n)$. This completes the proof in the case $d>0$.

If $d=0$ then every element $C_0(\cO_1,\k_1,\s_n)$ is a cycle (so $Z^L_0(\cO_1,\k_1,\s_n)$ is the free abelian group generated by $\Omega(\cO_1,\s_n)$). We have shown for any $x_1,x_2 \in \Omega(\cO_1,\s_n)$ there is a $\d$-path  in $\Omega(\cO_2,\s_n)$ connecting them (namely $x_1^{(0)}, \ldots, x_1^{(M)}, x_2^{(M)}, \ldots, x_2^{(0)}$). Therefore $x_1-x_2 \in B_0(\cO_2,\k_1,\s_n)$ which implies $H^L_d(\cO_{1},\cO_{2},\k_{1},\k_{2},\s_n)\cong \Z$ for all but finitely many $n$. 

\end{proof}

\subsection{Edit distance}
The goal of this section is to show that the homological invariants defined above depend on the sofic approximation only up to edit-distance zero as defined next.

For each finite set $S \subset \G$ and finite set $V$ there is a pseudo-metric $d_S$ on the set of all maps $\s:\G \to \sym(V)$ defined by 
$$d_S(\s_1,\s_2):=|V|^{-1} \#\{ v\in V:~ \exists s\in S \textrm{ such that } \s_1(s)v \ne \s_2(s)v\}.$$
Sofic approximations $\Si=\{\s_n:\G \to V_n\}$ and $\Si' = \{\s'_n:\G \to V_n\}$ with the same target sets $\{V_n\}_n$ are said be at {\bf edit-distance zero} if for every finite $S \subset \G$,
$$\limsup_{n\to\infty} d_S(\s_n,\s'_n) = 0.$$

\begin{prop}\label{prop:edit}
Suppose $\Si, \Si'$ are at edit-distance zero and $\mu \in \Prob_\G(\cX^\G)$. If $\mu$ has contractible model spaces with respect to $\Si$ then $\mu$ also has contractible model spaces with respect to $\Si'$.
\end{prop}

For the proof it will be helpful to have the next definition. 

\begin{defn}
An open subset $\cO \subset \Prob(\cX^\G)$ is {\bf $D$-local} (where $D\subset \G$ is finite) if there is an open set $\tilde{\cO} \subset \Prob(\cX^D)$ such that $\cO$ is the inverse image of $\tilde{\cO}$ under the projection map $\Prob(\cX^\G) \to \Prob(\cX^D)$. An open subset $\cO \subset \Prob(\cX^\G)$ is {\bf local} if it is $D$-local for some $D$.

\end{defn}

The next lemma can be used to show that many sofic invariants (such as entropy) depend on $\Si$ only up to edit-distance zero. We will use it to prove Proposition \ref{prop:edit}.

\begin{lem}\label{lem:edit0}
Suppose $\Si, \Si'$ have edit-distance zero and $\mu \in \Prob_\G(\cX^\G)$. Then for any open neighborhood $\cO \ni \mu$ there exists an open neighborhood $\cO'$ with $\mu \in \cO' \subset \cO$ such that $\Omega(\cO',\s'_n) \subset \Omega(\cO, \s_n)$ for all but finitely many $n$.
\end{lem}

\begin{proof}
 Let an open set $\cO \ni \mu$ be given. Let $\cO'$ be an open set containing $\mu$ such that the weak* closure of $\cO'$ is contained in $\cO$.  By compactness and Lemma \ref{L:TV}, there exists a finite set $E \subset \G$ and $\eps>0$ such that if $\nu \in \cO'$ and $\nu' \in \Prob(\cX^{D})$ is such that $d_{\textrm{TV}}(\Res^E_*\nu,\Res^E_*\nu')<\eps$ then $\nu' \in \cO$. 

Because $\Si, \Si'$ have edit-distance zero, $d_E(\s_n,\s'_n)<\eps$ for all but finitely $n$. This condition implies $d_{\textrm{TV}}(\Res^E_*(P^{\s_n}_x), \Res^E_*(P^{\s'_n}_x)) < \eps$ for all $x \in \cX^{V_n}$. In particular, if $x\in \Omega(\cO',\s'_n)$ then $x \in \Omega(\cO,\s_n)$. Thus $\Omega(\cO',\s'_n) \subset \Omega(\cO, \s_n)$.
\end{proof}

\begin{proof}[Proof of Proposition \ref{prop:edit}]

Let an open set $\cO'_2 \ni \mu$ and $\d>0$ be given. By Lemma \ref{lem:edit0} there exists an open set $\cO_2$ such that $\mu \in \cO_2 \subset \cO'_2$ and  $\Omega(\cO_2,\s_n) \subset \Omega(\cO'_2, \s'_n)$ for all but finitely many $n$. 

Let $\cO_1 \subset\cO_2$ be an open neighborhood of $\mu$ satisfying the definition of contractible model spaces with respect to $\Si$. So  for every $0<K<\infty$ and all but finitely many $n$, if $x_i \in \Omega(\cO_1, \s_n)$ $(1\le i \le K$) there exist $x_i^{(j)}$ (for $j\ge 0$) such that 
\begin{enumerate}
\item $x_i^{(j)} \in \Omega(\cO_2,\s_n)$ for all $i,j$, 
\item $x_i^{(0)}=x_i$ for all $i$,
\item $d^{V_n}_\cX(x_i^{(j+1)}, x_k^{(j+1)}) \le d^{V_n}_\cX(x_i^{(j)}, x_k^{(j)})$ for all $i,j,k$,
\item $d^{V_n}_{\cX} (x_i^{(j)}, x_i^{(j+1)}) < \d$ for all $i,j$
\item there exists $M$ such that $x_1^{(M)} = \cdots = x_K^{(M)}$. 
\end{enumerate}

By Lemma \ref{lem:edit0} there exists an open set $\cO'_1$ such that $\mu \in \cO'_1 \subset \cO_1$ and $\Omega(\cO'_1,\s'_n) \subset \Omega(\cO_1,\s_n)$ for all but finitely many $n$. 

Now let $x_1,\ldots, x_K \in \Omega(\cO'_1, \s'_n)$. Since $\Omega(\cO'_1,\s'_n) \subset \Omega(\cO_1,\s_n)$, there exist $x_i^{(j)}$ for $j \ge 0$ satisfying the above conditions (if $n$ is sufficiently large). In particular, 
$$x_i^{(j)} \in \Omega(\cO_2,\s_n) \subset \Omega(\cO'_2,\s'_n)$$
for $n$ sufficiently large (independent of $x_1,\ldots, x_K$). This proves $\mu$ has contractible model spaces with respect to $\Si'$. 
\end{proof}

\subsection{Diffuse sofic approximations}

One of the main goals of this section is to prove that if $\G$ is amenable then every $\mu \in \Prob_\G(\cX^\G)$ has contractible model spaces. In fact, more is true, one only needs that the sofic approximation $\Si$ is diffuse. This condition, explained below, holds automatically if $\G$ is amenable. Moreover, even if $\G$ is non-amenable then diffuse sofic approximations can be constructed out of arbitrary sofic approximations.

The {\bf disjoint union} of maps $\s_i:\G \to \sym(V_i)$ (for $i=1,2$) is the map $\s_1\sqcup \s_2: \G \to \sym(V_1 \sqcup V_2)$ defined by $\s_1\sqcup \s_2(g)v = \s_i(g)v$ if $v \in V_i$.

\begin{defn}
A sofic approximation $\Si=\{\s_n: \G \to \sym(V_n)\}_{n \in \N}$ is {\bf diffuse} if there exists a sofic approximation $\Si'=\{\s'_n: \G \to \sym(V_n)\}_{n \in \N}$ such that
\begin{itemize}
\item $\Si$ and $\Si'$ have edit-distance zero,
\item for every $n$, $\s'_n$ can be expressed as a disjoint union $\s'_n:=\s'_{n,1} \sqcup \cdots \sqcup \s'_{n,m_n}$ such that if $\s'_{n,i}:\G \to \sym(V_{n,i})$ then
$$\max_{1\le i \le m_n} |V_{n,i}| = o(|V_n|).$$ 
\end{itemize}
\end{defn}

\begin{example}
Let $\Si=\{\s_n: \G \to \sym(V_n)\}_{n \in \N}$ by any sofic approximation to any group $\G$ and let $\{W_n\}_n$ be a sequence of finite sets with $|W_n| \to \infty$   as $n\to\infty$. Define $\s'_n: \G \to \sym(V_n \times W_n)$ by 
$$\s'_n(g)(v,w)= ( \s_n(g)v, w).$$
In other words, $\s'_n$ is the direct product of $\s_n$ with the trivial homomorphism $\G \to \sym(W_n)$. Then $\Si' =\{\s'_n\}$ is diffuse. In fact $\s'_n$ is the disjoint union of $W_n$ copies of $\s_n$.
\end{example}

\begin{lem}
If $\G$ is amenable then every sofic approximation $\Si$ to $\G$ is diffuse.
\end{lem}

\begin{proof}
The special case in which $\G$ is finitely generated is a direct consequence of \cite[Proposition 2.8]{MR2823074}. The general case follows from the finitely generated case by a diagonalization argument. 
\end{proof}

\subsubsection{Diffuse approximations and contractibility}

\begin{thm}\label{thm:diffuse}
If $\Si$ is diffuse and $\mu \in \Prob_\G(\cX^\G)$ then $\mu$ has contractible model spaces with respect to $\Si$. In particular, if $\G$ is amenable then $\mu$ has contractible model spaces with respect to every sofic approximation.
\end{thm}

Because $\Si$ is diffuse, after replacing it with another sofic approximation at edit-distance zero, we may assume that each set $V_n$ comes equipped with a partition $V_n = \sqcup_i V_{n,i}$ so that the image of $\s_n:\G \to \sym(V_n)$ lies in the direct product $\prod_i \sym(V_{n,i})$. The idea behind the proof is to modify  $x_1,\ldots, x_K \in \cX^{V_n}$ on each of the $V_{n,i}$'s to form paths $(x^{(j)}_i)_{i,j}$ that merge together. However, we must be careful so that each of the $x^{(j)}_i$'s has empirical distribution close to $\mu$. To accomplish this, we coarsen the given partition so that each $x_i$ restricted to any part of the good partition has empirical distribution close to $\mu$. This explains why the next two lemmas are needed. (No attempt has been made to optimize the constants below).

\begin{lem}
Let $(\Omega,\P)$ be a standard probability space and $f:\Omega \to [0,1]$ be a measurable random variable. Let $\delta \in (0,1/4)$, $0<\eps<\delta/400$ and suppose that $\P(\{\omega\})<\eps$ for every $\omega \in \Omega$. Then there exists a finite measurable partition $\cP$ of $\Omega$ such that 
\begin{enumerate}
\item $\P(P) \le 100\eps/\delta$ for every $P \in \cP$,
\item $\|\E[f|\cP] - \E[f]\|_{L^\infty(\Omega,\P)} \le  \delta$
\end{enumerate}
where $\E$ denotes expectation with respect to $\P$. 
\end{lem}

\begin{proof}
After passing to an image of the measure space $(\Omega,\P)$ if necessary, we may assume without loss of generality that $\Omega$ is a finite set. So let $\Omega=\{\omega_1,\ldots, \omega_n\}$ be ordered so that $f(\omega_1)\le f(\omega_2)\le \cdots \le f(\omega_n)$. After making a small perturbation if necessary, we may also assume that $f(\omega) \ne \E[f]$ for any $\omega \in \Omega$. 

Define a piecewise constant function $F:[0,1] \to [0,1]$ by
$$F(x)= f(\omega_{h(x)})$$
where $h(x)$ is the smallest number such that $x \le \P(\{\omega_1,\ldots, \omega_{h(x)}\})$. Note that $\int F~dx=\E[f]$ and $F$ is monotone increasing. We will first solve the problem with $(\Omega, \P)$ and $f$ replaced by Lebesgue measure on the unit interval and $F$. 

\noindent {\bf Claim 1}. There exists a natural number $m$ such that $20\eps/\delta \le 1/m \le 21\eps/\delta$. 

\begin{proof}
Let $m \in \N$ be such that $\frac{1}{m+1} < 20\eps/\delta \le \frac{1}{m}$. Because $\eps<\delta/400$, $20\eps/\delta <1/20$ and $m \ge 20$. 

It suffices to show $1/m \le 21\eps/\delta$. Equivalently, it suffices to show $21m \eps/\delta \ge 1$. The condition $\frac{1}{m+1} < 20\eps/\delta $ implies $20(m+1)\eps/\delta>1$. Because $m\ge 20$, we have $21m \ge 20(m+1)$. So
$21m \eps/\delta  \ge 20(m+1)\eps/\delta>1.$
\end{proof}

\noindent {\bf Claim 2}. There is a finite partition $\cQ$ of $[0,1]$ such that every $Q \in \cQ$ satisfies
\begin{enumerate}
\item $Q$ is a union of at most two disjoint intervals;
\item $20\eps/\delta \le |Q| \le 21\eps/\delta$ where $|\cdot |$ denotes Lebesgue measure;
\item $|Q|^{-1} \int_Q F(x)~dx = \E[f].$
\end{enumerate}

\begin{proof}
Because $f(\omega) \ne \E[f]$ for any $\omega \in \Omega$, $F(t) \ne \int_0^1~F(x)~dx$ for any $t$. Since $F$ is monotone increasing, there exists a unique $x_0 \in [0,1]$ such that for all $a$ and $b$ with $0\le a<x_0<b\le 1$,
$$F(a)<\int_0^1 F~dx<F(b).$$
Because $F$ is monotone, for every $0\le a < x_0$, there exists a unique $g(a)$ with $x_0 < g(a)\le 1$ such that
 $$\frac{1}{g(a)-a} \int_a^{g(a)} F~dx = \int_0^1 F ~dx= \E[f].$$
 Note $g$ is  an orientation-reversing homeomorphism of $[0,x_0)$ onto $(x_0,1]$. 
 
 
Using Claim 1, choose points $0=a_1 <a_2<\cdots<a_m < x_0$ such that for each $1\le i<m$,
$$20\eps/\delta \le 1/m = (a_{i+1}-a_i) + g(a_i)-g(a_{i+1}) \le 21\eps/\delta$$
and 
$$20\eps/\delta \le 1/m= g(a_m)-a_m \le 21\eps/\delta.$$

Let $\cQ$ be the partition containing  $[a_m,g(a_m)]$ and $[a_i,a_{i+1}) \cup (g(a_{i+1}),g(a_i)]$ for $i<m$.  
\end{proof}

Claim 2 solves the problem for $F$. To solve it for the original function $f$, we will obtain a partition $\cP$ that approximates $\cQ$. 

Recall that for $x\in [0,1]$, $h(x)$ is the smallest index such that $x \le \P(\{\omega_1,\ldots, \omega_{h(x)}\})$. Define $h_\Omega:[0,1] \to \Omega$ by $h_\Omega(x)=\omega_{h(x)} \in \Omega$. For $X \subset [0,1]$, let 
\begin{eqnarray*}
X^+ &=& h_\Omega^{-1}(h_\Omega(X)),\\
X^- &=& \{x \in X:~ h_\Omega^{-1}(h_\Omega(x)) \subset X\}.
\end{eqnarray*}
Let $\tcP$ be an $h_\Omega$-measurable partition of $[0,1]$ such that for every $P \in \tcP$ there exists $Q \in \cQ$ such that $Q^- \subset P \subset Q^+$. Because $Q$ is a disjoint union of at most two intervals, $h_\Omega(Q^+) \setminus h_\Omega(Q^-)$ contains at most 4 elements of $\Omega$. Since each element of $\Omega$ has measure $<\eps$ and $h_\Omega$ is measure-preserving, $|Q\vartriangle P| \le 4\eps$. So  
$$|P|  \le |Q| + 4\eps \le 21\eps/\delta + 4\eps \le 100\eps/\delta.$$ 
Also
\begin{eqnarray*}
\left| |P|^{-1} \int_P F~dx - \E[f] \right| &=& \left| |P|^{-1} \int_{P} F(x)~dx - |Q|^{-1} \int_{Q} F(x)~dx \right|.
\end{eqnarray*}
We may decompose the first integral as $\int_{P} = \int_{P \cap Q} + \int_{P\setminus Q}$ and the second one similarly. From this, we see that the above is bounded by
\begin{eqnarray*}
&&| |P|^{-1} - |Q|^{-1}| \int_{P\cap Q} F(x)~dx+ |P|^{-1}\int_{P \setminus Q}F(x)~dx + |Q|^{-1}\int_{Q \setminus P}F(x)~dx \\
&\le& | |P|^{-1} - |Q|^{-1}| |P \cap Q|  + |P|^{-1}|P \setminus Q| + |Q|^{-1}|Q \setminus P| \\
&\le& 4\eps\left( \frac{|Q \cap P|}{|P||Q|}  + |P|^{-1}+ |Q|^{-1} \right).
\end{eqnarray*}
We use the bounds $|P\cap Q| \le 21\eps/\delta$ and $|Q|\ge 20\eps/\delta$, $|P| \ge 20\eps/\delta-4\eps \ge 19\eps/\delta$ to obtain
$$\left| |P|^{-1} \int_P F ~dx - \E[f] \right| <  \delta.$$
Since $P$ is arbitrary, this shows $\|\E[F|\tcP] - \E[F]\|_{\infty} \le  \delta$. Since $\tcP$ is $h_\Omega$-measurable, it induces a partition $\cP=h_\Omega(\tcP)$ on $\Omega$. Because $h_\Omega$ is measure-preserving, the required properties of $\cP$ follow from the corresponding properties of $\tcP$. 
 
\end{proof}

\begin{lem}\label{lem:diffuse2}
Let $(\Omega,\P)$ be a standard probability space and $f_1,\ldots, f_m:\Omega \to [0,1]$ be random variables. Let $0<\delta<1/4$, $0<\eps$,  and suppose that $\P(\{\omega\})<\eps$ for every $\omega \in \Omega$. Assume $\eps < (\d/400)^m$. Then there exists a measurable partition $\cP$ of $\Omega$ such that 
\begin{enumerate}
\item $\P(P) \le \eps(100/\d)^m$ for every $P \in \cP$,
\item $\|\E[f_i|\cP] - \E[f_i]\|_{L^\infty(\Omega,\P)} \le \delta$ for every $1\le i \le m$.
\end{enumerate}
\end{lem}

\begin{proof}
We prove this by induction on $m$. The previous lemma establishes the base case $m=1$.

For the inductive step, assume $m\ge 2$ and there is a measurable partition $\cQ$ of $\Omega$ such that 
\begin{enumerate}
\item $\P(Q) \le \eps(100/\d)^{m-1}$ for every $ Q\in \cQ$,
\item $\|\E[f_i|\cQ] - \E[f_i]\|_{L^\infty(\Omega,\P)} \le \delta$ for every $1\le i \le m-1$.
\end{enumerate}
Apply the previous lemma with $(\cQ,\P)$ in place of $(\Omega,\P)$, $\E[f_m|\cQ]$ in place of $f$ and $\eps(100/\d)^{m-1}$ in place of $\eps$ to obtain a partition $\cP$ of $\Omega$ that coarsens $\cQ$ and satisfies
\begin{enumerate}
\item $\P(P) \le \eps(100/\d)^{m}$ for every $ P\in \cP$,
\item $\|\E[\E[f_m|\cQ] | \cP] - \E[f_m]\|_{L^\infty(\Omega,\P)} = \|\E[f_m|\cP] - \E[f_m]\|_{L^\infty(\Omega,\P)}  \le \delta$.
\end{enumerate}
Since $\cP$ coarsens $\cQ$, $\|\E[f_i|\cP] - \E[f_i]\|_{L^\infty(\Omega,\P)} \le \delta$ holds for every $1\le i \le m-1$ too.
\end{proof}

\begin{proof}[Proof of Theorem \ref{thm:diffuse}]
 By Proposition \ref{prop:edit}, we may assume without loss of generality that for every $n$, $\s_n$ can be expressed as a disjoint union $\s_n:=\s_{n,1} \sqcup \cdots \sqcup \s_{n,m_n}$ such that if $\s_{n,i}:\G \to \sym(V_{n,i})$ then
$$\max_{1\le i \le m_n} |V_{n,i}| = o(|V_n|).$$ 

Let $\mu \in \Prob_\G(\cX^\G)$, $\cO_2$ be an open neighborhood of $\mu$ in $\Prob(\cX^\G)$ and $\d>0$ be given. By choosing $\cO_2$ smaller if necessary, we may assume it is convex. Let  $\cO_1 \subset \cO_2$ be an open neighborhood of $\mu$ such that the closure of $\cO_1$ is contained in $\cO_2$. It follows that there exist $\d_2>0$ and continuous functions $\tf_1,\ldots, \tf_m: \cX^\G \to [0,1]$ such that if $\nu \in \cO_1$ and $\nu'$ is such that $|\nu(\tf_i)-\nu'(\tf_i)|<\d_2$ for all $1\le i \le m$ then $\nu' \in \cO_2$.



Let $x_0,\ldots, x_K \in \Omega(\cO_1,\s_n)$. Let $\Omega=\{1,\ldots, m_n\}$ and $\P$ be the probability measure on $\Omega$ given by $\P(k)=|V_{n,k}|/|V_n|$. 

Let $f_{i,j}:V_n \to [0,1]$ be given by
$$f_{i,j}(v) = \tf_i (\Pi^{\s_n}_v x_j).$$
Define $\barf_{i,j}:\Omega \to [0,1]$ by 
$$\barf_{i,j}(k) = |V_{n,k}|^{-1} \sum_{v \in V_{n,k} } f_{i,j}(v).$$

For sufficiently large $n$, apply Lemma \ref{lem:diffuse2} to $(\Omega,\P)$ and the functions $(\barf_{i,j})$ to obtain a partition $\cP_n$ of $\Omega$. Let $\cQ_n$ be the partition of $V_n$ defined by pulling back the partition $\cP_n$ under the map $V_n \to \Omega$ defined by $v \mapsto i$ if $v \in V_{n,i}$. Then
\begin{enumerate}
\item $\cQ_n$ coarsens the partition $V_n = \sqcup_i V_{n,i}$.
\item For each $i,j$, $\| \E[f_{i,j}| \cQ_n] - \E[f_{i,j}] \|_{L^\infty(\Omega,\P)} \le \d_2$. By the choice of $f_{i,j}$ and $\d_2$, this implies that for every $Q \in \cQ_n$, the empirical measure of $x_j \resto Q$ with respect to $\s_n\resto Q$ lies in $\cO_2$ (where $\s_n \resto Q$ is the map $\G \to \sym(Q)$ obtained by restriction). In symbols, $P_{x_j \resto Q}^{\s_n \resto Q} \in \cO_2$. 

\item $\max_{Q \in \cQ_n} |Q| \le \eps_n |V_n|$ for some constants $\eps_n>0$ with $\eps_n \to 0$ as $n\to\infty$.
\end{enumerate}


For $0\le j$, define $\Psi_j:\cX^{V_n} \to \cX^{V_n}$ by $\Psi_j(x)(v)=x_0(v)$ if $v \in Q_i$ for some $i\le j$ and $\Psi_j(x)(v)=x(v)$ otherwise. Set $x_i^{(j)} = \Psi_j(x_i)$. 

Since $P_{x_i \resto Q}^{\s_n \resto Q} \in \cO_2$ for every $Q \in \cQ_n$ and $\cO_2$ is convex, $x_i^{(j)} \in \Omega(\cO_2,\s_n)$ for all $i,j$.

If $n$ is large enough then $\eps_n \diam(\cX)< \delta$ and therefore $d^{V_n}_{\cX} (x_i^{(j)}, x_i^{(j+1)}) < \d$ for all $i,j$. The map $\Psi_j$ is distance contracting, $\Psi_0$ is the identity and $\Psi_{m_n}$ maps all of $\{x_0,\ldots, x_K\}$ to $x_0$. This verifies all of the conditions in the definition of contractible model spaces.
\end{proof}

\subsection{Bernoulli shifts}

\begin{thm}\label{thm:bernoulli}
Let $\G$ be a countably infinite group, $\Si$ a sofic approximation, $(\cX,d_\cX), (\cY,d_\cY)$ totally disconnected compact metric spaces and $\beta \in \Prob(\cX)$ a probability measure on $\cX$. Then the Bernoulli shift $\G \cc (\cX,\beta)^\G$ has contractible model spaces with respect to $\Si$. Moreover, for any $\nu \in \Prob_\G(\cY^\G)$, the $0$-dimensional sofic homology theories of $\nu$ and $\b^\G\times \nu$ are equivalent. 
\end{thm}

In  \cite{MR3543677}, Tim Austin proved a similar result. The proof here follows the same strategy. 

\begin{remark}
Unfortunately, it is not clear whether the $d$-dimensional sofic homology theories of $\nu$ and $\b^\G\times \nu$ are equivalent for $d>0$. To explain why, suppose $z,z' \in Z_d^L(\cO_1,\k_1,\s)$ are each representable as a weighted sum of at most $L$ simplices and $z-z' \in B_d(\cO_2,\k_2,\s)$. Then there is a $(d+1)$-chain $w$ such that $z-z' = \partial_d w$. But the proof in the $0$-dimensional case uses that there exists a sequence $z=z_1,\ldots, z_k=z'$ of cycles interpolating between $z$ and $z'$ such that there is a uniform bound on the $\ell^1$-norm $\|z_i -z_{i+1}\|_1$ for each $i$ and each $z_i$ is representable as a weighted sum of at most $L'$ cycles for some $L'$ that does not depend on $z,z'$. There does not appear to be any good reason why this property should hold  if $d>0$. 

\end{remark}

We will use special neighborhoods of $\beta^\G\times \nu$ defined as follows. For any finite $D \subset \G$, let $\cB_{\cX^D}$ be the smallest Borel sub-sigma-algebra on $\cX^\G\times \cY^\G$ such that the projection $\cX^\G \times \cY^\G \to \cX^{D} \times \cY^\G$ is $\cB_{\cX^D}$-measurable.

A subset $\cF \subset C( \cX^\G\times \cY^\G)$ of continuous functions is {\bf hereditary} if there is some finite $D \subset \G$ such that every $f\in \cF$ is $D$-local and the conditional expectation $\E_{\beta^\G\times \nu}[f | \cB_{\cX^C}] \in \cF$ for every $C \subset D$. Moreover, we require that $f$ is $1$-Lipschitz in the $\cX^D$-variable as a function from $\cX^D\times \cY^D$ to $\R$. To be precise this means that
$$|f(x_1,y) - f(x_2,y)| \le d_\cX^D(x_1,x_2)$$
for every $x_1,x_2 \in \cX^D$ and $y \in \cY^D$ (where we have abused notation by identifying $f$ with its projection to $\cX^D \times \cY^D$).

A neighborhood $\cO \subset \Prob(\cX^\G\times \cY^\G)$ is {\bf hereditary} if there is a finite hereditary subset $\cF \subset C(\cX^\G\times \cY^\G)$ and $\d>0$ such that
$$\cO=\{\mu' \in \Prob(\cX^\G\times \cY^\G):~ |\mu'(f) - \b^\G\times\nu(f)| < \d \quad \forall f\in \cF\}.$$

 
 The next proposition shows that the measure $\beta^\G\times \nu$ satisfies a property that is a kind of relative version of having contractible model spaces. 
\begin{prop}\label{prop:contraction2}
Let $\cO_1 \subset \cO_2 \subset \Prob(\cX^\G\times \cY^\G)$ be open neighborhoods of $\beta^\G\times \nu$. Suppose that the closure of $\cO_1$ is contained in $\cO_2$ and that $\cO_1$ is hereditary. Then for any $\k>0$ and $K>0$ there exists $N$ such that if $n>N$ then for any $(x_i,y_i) \in \Omega(\cO_1,\s_n)$ $(1\le i \le K$) there exist $x_i^{(j)}$ (for $j\ge 0$) such that for all $i,j,k$,
\begin{enumerate}
\item $(x_i^{(j)}, y_i) \in \Omega(\cO_2,\s_n)$, 
\item $x_i^{(0)}=x_i$,
\item $d^{V_n}_{\cX} (x_i^{(j)}, x_i^{(j+1)}) < \k$,
\item $d^{V_n}_\cX(x_i^{(j+1)}, x_k^{(j+1)}) \le d^{V_n}_\cX(x_i^{(j)}, x_k^{(j)})$ and
\item there exists $M$ such that $x_1^{(M)} = \cdots = x_K^{(M)}$. In fact we may choose the $x_i^{(j)}$'s so that $M = \lceil \diam(\cX)/\k\rceil + 1$.  
\end{enumerate}
\end{prop}


\begin{lem}\label{lem:contraction3}
Let $\cO \subset \Prob(\cX^\G\times \cY^\G)$ be an open neighborhood of $\b^\G\times \nu$. Then there exists an open neighborhood  $\cO' \subset \Prob(\cY^\G)$ of $\nu$ such that for all but finitely many $n$, if $y \in \Omega(\cO',\s_n)$ then there exists $x \in \cX^{V_n}$ such that $(x,y) \in \Omega(\cO,\s_n)$.  
\end{lem}
\begin{proof}
This is equivalent to saying that $\cX^\G\times \cY^\G \to \cX^\G$ is model-surjective (in the language of \cite{MR3543677}). This is proven implicitly in \cite[Theorem 8.1]{bowen-jams-2010} and \cite[Theorem 6.8]{MR3543677}. 
\end{proof}



\begin{proof}[Proof of Theorem \ref{thm:bernoulli} given Proposition \ref{prop:contraction2}]

The fact that Bernoulli shifts have contractible model spaces is implied by the special case of Proposition  \ref{prop:contraction2} in which $\nu$ is the Dirac mass on a fixed point. 

Before proving the second statement, note that the projection map $\Prob(\cX^\G\times \cY^\G) \to \Prob(\cY^\G)$ is open (this is implied by \cite[Theorem 2.5]{MR453960}). So if $\cO \subset \Prob(\cX^\G\times \cY^\G)$ is open then its image, which we denote by $\Proj_{\cY^\G}(\cO)$, is an open subset of $\Prob(\cY^\G)$.

First we will show that the $0$-dimensional sofic homology of $\nu$ is greater than or equal to the $0$-dimensional sofic homology of $\b^\G\times \nu$. In order to define the homology of $\b^\G\times \nu$, we identify $\cX^\G\times \cY^\G$ with $(\cX\times \cY)^\G$ and use the metric $d_{\cX\times \cY}((x,y),(x',y')) = d_\cX(x,x') + d_\cY(y,y')$. 

Let $\cO_2$ be an arbitrary open neighborhood of $\b^\G\times \nu$ in $\Prob(\cX^\G\times \cY^\G)$. Also let $\k_2>0$.  By Lemma \ref{lem:contraction3} there exists an open neighborhood $\cO'_2 \subset \Proj_{\cY^\G}(\cO_2)$ of $\nu$ such that for all but finitely many $n$, if $y \in \Omega(\cO'_2,\s_n)$ then there exists $x \in \cX^{V_n}$ such that $(x,y) \in \Omega(\cO_2,\s_n)$.  Let $\k'_2=\k_2$. Let $\cO'_1 \subset \cO'_2$ be an arbitrary open neighborhood of $\nu$. Also let $0< \k'_1$ be arbitrary. Choose an open neighborhood $\cO_1$ of $\b^\G\times \nu$ so that its closure is contained in $\cO_2$ and $\Proj_{\cY^\G}(\cO_1)$ is contained in $\cO'_1$. Because hereditary neighborhoods form a basis, we may also choose $\cO_{1}$ to be hereditary.  Let $\k_1=\k'_1$.

Let $\pi:(\cX\times \cY)^{V_n} \to \cY^{V_n}$ be the projection map. Let $S$ be the subgroup of $H_0(\cO'_1,\cO
'_2,\k'_1,\k'_2,\s_n)$ generated by the set of all 0-chains of the form $\pi_*([x,y])=[y]$ for $(x,y) \in \Omega(\cO_1,\s_n)$.  We claim that the map
$$\pi_*([x,y]) \mapsto [x,y]$$
from $S$ to $H_0(\cO_1,\cO_2,\k_1,\k_2,\s_n)$ is well-defined. It suffices to show that if $(x_1,y_1)$ and $(x_2,y_2)$ are both in $\Omega(\cO_1,\s_n)$ and  $y_1-y_2 \in B_0(\cO'_2,\k_2,\s_n)$  then there is a $\k_2$-path from $(x_1,y_1)$ to $(x_2,y_2)$ in $\Omega(\cO_2,\s_n)$. 

Because $y_1-y_2 \in B_0(\cO'_2,\k_2,\s_n)$ there exists a $\k_2$-path $w_1,\ldots, w_m \in \Omega(\cO'_2,\s_n)$ from $w_1=y_1$ to $w_m=y_2$. By Lemma \ref{lem:contraction3}, there exist $u_1,\ldots, u_m \in \cX^{V_n}$ such that $(u_i,w_i) \in \Omega(\cO_2,\s_n)$ for all $i$. We may assume $u_1=x_1, u_m=x_2$. 

Since $(x_1,y_1)-(x_2,y_2) = \sum_{i=1}^{m-1} (u_i,w_i) - (u_{i+1},w_{i+1})$, it suffices to show that $(u_i,w_i) - (u_{i+1},w_{i+1}) \in B_0(\cO_2,\k_2,\s_n)$ for all $1\le i < m$. 


So fix $i$ with $1\le i < m$. By Proposition \ref{prop:contraction2} (with $K=2, \kappa=\k_1$ and all large $n$) for $j\ge 0$ there exist elements $u_i^{(j)}, u_{i+1}^{(j)}$ such that for all $k\in \{i,i+1\}$ and $j\ge 0$, 
\begin{enumerate}
\item $(u_k^{(j)}, w_k) \in \Omega(\cO_2,\s_n)$,
\item $u_k^{(0)} = u_k$,
\item $d^{V_n}_{\cX} (u_k^{(j)}, u_k^{(j+1)}) < \k_1$,
\item there exists $M_i$ such that $u_i^{(M_i)} = u_{i+1}^{(M_i)}$. 
\end{enumerate}
It follows that 
$$(u_i,w_i), (u_i^{(1)}, w_i),\ldots, (u_i^{(M_i)}, w_i), (u_{i+1}^{(M_i)}, w_{i+1}), \ldots, (u_{i+1}^{(1)}, w_{i+1}), (u_{i+1}, w_{i+1})$$
is a $\k_2$-path from $(u_i,w_i)$ to $(u_{i+1}, w_{i+1})$ in $\Omega(\cO_2,\s_n)$ (with respect to the metric $d^{V_n}_{\cX\times \cY}$). Thus $(u_i,w_i) - (u_{i+1},w_{i+1}) \in B_0(\cO_2,\k_2,\s_n)$ as required.

So the map from $S$ to $H_0(\cO_1,\cO_2,\k_1,\k_2,\s_n)$ is well-defined. It is also surjective by construction.
 So $H_0(\cO_1,\cO_2,\k_1,\k_2,\s_n)$ is a QS-group of $H_0(\cO'_1,\cO'_2,\k'_1,\k'_2,\s_n)$ and therefore, the $0$-dimensional sofic homology of $\b^\G\times \nu$ is bounded by the $0$-dimensional sofic homology of $\nu$.

To finish the proof, we will show that the $0$-dimensional sofic homology of $\b^\G\times \nu$ is greater than or equal to the $0$-dimensional sofic homology of $\nu$. So let $\cO'_2$ be an arbitrary open neighborhood of $\nu$ in $\Prob(\cY^\G)$. Also let $\k'_2>0$. Let $\cO_2 = \Prob(\cX^\G)\times \cO'_2$ and $\k_2=\k'_2$. Let $\cO_1 \subset \cO_2$ be an arbitrary open neighborhood of $\b^\G\times \nu$. Let $0<\k_1\le \k_2$. By Lemma \ref{lem:contraction3}, there exists an open neighborhood $\cO'_1$ of $\nu$ such that $\cO'_1 \subset \cO'_2$ and for all but finitely many $n$, if $y\in \Omega(\cO'_1,\s_n)$ then there exists $x \in \cX^{V_n}$ such that $(x,y) \in \Omega(\cO_1, \s_n)$. Set $\k'_1=\k_1$. 

Let $S$ be the subgroup of $H_0(\cO_1,\cO_2,\k_1,\k_2,\s_n)$ generated by all $0$-cycles of the form $[x,y]$ with $y \in \Omega(\cO'_1,\s_n)$ and $(x,y) \in \Omega(\cO_1,\s_n)$. 
The map $[x,y] \to [y]$ from $S$ to $H_0(\cO'_1,\cO'_2,\k'_1,\k'_2,\s_n)$ is well-defined because if $[x_1,y_1] - [x_2,y_2] \in B_0(\cO_2,\k_2,\s_n)$ then there is a $\k_2$-path from $[x_1,y_1]$ to $[x_2,y_2]$ in $\Omega(\cO_2,\s_n)$. The projection of this path to $\cY^{V_n}$ is a $\k'_2$-path from $[y_1]$ to $[y_2]$ in $\Omega(\cO'_2,\s_n)$. So $[y_1]$ and $[y_2]$ represent the same element of $H_0(\cO'_1,\cO'_2,\k'_1,\k'_2,\s_n)$. 

We claim the map $[x,y] \to [y]$ from $S$ to $H_0(\cO'_1,\cO'_2,\k'_1,\k'_2,\s_n)$ is surjective. Suppose $[y] \in H_0(\cO'_1,\cO'_2,\k'_1,\k'_2,\s_n)$. Then $y \in \Omega(\cO'_1,\s_n)$. By choice of $\cO_1$, this implies the existence of $x \in \cX^{V_n}$ with $(x,y) \in \Omega(\cO_1,\s_n)$. Thus $[x,y] \in S$. Since $[y]$ is arbitrary this implies the claimed surjectivity. So $H_0(\cO'_1,\cO'_2,\k'_1,\k'_2,\s_n)$ is a QS-group of $H_0(\cO_1,\cO_2,\k_1,\k_2,\s_n)$ as required.

\end{proof}



We will use a well-known concentration inequality on Hamming cubes to prove Proposition \ref{prop:contraction2}. First we need some notation. Let $\l_s$ denote the probability measure on $\{0,1\}$ given by $\l_s(1)=s$, $\l_s(0)=1-s$. Let $\P_s = \l_s^{V_n} \times \b^{V_n}$ be a probability measure on $\{0,1\}^{V_n}\times \cX^{V_n}$. Let $\E_s$ denote  expectation with respect to $\P_s$.

\begin{prop}\label{prop:concentration}
There exists a constant $C>0$ (depending only on the diameter of $\cX$) such that for any $\eps>0$, any $s\in [0,1]$ and any 1-Lipschitz function $F:\{0,1\}^{V_n}\times \cX^{V_n} \to \R$,
$$\P_s\left\{ \Big| F - \E[F] \Big| \ge \eps \right\} \le 2e^{-C\eps^2 |V_n|}.$$
Here we are using the normalized Hamming metric $d^{V_n}_{\{0,1\}\times \cX}$ given by
\begin{eqnarray*}
d^{V_n}_{\{0,1\}\times \cX}( (\chi_1, x_1), (\chi_2,x_2) ) &=& d^{V_n}_{\{0,1\}}( \chi_1, \chi_2) +  d^{V_n}_{\cX}( x_1, x_2) \\
 &=& |V_n|^{-1} \#\{v\in V_n:~\chi_1(v) \ne \chi_2(v)\} + |V_n|^{-1}\sum_{v\in V_n} d_\cX(x_1(v),x_2(v)).
 \end{eqnarray*}
\end{prop}
For the proof see \cite[Corollary 1.17]{MR1849347}.

Proposition \ref{prop:contraction2} is proven by letting $x_i^{(j)}$ (for $j=0,1,\ldots$) be the result of a coupled random walk on $\cX^{V_n}$. To define this coupled random walk, for $x\in \cX^{V_n}$ define
$$\bfxi^x: \{0,1\}^{V_n} \times \cX^{V_n} \to \cX^{V_n}$$
by
 \begin{displaymath}
\bfxi^x(\chi, z)_v:= \left\{ \begin{array}{cc}
z_v & \textrm{ if } \chi(v) =1  \\
x_v & \textrm{ if } \chi(v) =0
\end{array}\right.
\end{displaymath}

We will think of $\bfxi^x$ as a random variable taking values in $\cX^{V_n}$. More precisely, we choose a random subset of $V_n$ with each vertex being chosen with probability $s$ (independently). Then we randomize the value of $x$ at each chosen vertex. This produces the new random element $\bfxi^x$.

\begin{lem}\label{lem:contraction4}
Fix notation as in Proposition \ref{prop:contraction2}. Then for any $s\in [0,1]$,
$$\lim_{n\to\infty} \inf\{ \P_s( (\bfxi^x, y) \in \Omega(\cO_2,\s_n) ):~  (x,y) \in \Omega(\cO_1,\s_n)\} = 1.$$
\end{lem}

\begin{proof}
Because $\cO_1$ is hereditary, there is a finite hereditary subset $\cF \subset C(\cX^\G\times \cY^\G)$ and $\d>0$ such that
$$\cO_1=\{\mu' \in \Prob(\cX^\G\times \cY^\G):~ |\mu'(f) - \b^\G\times\nu(f)| < \d \quad \forall f\in \cF\}.$$
Because $\cO_2$ contains the closure of $\cO_1$ there is a $\d'>\d$ such that
$$\cO_2 \supset \{\mu' \in \Prob(\cX^\G\times \cY^\G):~ |\mu'(f) - \b^\G\times\nu(f)| < \d' \quad \forall f\in \cF\}.$$

Fix $(x,y) \in \Omega(\cO_1,\s_n)$ and $f \in \cF$. It suffices to obtain a lower bound on 
$$\P_s \left( \left|P^{\s_n}_{\left(\bfxi^x,y\right)}(f) - \b^\G\times \nu(f)\right| < \d' \right)$$
that tends to 1 as $n\to\infty$ but does not depend on $(x,y)$. 

Let $D \subset \G$ be a finite subset such that $f$ is $D$-local. Let $W_n \subset V_n$ be the set of all vertices $v \in V_n$ such that the map 
$$g \in D\mapsto \s_n(g)^{-1}v$$
is injective. For any $v \in W_n$,
$$\E_s[f(\Pi^{\s_n}_v(\bfxi^x,y))] = \sum_{S \subset D} s^{|S|} (1-s)^{|D\setminus S|}\E_{\b^\G\times \nu}[f| \cB_{\cX^{D\setminus S}}](\Pi^{\s_n}_v(x,y)).$$
To see this, let $\bfS:\{0,1\}^{V_n} \to 2^D$ be the random subset of $D$ defined by 
$$\bfS(\chi)=\{g \in D:~\chi(\s_n(g)^{-1}v)=1 \}.$$
Then for any $g \in D$, $g \in \bfS$ with probability $s$ and these events are jointly independent over $g\in D$. Moreover, conditioned on $\bfS(\chi)=S$, the expected value of $f(\Pi^{\s_n}_v(\bfxi^x,y))$ is $\E_{\b^\G\times \nu}[f| \cB_{\cX^{D\setminus S}}](\Pi^{\s_n}_v(x,y))$.

Because $\Si$ is a sofic approximation, $\lim_{n\to\infty} |W_n|/|V_n|=1$. So
\begin{eqnarray*}
\E_s\left[P^{\s_n}_{\left(\bfxi^x,y\right)}(f) \right] &=& |V_n|^{-1} \sum_{v\in V_n} \E_s[f(\Pi^{\s_n}_v(\bfxi^x,y))] \\
&=&  |V_n|^{-1} \sum_{v\in V_n}  \sum_{S \subset D} s^{|S|} (1-s)^{|D\setminus S|}\E_{\b^\G\times \nu}[f| \cB_{\cX^{D\setminus S}}](\Pi^{\s_n}_v(x,y))  + O\left(|V_n|^{-1}\right)\\
&=& \sum_{S \subset D} s^{|S|} (1-s)^{|D\setminus S|} |V_n|^{-1} \sum_{v\in V_n}  \E_{\b^\G\times \nu}[f| \cB_{\cX^{D\setminus S}}](\Pi^{\s_n}_v(x,y))  + O\left(|V_n|^{-1}\right) \\
&=& \sum_{S \subset D} s^{|S|} (1-s)^{|D\setminus S|} P^{\s_n}_{(x,y)}(\E_{\b^\G\times \nu}[f| \cB_{\cX^{D\setminus S}}])  + O\left(|V_n|^{-1}\right).
\end{eqnarray*}

Because $\cF$ is hereditary, $\E_{\b^\G\times \nu}[f| \cB_{\cX^{D\setminus S}}] \in \cF$. So 
$$\left|P^{\s_n}_{(x,y)}\left(\E_{\b^\G\times \nu}[f| \cB_{\cX^{D\setminus S}}]\right) - \b^\G\times\nu(f)\right| < \d$$
for all $S \subset D$. Thus
 $$\left|\E_s\left[P^{\s_n}_{\left(\bfxi^x,y\right)}(f) \right] - \b^\G\times\nu(f)\right| < \d + O\left(|V_n|^{-1}\right).$$
 
Let $F: \{0,1\}^{V_n}\times \cX^{V_n} \to \R$ be the function
$$F(\chi,z) = P^{\s_n}_{\left(\bfxi^x(\chi,z),y\right)}(f).$$
To finish the proof, it suffices (by Proposition \ref{prop:concentration} and the previous inequality) to prove that $F$ is $C$-Lipschitz for some constant $C>0$ (that does not depend on $n$, $x$ or $y$ but may depend on other parameters). In fact,
 \begin{eqnarray*}
 |F(\chi,z) - F(\chi', z')| &=& |V_n|^{-1} \left| \sum_{v \in V_n}  f(\Pi^{\s_n}_v( \bfxi^x(\chi,z),y)) - f(\Pi^{\s_n}_v( \bfxi^x(\chi',z'),y)) \right| \\
  &\le& |V_n|^{-1}  \sum_{v \in V_n} \left| f(\Pi^{\s_n}_v( \bfxi^x(\chi,z),y)) - f(\Pi^{\s_n}_v( \bfxi^x(\chi',z'),y)) \right| \\
  &\le& |V_n|^{-1}  \sum_{v \in V_n} d_{\cX\times \cY}^{D} \left(   \Pi^{\s_n}_v( \bfxi^x(\chi,z),y), \Pi^{\s_n}_v( \bfxi^x(\chi',z'),y)\right).
  \end{eqnarray*}
 The last inequality above occurs because $f$ is $1$-Lipschitz as a function from $\cX^D\times \cY^D$ to $\R$.
 
 For fixed $v \in V_n$,
 \begin{eqnarray*}
 &&d_{\cX\times \cY}^{D} (   \Pi^{\s_n}_v( \bfxi^x(\chi,z),y), \Pi^{\s_n}_v( \bfxi^x(\chi',z'),y)) \\
 &\le &  |D|^{-1} \sum_{g\in D} d_\cX(z(\s_n(g)^{-1}v), z'(\s_n(g)^{-1}v))  + \diam(\cX)1_{\chi(\s_n(g)^{-1}v) \ne \chi'(\s_n(g)^{-1}v)}.
 \end{eqnarray*}
 Summing over all $v$, we obtain
  \begin{eqnarray*}
  |F(\chi,z) - F(\chi', z')| &\le &  |V_n|^{-1}  \sum_{v \in V_n} d_\cX(z(v), z'(v))  + \diam(\cX)1_{\chi_v\ne \chi'_v} \\
    &=& d_\cX^{V_n}(z,z') + \diam(\cX)d_{\{0,1\}}^{V_n}(\chi,\chi').
  \end{eqnarray*}
  So $F$ is $\max(1,\diam(\cX))$-Lipschitz.

\end{proof}

\begin{proof}[Proof of Proposition \ref{prop:contraction2}]
Fix notation as in the statement of Proposition  \ref{prop:contraction2}. Let $\Leb$ denote Lebesgue measure on the unit interval $[0,1]$. Let $\P = \Leb^{V_n} \times \b^{V_n}$ be the product measure on $[0,1]^{V_n}\times \cX^{V_n}$. Let $\E$ denote expectation with respect to $\P$.

For $x \in \cX^{V_n}$ and $s \in [0,1]$, define
$$\bfzeta(x,s| \cdot, \cdot): [0,1]^{V_n}\times \cX^{V_n}  \to \cX^{V_n}$$
by
\begin{displaymath}
\bfzeta(x,s|\tau,z)_v:= \left\{ \begin{array}{cc}
z_v & \textrm{ if } \tau_v \le s \\
x_v & \textrm{ if } \tau_v > s
\end{array}\right.
\end{displaymath}

The distribution of $\bfzeta(x,s| \cdot,\cdot)$ (with respect to $\Leb^{V_n} \times \b^{V_n}$) is the same as the distribution of $\bfxi^x$ (with respect to $\P_s$). So Lemma \ref{lem:contraction4} applies.

Fix a natural number $k>\diam(\cX)/\k$. Let $(\bftau,\bfz)$ be a random variable with distribution $\Leb^{V_n} \times \b^{V_n}$. Then with high probability (whp) as $n\to\infty$ the following events occur:
\begin{enumerate}
\item For every $0 \le j \le k$,
$$\#\left\{ v\in V_n:~ \bftau_v \in (j/k, j/k + 1/k] \right\} < \k |V_n|/\diam(\cX).$$
\item $\bfzeta\left(x_i,j/k| \cdot,\cdot\right) \in \Omega(\cO_2,\s_n)$ for all $1\le i \le K$ and $0\le j \le k$. 
\item $\bftau_v \ne 0$ for all $v \in V_n$.
\end{enumerate}
This first condition holds whp by the law of large numbers, the second by Lemma \ref{lem:contraction4}, and the last occurs with probability 1. 

So there is some $(\tau,z) \in [0,1]^{V_n}\times \cX^{V_n}$ such that all of the above conditions hold. Set $x^{(j)}_i = \bfzeta(x_i,j/k| \tau,z)$. This first 3 conclusions of Proposition  \ref{prop:contraction2} are immediate. The fourth occurs by definition of $\bfzeta(x_i,j/k| \tau,z)$. The fifth occurs with $M=k$. 
\end{proof}

\begin{lem}\label{lem:entropy}
$b_{d,\Si}(\nu) \le (d+1)h_\Si(\nu)$ for all $d,\Si, \nu$.
\end{lem}

\begin{proof}
Let $\cO_1 \subset \cO_2 \subset \Prob(\cY^\G)$ be open neighborhoods of $\nu$ and $0< \k_1\le \k_2/3$. Let $S_n \subset \Omega(\cO_1,\s_n)$ be a subset whose $\k_1$-neighborhood contains $\Omega(\cO_1,\s_n)$ and 
$$|S_n| = \cov_{\k_1}(\Omega(\cO_1,\s_n), d_\cY^{V_n}).$$
So there exists a map $\Psi: \Omega(\cO_1,\s_n) \to S_n$ such that $d^{V_n}_\cY(\Psi(y),y) < \k_1$ for all $y$. If $y \in \cY^{V_n} \setminus \Omega(\cO_1,\s_n)$, then define $\Psi(y)=y$ so that now we can consider $\Psi$ as a map from $\cY^{V_n}$ to itself. By Lemma \ref{lem:cylinder}, $\Psi_*(z)-z \in B_d(\cO_2, \k_2, \s_n)$ for any $z \in Z_d^L(\cO_1,\k_1,\s_n)$. This uses $\cO_1 \subset \cO_2$ and $3\k_1 \le \k_2$. It follows that 
$$\dim_\Q(H_d^L(\cO_1,\cO_2,\k_1,\k_2,\s_n) \otimes_\Z \Q) \le \#S_n^{d+1} = \cov_{\k_1}(\Omega(\cO_1,\s_n), d_\cY^{V_n})^{d+1}.$$
The lemma now follows from the definitions of $b_{d,\Si}(\nu)$ and $h_\Si(\nu)$.
\end{proof}

\begin{cor}
If $\G \cc (X,\mu)$ has the Weak Pinsker Property then $b_{0,\Si}(\mu)=0$. 
\end{cor}

\begin{proof}
Let $\eps>0$. Then $\G \cc (X,\mu)$ is isomorphic to the direct product of a Bernoulli shift and an action with entropy $<\eps$. By Theorem \ref{thm:bernoulli}, the $0$-dimensional sofic homology of $\G \cc (X,\mu)$ is equivalent to the $0$-dimensional sofic homology of an action with entropy $<\eps$. By Lemma \ref{lem:entropy}, this shows $b_{0,\Si}(\mu)<\eps$.
\end{proof}

\section{An action without the Weak Pinsker Property}\label{sec:markov}

This section proves Theorem \ref{thm:indep0}. Here is an outline:
\begin{itemize}

\item \S \ref{sec:configuration} shows that two models of random graphs, the configuration model and the permutation model, are closely related, allowing the transfer of results about one to the other. This is useful here because results in the literature are generally proven for the configuration model, but it is the permutation model that gives actions of the free group. 




\item \S \ref{sec:moment} is about first moment computations of the numbers of independent sets and pairs of independent sets of a random regular graph. 

\item \S \ref{sec:planted} explains two models of random pairs $(G,I)$ where $G$ is a regular graph and $I \subset V$ is an independent set. These models are called the planted and the uniform model. Usually it is easy to estimate probabilities with respect to the planted model but not with respect to the uniform model. However, sofic entropy is more closely related to the uniform model. Fortunately, there is an inequality relating the two models.

\item \S \ref{sec:clusters0} has the main technical result bounding the size of clusters of independent sets in a random regular graph.

\item \S \ref{sec:variational} provides a general result for obtaining an invariant measure $\mu$ on $\cX^\G$ whose model spaces $\Omega(\cO,\s_n)$ have large intersections with fixed subsets $\sW_n \subset \cX^{V_n}$. This is applied later with $\sW_n$ equal to the set of ``good'' independent subsets to obtain the invariant measure in Theorem \ref{thm:indep0}.

\item \S \ref{sec:final} finishes the proof of Theorem \ref{thm:indep0}.

\end{itemize}


\subsection{The configuration model and the permutation model}\label{sec:configuration}

The proof of Theorem \ref{thm:indep} is made simpler by borrowing results about independent sets on the configuration model of random regular graphs and transferring them to the permutation model. This section explains the two models and two key theorems linking them together. 

\begin{defn}[The configuration model]
Let $d \ge 3, n \ge 1$ be integers such that $dn$ is even. Let $\bfpi$ be a uniformly random perfect matching on $[n]\times [d]$. Let $G_{\rm{conf}}(\bfpi)$ be the random multi-graph with vertex set $[n]$ such that the number of edges from $i$ to $j$ equals the number of edges between $\{i\} \times [d]$ and $\{j\} \times [d]$ in the matching $\bfpi$. This is called the {\bf configuration model} \cite{MR1864966, MR1782847}. It gives a random $d$-regular multi-graph on $n$ vertices. Let $\P^{\rm{conf}}_{d,n}$ denote the law of $G_{\rm{conf}}(\bfpi)$ and let $\E^{\rm{conf}}_{d,n}$ be its expectation operator. So $\P^{\rm{conf}}_{d,n}$ is a probability measure on $\Graphs(d,n)$,  which is   the set of all $d$-regular multi-graphs on $[n]$.

\end{defn}

\begin{defn}[The permutation model]
Let $\G=\F_r = \langle a_1,\ldots, a_r \rangle$ be the rank $r$ free group. Let $\P^{\rm{perm}}_{r,n}$ be the uniform probability measure on the set $\Hom(\F_r,\sym(n))$ of homomorphisms from $\F_r$ to $\sym(n)$. Also let $\E^{\rm{perm}}_{r,n}$ be its expectation operator. For $\s \in \Hom(\F_r,\sym(n))$, let $G(\s)$ be the multi-graph with vertex set $[n]$ and edges $\{v, \s(a_i)v\}$ (over $v\in [n], 1\le i \le r$). If $\bfsig$ is random with law $\P^{\rm{perm}}_{r,n}$ then $G(\bfsig)$ is called the {\bf permutation model}. The law of $G(\bfsig)$ is a probability measure on $\Graphs(2r,n)$, which by abuse of notation, we will also denote by $\P^{\rm{perm}}_{r,n}$. 

\end{defn}


One of the main results of \cite{MR1909503} is:
\begin{thm}\label{thm:contiguous}
 Let $A_n \subset \Graphs(d,n)$ be any sequence of subsets. Suppose $d\ge 4$ is even and let $2r=d$. Then $\lim_n \P^{\rm{conf}}_{d,n}(A_n) = 1$ if and only if $\lim_n \P^{\rm{perm}}_{r,n}(A_n)=1$. Equivalently, the permutation and configuration models model are {\bf contiguous}.
\end{thm}

\begin{cor}\label{thm:sofic-random}
Let $\bfsig_n$ be a random homomorphism from $\F_r$ to $\sym(n)$ with law $\P^{\rm{perm}}_{r,n}$. Then for every finite $D \subset \F_r$ and $\d>0$,
$$\lim_{n\to\infty} \P^{\rm{perm}}_{r,n}( \bfsig_n \textrm{ is $(D,\d)$-sofic} ) = 1.$$

\end{cor}

\begin{proof}
It suffices to show that for any nontrivial $w \in \F_r$ the number of $v \in \{1,\ldots,n\}$ such that $\bfsig_n(w)v=v$ is $o(n)$ with high probability as $n\to\infty$. To phrase this a different way, it suffices to prove that for any $L>0$ the number of simple closed cycles of length $\le L$ in $G(\bfsig_n)$ is $o(n)$ with high probability as $n\to\infty$. This statement is proven in \cite{MR595929} for the configuration model. Since the two models are contiguous, it also holds for the permutation model. 
\end{proof}

\subsubsection{Expectations}
The next result shows that the first moment method applied to counting vertex-labelings of either the configuration or permutation model results in the same calculation up to subexponential factors. To explain further we need some notation. 

\begin{defn}[Admissible pairs]
Let $\cX$ be a finite set. A pair of vectors $(\pi^{\rm{vert}}, \pi^{\rm{\rm{edge}}}) \in \Prob(\cX) \times \Prob(\cX\times \cX) \subset \R^{\cX} \times \R^{\cX \times \cX}$ is {\bf admissible} if both $\pi^{\rm{vert}}$ and $\pi^{\rm{edge}}$ are probability vectors, $\pi^{\rm{edge}}$ is symmetric in the sense that $\pi^{\rm{edge}}(p,q)=\pi^{\rm{edge}}(q,p)$ for all $p,q$ and both marginals of $\pi^{\rm{edge}}$ equal $\pi^{\rm{vert}}$. The latter condition means for every $p\in \cX$,
$$\pi^{\rm{vert}}(p) = \sum_{q\in \cX} \pi^{\rm{edge}}(p,q) = \sum_{q\in \cX} \pi^{\rm{edge}}(q,p).$$
Let $\cA(\cX) \subset \R^{\cX} \times \R^{\cX \times \cX} $ be the compact space of all admissible pairs. 
\end{defn}

\begin{example}
Suppose $G=(V,E)$ is a finite $	d$-regular multi-graph and $x:V \to \cX$ a map. Choose a vertex $\bfv$ and a directed edge $\bfe$ independently and uniformly at random. Then the distributions of $x(\bfv)$ and $x(\bfe)$ form a pair of admissible vectors. Let $(\pi^{\rm{vert}}_x, \pi^{\rm{edge}}_x) \in \cA(\cX)$ denote this pair of distributions. 
\end{example}

\begin{defn}[$\Omega(\pi^{\rm{vert}}, \pi^{\rm{edge}};G)$]
Given a finite graph $G=(V,E)$ and an admissible pair of vectors $(\pi^{\rm{vert}}, \pi^{\rm{edge}}) \in \cA(\cX)$, let $\Omega(\pi^{\rm{vert}}, \pi^{\rm{edge}};G)$ be the set of all vertex-labelings $x:V \to \cX$ such that
 $$\pi^{\rm{vert}} = \pi^{\rm{vert}}_x, \quad \pi^{\rm{edge}} =\pi^{\rm{edge}}_x.$$
Write $\#\Omega( \pi^{\rm{vert}}, \pi^{\rm{edge}})$ for the random variable 
$$\bfG \mapsto \#\Omega( \pi^{\rm{vert}}, \pi^{\rm{edge}}; \bfG)$$
where $\bfG$ is either a random sample of the configuration or the permutation model, depending on context.
\end{defn}

\begin{defn}
For $x>0$, let $\eta(x)=-x\log(x)$. Extend this by continuity so that $\eta(0)=0$. Given a vector $\vec{p}=(p_1,\ldots, p_k)$ of nonnegative real numbers, let 
$$H(\vec{p})=H(p_1,\ldots, p_k) = \sum_{i=1}^k \eta(p_i)$$
be the {\bf Shannon entropy} of $\vec{p}$. For example, it is well-known that if $\vec{p}$ is a probability vector then the associated multinomial coefficients satisfy
$$ \lim_{n\to\infty} n^{-1} \log {n \choose m_{n,1},\ldots, m_{n,k}} = H(p_1,\ldots, p_k)$$
where $m_{n,1},\ldots, m_{n,k}$ are any choice of nonnegative integers satisfying $m_{n,1}+\cdots + m_{n,k} = n$ and $\lim_{n\to\infty} m_{n,i}/n=p_i $ for all $i$.
\end{defn}

\begin{lem}\label{lem:F}
Let $(\pi^{\rm{vert}}, \pi^{\rm{edge}}) \in \cA(\cX)$ be an admissible pair of vectors.  If $d\ge 2, n\ge 1$ are integers, $dn$ is even, $\pi^{\rm{vert}}$ takes values in $\frac{1}{n}\Z$ and $\pi^{\rm{edge}}$ takes values in $\frac{2}{dn}\Z$ then 
$$n^{-1} \log \E^{\rm{conf}}_{d,n}[\#\Omega(\pi^{\rm{vert}},\pi^{\rm{edge}})] = (d/2)H(\pi^{\rm{edge}}) - (d-1) H(\pi^{\rm{vert}}) + o_n(1)$$
and for any integer $r\ge 1$,
$$n^{-1} \log \E^{\rm{perm}}_{r,n}[\#\Omega(\pi^{\rm{vert}},\pi^{\rm{edge}})] = rH(\pi^{\rm{edge}}) - (2r-1) H(\pi^{\rm{vert}}) + o_n(1).$$
\end{lem}

\begin{proof}
The case of the permutation model is handled in \cite{MR2653969}. It seems likely that the configuration model result has been known for some time because it is essentially the ``first moment method'' which has been a standard tool in this area of probabilistic combinatorics for decades. However, I have been unable to find a suitable reference. 

The proof will reduce to \cite[Lemma 4.1]{MR3854040} once we verify that $\E^{\rm{conf}}_{d,n}[\#\Omega(\pi^{\rm{vert}},\pi^{\rm{edge}})]$ is positive. 

Because $\pi^{\rm{vert}}$ takes values in $\frac{1}{n}\Z$, there exists a map $x:[n] \to \cX$ such that $\pi^{\rm{vert}}_x = \pi^{\rm{vert}}$. Let $\tx:[n] \times [d] \to \cX$ be the lift defined by $\tx(v,i) = x(v)$. Because $\pi^{\rm{edge}}$ takes values in $\frac{2}{dn}\Z$, there exists some perfect matching $\mu$ on $[n] \times [d]$ such that the distribution on pairs of vertex-labels from $x$ assigned to the endpoints of a uniformly random oriented edge of $\mu$ is $\pi^{\rm{edge}}$. Here are some details to justify this claim. First $\pi^{\rm{vert}} \in (\frac{2}{dn}\Z)^{\cX}$ because $\pi^{\rm{edge}}$ takes values in $\frac{2}{dn}\Z$ and $\pi^{\rm{vert}}$ is the first marginal of $\pi^{\rm{edge}}$. So there exists a partition
$$[n]\times [d] = P \sqcup Q$$
such that
$$|P \cap \tx^{-1}(p)| = |Q \cap \tx^{-1}(p)| = \pi^{\rm{vert}}(p)nd/2$$
for all $p \in \cX$. For every $p \in \cX$, choose partitions
$$P \cap \tx^{-1}(p) = \sqcup_{q \in \cX} P_{p,q}, \quad Q \cap \tx^{-1}(p) = \sqcup_{q \in \cX} Q_{q,p}$$
such that
$$\#P_{p,q} = \#Q_{q,p}= \pi^{\rm{edge}}(p,q)nd/2.$$
Finally let $\mu$ be any perfect matching of $[n]\times [d]$ that restricts to a perfect matching from $P_{p,q}$ to $Q_{q,p}$ for all $p,q \in \cX$. This matching satisfies the claim.

By \cite[Lemma 4.1]{MR3854040}, the probability that a uniformly random perfect matching of $[n]\times [d]$ has edge distribution $\pi^{\rm{edge}}$ is
$$e^{ (dn/2) H(\pi^{\rm{edge}}) } e^{- dnH(\pi^{\rm{vert}}) }$$
(up to a multiplicative factor that is subexponential in $n$). To see this, replace $n$ in \cite[Lemma 4.1]{MR3854040} with $dn$, $\mu$ with $\pi^{\rm{vert}}$ and $\nu$ with $\pi^{\rm{edge}}$.  Since the number of ways to choose $x$ is $\exp( n H(\pi^{\rm{vert}}))$ (up to a multiplicative factor that is subexponential in $n$), this implies the result.

\end{proof}

\begin{thm}\label{thm:F}
Fix an integer $d \ge 2$. For $n \in \N$ with $dn$ even, suppose $\cK_n \subset \cA(\cX)$ satisfies the following: if $(\pi^{\rm{vert}}, \pi^{\rm{edge}}) \in \cK_n$ then $\pi^{\rm{vert}}$ takes values in $\frac{1}{n}\Z$ and $\pi^{\rm{edge}}$ takes values in $\frac{2}{dn}\Z$. Suppose that $\cK_n$ converges to a closed subspace $\cM$ in the Hausdorff topology (on the space of all closed subsets of $\cA(\cX)$) as $n\to\infty$. Then
\begin{eqnarray*}
 && \lim_{n\to\infty} n^{-1} \log \E^{\rm{conf}}_{d,n}[\#\{x \in \cX^n:~ (\pi_x^{\rm{vert}},\pi_x^{\rm{edge}}) \in \cK_n\}]  \\
  &=& \max \{(d/2)H(\pi^{\rm{edge}}) - (d-1) H(\pi^{\rm{vert}}):~ (\pi^{\rm{vert}}, \pi^{\rm{edge}}) \in \cM\}.
\end{eqnarray*}
Moreover, if $d=2r$ for some integer $r\ge 1$ then the quantity above also equals
\begin{eqnarray*}
 && \lim_{n\to\infty} n^{-1} \log \E^{\rm{perm}}_{r,n}[\#\{x \in \cX^n:~ (\pi_x^{\rm{vert}},\pi_x^{\rm{edge}}) \in \cK_n\}].
\end{eqnarray*}
\end{thm}

\begin{proof}
This follows from the previous Lemma since the function 
$$(\pi^{\rm{vert}}, \pi^{\rm{edge}}) \mapsto (d/2)H(\pi^{\rm{edge}}) - (d-1) H(\pi^{\rm{vert}})$$
is uniformly continuous and there are at most a polynomial (in $n$)  number of admissible pairs $(\pi^{\rm{vert}}, \pi^{\rm{edge}}) \in \cK_n$ such that $\pi^{\rm{edge}}$ takes values in $\frac{2}{dn}\Z$. In fact this number is bounded by $(dn + 1)^{|\cX|^2}$.
\end{proof}

\subsection{First moment computations}\label{sec:moment}

Let $\Ga=\F_r=\langle a_1,\ldots, a_r\rangle$ denote the free group of rank $r \ge 2$. For $s,s',t \in [0,1]$ and $\s \in \Hom(\F_r,\sym(n))$, let  $\cI_{s,s',t}(\s)$ be the set of all of pairs $(W,W')$ of independent subsets of $G(\s)$ satisfying
\begin{eqnarray}\label{close}
\#W = \lfloor sn\rfloor, \quad \#W' = \lfloor s'n\rfloor, \quad \#(W \cap W') = \lfloor tn\rfloor.
\end{eqnarray}
In the arguments to follow,  $\#\cI_{s,s',t}$ is regarded as a random variable with respect to the permutation model. Let $f(r,s,s',t)$ denote the upper exponential growth rate of the expected value of $\#\cI_{s,s',t}$. 
$$f(r,s,s',t) :=  \limsup_{n\to\infty}  n^{-1} \log \E^{\rm{perm}}_{r,n}[\#\cI_{s,s',t}].$$
\begin{thm}\label{thm:continuity}
The function $f$ is uniformly continuous. Moreover, the limit exists so that
$$f(r,s,s',t) :=   \lim_{n\to\infty}  n^{-1} \log \E^{\rm{perm}}_{r,n}[\#\cI_{s,s',t}].$$
\end{thm}

\begin{proof}
Let $\cX=\{(0,0),(0,1),(1,0),(1,1)\}$. Let $\cK_n$ be the set of all admissible pairs $(\pi^{\rm{vert}},\pi^{\rm{edge}})$ such that $\pi^{\rm{vert}}$ takes values in $\frac{1}{n}\Z$, $\pi^{\rm{edge}}$ takes values in $\frac{2}{dn}\Z$ ($d=2r$) and these linear equations are satisfied:
$$\pi^{\rm{vert}}(1,0)+\pi^{\rm{vert}}(1,1)=\lfloor sn\rfloor /n$$
$$\pi^{\rm{vert}}(0,1) +\pi^{\rm{vert}}(1,1)=\lfloor s'n\rfloor /n$$
$$\pi^{\rm{vert}}(1,1) = \lfloor tn\rfloor /n$$
$$\pi^{\rm{edge}}( (i_1,j_1), (i_2,j_2) ) = 0 \textrm{ if either } (i_1=i_2=1) \textrm{ or } (j_1=j_2=1).$$
A pair $(W,W')$ of independent subsets of $G(\s)$ satisfies (\ref{close}) if and only if $(\pi_x^{\rm{vert}},\pi_x^{\rm{edge}}) \in \cK_n$ where  
$x=(1_W,1_{W'})$. 

Let $\cM(r,s,s',t)$ be the set of all admissible pairs $(\pi^{\rm{vert}},\pi^{\rm{edge}})$  satisfying these linear equations:
$$\pi^{\rm{vert}}(1,0)+\pi^{\rm{vert}}(1,1)=s$$
$$\pi^{\rm{vert}}(0,1) +\pi^{\rm{vert}}(1,1)=s'$$
$$\pi^{\rm{vert}}(1,1) = t$$
$$\pi^{\rm{edge}}( (i_1,j_1), (i_2,j_2) ) = 0 \textrm{ if either } (i_1=i_2=1) \textrm{ or } (j_1=j_2=1).$$
Because $\cK_n$ converges to $\cM(r,s,s',t)$ in the Hausdorff topology,  Theorem \ref{thm:F} implies
\begin{eqnarray}
f(r,s,s',t) &=& \max_{(\pi^{\rm{vert}},\pi^{\rm{edge}}) \in \cM(r,s,s',t)}  -(2r-1)H(\pi^{\rm{vert}}) + r H(\pi^{\rm{edge}}) \label{E:max}\\
&=&   \lim_{n\to\infty}  n^{-1} \log \E^{\rm{perm}}_{r,n}[\#\cI_{s,s',t}].
\end{eqnarray}

Continuity of $f$ now follows from continuity of
$$(\pi^{\rm{vert}},\pi^{\rm{edge}}) \mapsto -(2r-1)H(\pi^{\rm{vert}}) + r H(\pi^{\rm{edge}})$$  
and 
$$(r,s,s',t) \mapsto \cM(r,s,s',t)$$
  where the latter is with respect to the Hausdorff topology on closed subsets of the space of admissible vector pairs. Uniform continuity follows from continuity and compactness.
\end{proof}


We are most interested in the special case in which $s'$ is close to $s$. In this case we will use the continuity result to reduce the study of this function to the special case $s=s'$. To simplify notation, let $f(r,s,t):=f(r,s,s,t)$. 

For $\s \in \Hom(\F_r,\sym(n))$, let $\cI_s(\s)$ be the collection of independent sets $W$ of the graph $G(\s)$ such that  $\#W= \lfloor sn \rfloor$. Also let 
$$f(r,s):=  \limsup_{n\to\infty}  n^{-1} \log \E^{\rm{perm}}_{r,n}[\#\cI_{s}]$$
be the upper exponential growth rate of the expected number of independent sets of cardinality close to $sn$.


\begin{thm}\label{thm:f}
Fix $0\le \bt \le \bs \le 1$. If $t:=t(r):=\bt \frac{\log(2r)}{r}$ and $s  :=s(r):= \bs \frac{\log(2r)}{r}$ then 
\begin{eqnarray*}
f(r,s)&=& \eta(s) -rs^2 + O(\log(r)/r)\\
f(r,s,t) &=& \eta(t) +2 \eta(s-t) + r[t^2-2s^2] +O(\log(r)/r)
\end{eqnarray*}
where the error term implicit in the big $O(\cdot)$ notation does not depend on $\bs,\bt$. 
\end{thm}

\begin{proof}
We keep notation as in the previous proof.  Observe that $f(r,s,s)=f(r,s,s,s)$ counts pairs of identical independent sets. So $f(r,s,s)=f(r,s)$. So it suffices to obtain the estimate for $f(r,s,t)$. The calculation is similar to one in Gamarnik-Sudan \cite{MR3359490}.

By (\ref{E:max}),
$$f(r,s,t) =  \max_{(\pi^{\rm{vert}},\pi^{\rm{edge}}) \in \cM(r,s,s,t)}  -(2r-1)H(\pi^{\rm{vert}}) + r H(\pi^{\rm{edge}}).$$
Let $\pi=(\pi^{\rm{vert}},\pi^{\rm{edge}}) \in \cM(r,s,s,t)$. We claim:
\begin{eqnarray*}
\pi^{\rm{vert}}(0,0) = 1-2s+t &\quad & \pi^{\rm{vert}}(0,1) = s-t \\
 \pi^{\rm{vert}}(1,0) = s-t &\quad & \pi^{\rm{vert}}(1,1) = t.
 \end{eqnarray*}
By definition, $\pi^{\rm{vert}}(1,1) = t$. Since  $\pi^{\rm{vert}}(1,0) + \pi^{\rm{vert}}(1,1) = s$ this determines the value of $\pi^{\rm{vert}}(1,0)$. The other two values hold by symmetry and the fact that $\pi^{\rm{vert}}$ is a probability measure. 
 So
 $$H(\pi^{\rm{vert}}) = \eta(t) + 2\eta(s-t) + \eta(1-2s+t).$$
 
 Because Shannon entropy $H(\cdot)$ is strictly concave and $\cM(r,s,s,t)$ is convex, it follows that there is a unique $\pi_0 \in \cM(r,s,s,t)$ such that
$$f(r,s,t) =  -(2r-1)H(\pi_0^{\rm{vert}}) + r H(\pi_0^{\rm{edge}}).$$
The space $\cX \times \cX$ is invariant under these two symmetries:
$$((a,b),(c,d)) \mapsto ((c,d),(a,b)),$$
$$((a,b),(c,d)) \mapsto ((b,a),(d,c)).$$
These generate a group of order 4. This induces a group of symmetries on $\Prob(\cX\times \cX)$. Moreover, if a measure $\nu \in \Prob(\cX\times \cX)$ has both marginals equal to $\pi^{\rm{vert}}$, then each of its images under this group also has marginals equal to $\pi^{\rm{vert}}$ (for any $\pi \in \cM(r,s,s,t)$). Moreover, the Shannon entropy is preserved under the action of this group. Because $\pi_0^{\rm{edge}} \in \Prob(\cX \times \cX)$ is unique, it is necessarily invariant under the action of this group. So if $x_0 =  \pi_0^{\rm{edge}}( (0,0), (0,1))$ then 
$$x_0=  \pi_0^{\rm{edge}}( (0,0), (1,0)) =  \pi_0^{\rm{edge}}(  (0,1), (0,0))= \pi_0^{\rm{edge}}( (1,0), (0,0)).$$
 Since both marginals of $\pi^{\rm{edge}}$ equal $\pi^{\rm{vert}}$, this implies
 \begin{eqnarray*}
&&\pi^{\rm{edge}}_0((0,0), (0,0)) = 1-2s-2x_0, \quad  \pi^{\rm{edge}}_0((0,0), (1,1)) =\pi_0^{\rm{edge}}((1,1), (0,0)) =t  \\
 &&\pi_0^{\rm{edge}}((1,0),(0,1))  =\pi_0^{\rm{edge}}((0,1),(1,0)) = s-t-x_0.
 \end{eqnarray*}
So $H(\pi_0^{\rm{edge}})= G(x_0)$ where
$$G(x)=2\eta(t) + 4\eta(x) + 2\eta(s-t-x) + \eta(1-2s-2x).$$
Use $\frac{d}{dx}\eta(x) = -1 -\log(x)$, to obtain:
\begin{eqnarray*}
 \frac{\partial G(x)}{\partial x} &=& \frac{\partial}{\partial x}\left( 4\eta(x) + 2\eta(s-t-x) + \eta(1-2s-2x) \right) \\
&=& 2 \log \left( \frac{(s-t-x)(1-2s-2x)}{x^2} \right).
\end{eqnarray*}
Since $\pi_0$ is maximizing, $\frac{\partial G(x)}{\partial x}(x_0)=0$. So $(s-t-x_0)(1-2s-2x_0) = x_0^2$. 

To simplify notation, we let $x=x_0$ from now on. By collecting terms with a factor of $x$, we obtain
\begin{eqnarray*}
&&H(\pi_0^{\rm{edge}}) = 2\eta(t) + 4\eta(x) + 2\eta(s-t-x) + \eta(1-2s-2x) \\
&=& 2\eta(t) -4x\log(x) -2(s-t-x)\log(s-t-x) - (1-2s-2x)\log(1-2s-2x) \\
&=& 2\eta(t) +2x \log \left( \frac{(s-t-x)(1-2s-2x)}{x^2} \right) -2(s-t)\log(s-t-x) - (1-2s)\log(1-2s-2x) \\
&=& 2\eta(t) -2(s-t)\log(s-t-x) - (1-2s)\log(1-2s-2x).
\end{eqnarray*}
Use the above and collect terms with a factor of $r$ to obtain
\begin{eqnarray*}
f(r,s,t) &=& -(2r-1)[\eta(t) + 2\eta(s-t) + \eta(1-2s+t)] \\
&&+ r[2\eta(t) -2(s-t)\log(s-t-x) - (1-2s)\log(1-2s-2x)] \\
&=& \eta(t) + 2\eta(s-t) - (2r-1)\eta(1-2s+t)\\
&& + r[-4\eta(s-t) -2(s-t)\log(s-t-x)  - (1-2s)\log(1-2s-2x)] 
\end{eqnarray*}

Next, we estimate $x$. Since $(s-t-x)(1-2s-2x) = x^2$, 
\begin{eqnarray*}
x^2 -(1-2t)x + (s-t)(1-2s) = 0.
\end{eqnarray*}
Let $\Delta=s-t$. Since $x \le s-t$,
\begin{eqnarray*}
2x &=& 1-2t - \sqrt{(1-2t)^2 - 4(s-t)(1-2s)}\\
&=& 1-2s +2\Delta - \sqrt{ (1-2s +2\Delta)^2 - 4\Delta (1-2s)} \\
&=& 1-2s +2\Delta - \sqrt{ (1-2s)^2 +4\Delta^2}.
\end{eqnarray*}
By Taylor series expansion, if $C > 0$ is a constant then $\sqrt{C^2+\eps}=C + \frac{\eps}{2C}  +  O(\eps^2)$. Use this with $C=1-2s$, $\eps=4\Delta^2$ to obtain
\begin{eqnarray}\label{xxx1}
2x &=& 2\Delta - \frac{2\Delta^2}{(1-2s)} + O(\Delta^4).
\end{eqnarray}

Thus
\begin{eqnarray*}
\log(s-t-x) &=& \log\left( \frac{(s-t)^2}{1-2s} + O\left( \Delta^4 \right) \right) \\
&=& 2\log(s-t) +2s + O((s-t)^2)  =  2\log(s-t) +2s + O\left( s^2 \right).
\end{eqnarray*}
Thus
\begin{eqnarray*}
-4\eta(s-t) -2(s-t)\log(s-t-x) &=& 2(s-t)(2\log(s-t) - \log(s-t-x)) =  -4(s-t)s+ O(s^3).
\end{eqnarray*}
 Plugging this back in to the equation for $f$ gives
 \begin{eqnarray*}
f(r,s,t) &=& \eta(t) + 2\eta(s-t)  - (2r-1)\eta(1-2s+t)\\
&& + r[-4(s-t)s   - (1-2s)\log(1-2s-2x)] + O(rs^3).
\end{eqnarray*}
We use the Taylor series estimates
\begin{eqnarray*}
\eta(1-2s+t) &=& 2s-t - (2s-t)^2/2 + O(s^3)\\
\log(1-2s-2x) &=&  -(2s+2x ) - (2s+2x )^2/2  + O(s^3)  = -2[ s+x + (s+x)^2]  + O(s^3) \\
&=&  -2(1+ s+x)(s+x)  + O(s^3)
\end{eqnarray*} 
to obtain
 \begin{eqnarray*}
f(r,s,t) &=& \eta(t) + 2\eta(s-t) -(2r-1)[2s-t - (2s-t)^2/2] \\
&& -4rs(s-t) +2r(1-2s)(1+ s+x)(s+x) + O(rs^3).
\end{eqnarray*}
If $r$ is large enough then $\log^2(r)>r$ which implies $rs^3 \le s$. So we can replace the $O(rs^3)$ with $O(s)$. This allows us to replace the coefficient $(2r-1)$ on the bracketted term with $2r$. We can also simplify the last term by
$$2r(1-2s)(1+ s+x)(s+x) = 2r(1-s+x)(s+x) + O(s)$$
to obtain
\begin{eqnarray*}
f(r,s,t) &=& \eta(t) + 2\eta(s-t) -4rs(s-t) -2r[2s-t - (2s-t)^2/2] +2r(1-s+x)(s+x)  + O(s).
\end{eqnarray*}
By factoring out a $2r$ we obtain
\begin{eqnarray*}
f(r,s,t) &=& \eta(t) + 2\eta(s-t)\\
&& +2r[-2s(s-t) -2s+t + (2s-t)^2/2 +(1-s+x)(s+x)]  + O(s).
\end{eqnarray*}
After multiplying out and collecting terms, the quantity in brackets simplifies to
$$[-s-s^2+t+t^2/2 +x + x^2].$$
By (\ref{xxx1}),
this simplifies to $[-s^2 + t^2/2   + O(s^3)].$ So
\begin{eqnarray*}
f(r,s,t) &=& \eta(t) + 2\eta(s-t) +r[t^2-2s^2]  + O(s).
\end{eqnarray*}
Since $s = O(\log(r)/r)$, this implies the theorem. 
\end{proof}

\begin{cor}\label{cor:maximization}
Keep notation as in Theorem \ref{thm:f}. In addition, assume $2/3 < \bs < 1$. Then
$$\max \{f(r,s,t) - f(r,s):~ s/2\le t \le s\} = f(r,s) -\bs(1-\bs)\frac{\log^2(r)}{r} + O\left( \frac{\log(r)\log\log(r)}{r} \right).$$
\end{cor}

\begin{proof}
By Theorem \ref{thm:f}, 
\begin{eqnarray*}
f(r,s,t)-f(r,s) &=& f(r,s) + \eta(t) + 2\eta(s-t) - 2\eta(s) + rt^2 + O\left( \frac{\log(r)}{r} \right).
\end{eqnarray*}
Since $s = \bs \frac{\log(2r)}{r}$, 
$$\eta(s) = -s\log(s) = \bs\frac{\log^2(r)}{r}  + O\left(  \frac{\log(r)\log\log(r)}{r} \right).$$
Similar estimates hold for $\eta(t)$ and $\eta(s-t)$ and the constant implicit in the $O(\cdot)$ notation is uniform over $\bs, \bt$. Therefore,
$$\eta(t) + 2\eta(s-t) - 2\eta(s) = -\bt\frac{\log^2(r)}{r}  + O\left(  \frac{\log(r)\log\log(r)}{r} \right).$$
Since $rt^2 = \bt^2 \frac{\log^2(r)}{r}$, this implies
\begin{eqnarray*}
f(r,s,t)-f(r,s) &=& f(r,s)  - (\bt-\bt^2) \frac{\log^2(r)}{r} + O\left( \frac{\log(r)\log\log(r)}{r} \right).
\end{eqnarray*}
The minimum value of $ (\bt-\bt^2)$ for $\bt \in [\bs/2, \bs]$ is attained when $\bt=\bs$. This is because $x\mapsto x-x^2$ is concave, so the minimum is achieved at either $\bs/2$ or $\bs$. But since  $\bs \in (2/3,1)$, $(\bs-\bs^2) < (\bs/2 - \bs^2/4)$. Substituting $\bs$ for $\bt$ finishes the corollary.
\end{proof}

\subsection{The planted model and the uniform model}\label{sec:planted}

There are two related models of random independent sets on random regular graphs that we need to consider to prove Theorem \ref{thm:indep}. To explain, recall that for $\s \in \Hom(\F_r,\sym(n))$, $\cI_s(\s)$ is the collection of independent sets $W$ of the graph $G(\s)$ such that  $\#W= \lfloor s \#V \rfloor$. The {\bf uniform model} is  the probability measure $\P^{\rm{unif}}_{r,s,n}$ on $\Hom(\F_r,\sym(n))\times 2^n$ defined by:
$$\P^{\rm{unif}}_{r,s,n}(\s,W) = \frac{1}{\#\Hom(\F_r,\sym(n)) \times \#\cI_s(\s) }$$
if $W \in \cI_s(\s)$ and  $\P^{\rm{unif}}_{r,s,n}(\s,W) = 0$ otherwise. A random sample $(\bfsig, \bfW)$ with law $\P^{\rm{unif}}_{r,s,n}$ can be obtained by first choosing $\bfsig \in  \Hom(\F_r,\sym(n))$ uniformly at random and then choosing $\bfW \in \cI_s(\bfsig)$ uniformly at random.

The {\bf planted model} is the probability measure $\P^{\rm{plant}}_{r,s,n}$ on $\Hom(\F_r,\sym(n))\times 2^n$ that is uniformly distributed on pairs $(\s,W)$ such that $W \in \cI_s(\s)$. Thus
$$\P^{\rm{plant}}_{r,s,n}(\s,W) = \frac{1}{\#\Hom(\F_r,\sym(n)) \times \E^{\rm{perm}}_{r,n}[\#\cI_s(\s)] }$$
if $W \in \cI_s(\s)$ and $0$ otherwise. Let $\E^{\rm{unif}}_{r,s,n}$ and $\E^{\rm{plant}}_{r,s,n}$ be the corresponding expectation operators.

It is relatively easy to compute probabilities with respect to the planted model instead of the uniform model. However, to prove Theorem \ref{thm:indep} we need to work with the uniform model. The next result forms a bridge between the two models. 



\begin{thm}\label{thm:planted}
Fix $\bs \in (0,1)$ and set $s=\bs \frac{\log(2r)}{r}$. Let $R \ge r$ be an integer satisfying
$$\liminf_n \P^{\rm{perm}}_{R,n}[ \#\cI_s >0] = 1.$$
Let
$$X_{R,r,s,n} = \left\{\s \in \Hom(\F_r,\sym(n)):~ \# \cI_s(\s) \ge \frac{\E^{\rm{perm}}_{r,n}[ \# \cI_s] }{2\E^{\rm{perm}}_{R,n}[ \# \cI_s]  } \right\}.$$
Then
$$\liminf_{n\to\infty} \P^{\rm{perm}}_{r,n}\left( X_{R,r,s,n} \right) =1.$$
Moreover, if $A \subset X_{R,r,s,n} \times 2^n$ then
$$\P^{\rm{unif}}_{r,s,n}(A) \le  \P^{\rm{plant}}_{r,s,n}(A) \times 2\E^{\rm{perm}}_{R,n}[ \# \cI_s].$$
\end{thm}

\begin{remark}
The proof below is modeled after an analogous result for sparse Erd\"os-Renyi graphs obtained in \cite{MR3385742}. 
\end{remark}


\begin{proof}[Proof of Theorem \ref{thm:planted}]

Let $\F_R=\langle a_1,\ldots, a_R\rangle$ and $\F_r =\langle a_1,\ldots, a_r\rangle$. Thus $\F_r$ is a subgroup of $\F_R$. Let $\bfsig$ be a sample of $\P^{\rm{perm}}_{R,n}$ (so $\bfsig$ is a uniformly random homomorphism from $\F_R$ to $\sym(n)$). Let $\bfsig \resto \F_r$ denote the restriction of $\bfsig$ to $\F_r$.


Any independent set of $G(\bfsig)$ is automatically an independent set of $G(\bfsig \resto \F_r)$ since the latter is a subgraph of the former. If $W \subset [n]$ is an independent set for $G(\bfsig \resto \F_r)$ then the probability that $W$ is also an independent set for $G(\bfsig)$ depends only on $|W|, (R-r)$ and $n$. It is easy to derive an exact expression for this probability, but we do not need it.




Let $Y_n$ be the set of all $\s \in \Hom(\F_R, \sym(n))$ such that the restriction $\s \resto \F_r \notin X_{R,r,s,n}$. Since $\P^{\rm{perm}}_{R,n}(Y_n) = 1-\P^{\rm{perm}}_{r,n}(X_{R,r,s,n})$, it suffices to prove $\P^{\rm{perm}}_{R,n}(Y_n) \to 0$ as $n\to\infty$.  Since the probability that an independent subset $W$ of $G(\bfsig \resto \F_r)$ is an independent subset of $G(\bfsig)$ does not depend on $\bfsig \resto \F_r$,
\begin{eqnarray*}
\E^{\rm{perm}}_{R,n}[ \# \cI_s | Y_n] &=& \E^{\rm{perm}}_{R,n}\left[ \#\cI_s(\bfsig \resto\F_r) \Big| Y_n \right] \frac{\E^{\rm{perm}}_{R,n}[ \# \cI_s] }{\E^{\rm{perm}}_{r,n}[ \# \cI_s] }.
\end{eqnarray*}
By definition of $Y_n$ and $X_{R,r,s,n}$, 
$$\E^{\rm{perm}}_{R,n}\left[ \#\cI_s(\bfsig \resto\F_r) \Big| Y_n \right]< \frac{\E^{\rm{perm}}_{r,n}[ \# \cI_s] }{2\E^{\rm{perm}}_{R,n}[ \# \cI_s]  }.$$
Combine this with the equality above to obtain
\begin{eqnarray}\label{eqn:thing0}
\E^{\rm{perm}}_{R,n}[ \# \cI_s | Y_n] &<& \frac{1}{2}.
\end{eqnarray}
By Markov's inequality,
$$ \P^{\rm{perm}}_{R,n}\big( \# \cI_s \ge 2 \E^{\rm{perm}}_{R,n}[ \#\cI_s|Y_n] \big| Y_n \big) \le \frac{1}{2}.$$
Multiply both sides of the inequality above by -1 and add 1 to obtain 
\begin{eqnarray*}
1/2& \le & \P^{\rm{perm}}_{R,n}\big( \# \cI_s < 2 \E^{\rm{perm}}_{R,n}[ \#\cI_s|Y_n] \big| Y_n \big) \le \P^{\rm{perm}}_{R,n}\big( \# \cI_s< 1 \big| Y_n \big).
\end{eqnarray*}
The second inequality above follows from (\ref{eqn:thing0}). Since $\P^{\rm{perm}}_{R,n}\big( \# \cI_s< 1 \big| Y_n \big) \le \frac{\P^{\rm{perm}}_{R,n}(\#\cI_s < 1)}{\P^{\rm{perm}}_{R,n}(Y_n)}$, multiply denominators in the inequality $1/2 \le \frac{\P^{\rm{perm}}_{R,n}(\#\cI_s < 1)}{\P^{\rm{perm}}_{R,n}(Y_n)}$ to obtain
$$ \P^{\rm{perm}}_{R,n}(Y_n) \le 2 \P^{\rm{perm}}_{R,n}\big( \# \cI_s < 1 \big).$$
However, $\P^{\rm{perm}}_{R,n}( \# \cI_s < 1 )$ tends  to zero as $n\to\infty$ by assumption. This shows $\P^{\rm{perm}}_{R,n}(Y_{n}) \to 0$ and therefore $\P^{\rm{perm}}_{r,n}(X_{R,r,s,n}) \to 1$ as $n\to\infty$.

To verify the last statement, let  $(\s,W) \in X_{R,r,s,n} \times 2^n$ be such that $W \in \cI_s(\s)$. Then
\begin{eqnarray*}
\P^{\rm{unif}}_{r,s,n}(\s,W) &=& \frac{1}{ \#\cI_s(\s) \#\Hom(\F_r,\sym(n))}\\
& \le & \frac{2 \E^{\rm{perm}}_{R,n}[\#\cI_s] }{ \E^{\rm{perm}}_{r,n}[\#\cI_s]\#\Hom(\F_r,\sym(n)) } \\
&=& \P^{\rm{plant}}_{r,s,n}(\s,W) \times 2 \E^{\rm{perm}}_{R,n}[\#\cI_s].
\end{eqnarray*}
These inequalities are justified in turn by the definition of $\P^{\rm{unif}}_{r,s,n}$, the assumption that $\s \in X_{R,r,s,n}$ and the definition of $\P^{\rm{plant}}_{r,s,n}$.
\end{proof}

Define $\MAXIND:\Hom(\F_r,\sym(n)) \to \R$ by 
$$\MAXIND(\s)=n^{-1} \max\{ \#W:~ W \textrm{ is an independent subset of } G(\s)\}.$$

\begin{prop}\label{prop:limit}
For all $r$ sufficiently large, there is a constant $\alpha(r)$ such that
$$\inf_{\d>0} \lim_{n\to\infty} \P^{\rm{perm}}_{r,n}(\MAXIND \in ( \alpha(r) - \d, \a(r)+\d) ) = 1.$$
Moreover, 
$$\alpha(r) = \frac{\log(r)}{r} + O(\log\log(r)/r).$$  
\end{prop}

\begin{remark}
The limit was proven to exist in \cite{MR3161470}. The asymptotic statement follows from earlier results of Frieze-Luczak \cite{MR1142268} in the case of the configuration model. Since the two models are contiguous \cite{MR1909503} this implies the proposition. An exact formula for $\alpha(r)$ (for sufficiently large $r$) was recently obtained by Ding-Sly-Sun \cite{MR3689942}.
\end{remark}

\begin{cor}\label{cor:planted}
Fix $\bs \in (0,1)$ and set $s=\bs \frac{\log(2r)}{r}$. Let $\bfsig_n \sim \P^{\rm{perm}}_{r,n}$. Then
$$n^{-1}\log \#\cI_s(\bfsig_n)   \ge f(r,s) - o_n(1) - O\left(\frac{\log(r)\log\log(r)}{r}\right)$$
with probability tending to $1$ as $n\to\infty$. To be precise, this means that there exists $\d(r) = O\left(\frac{\log(r)\log\log(r)}{r}\right)$ such that
for every $\eps>0$,
$$\lim_{n\to\infty} \P^{\rm{perm}}_{r,n}\left(n^{-1}\log \#\cI_s(\bfsig_n)   \ge f(r,s) - \eps - \d(r)\right)=1.$$
\end{cor}

\begin{proof}
 Let $R$ be the largest integer such that $\a(R) > s$. By Proposition \ref{prop:limit},
$$s < \a(R)=  \frac{\log(R)}{R} + O(\log\log(R)/R).$$ 
Observe that the function $x \mapsto \frac{\log(x)}{x}$ is monotone increasing for all $x\ge e$. So if $x \ge e$ then
\begin{eqnarray*}
0\le \frac{\log(x)}{x} - \frac{\log(x+1)}{x+1} &=&  \frac{\log(x)}{x} - \frac{\log(x)+\log(1+1/x))}{x+1} \\
&=& \frac{\log(x)}{x(x+1)} - \frac{\log(1+1/x)}{x+1} \le \frac{\log(x)}{x^2}.
\end{eqnarray*}
Since $\frac{\log(x)}{x^2} \le \frac{\log\log(x)}{x}$,  this implies
$$\a(R)-\a(R+1) = O(\log\log(R)/R)$$
and therefore we can improve the previous inequality to an equality: $s = \frac{\log(R)}{R} + O(\log\log(R)/R)$. 

By Theorem \ref{thm:f},
\begin{eqnarray*}
f(R,s) = \eta(s) - s^2R + O(\log(R)/R).
\end{eqnarray*}
Since $s = \frac{\log(R)}{R} + O(\log\log(R)/R)$,
\begin{eqnarray*}
\eta(s) &=&  \frac{\log^2(R)}{R} + O\left(\frac{\log(R)\log\log(R)}{R}\right) \\
s^2R &=& \frac{\log^2(R)}{R} + O\left(\frac{\log(R)\log\log(R)}{R}\right) .
\end{eqnarray*}
So $f(R,s) = O\left(\frac{\log(R)\log\log(R)}{R}\right)$. Since $x \mapsto \frac{\log(x)\log\log(x)}{x}$ is monotone decreasing (for all large enough $x$), this implies 
\begin{eqnarray}\label{R}
f(R,s) = O\left(\frac{\log(r)\log\log(r)}{r}\right).
\end{eqnarray}

Theorem \ref{thm:planted} implies that if $\bfsig_n \sim \P^{\rm{perm}}_{r,n}$ then with probability tending to $1$ as $n\to\infty$,
\begin{eqnarray*}\label{thing12}
n^{-1}\log \#\cI_s(\bfsig_n)   &\ge& n^{-1}\log\E^{\rm{perm}}_{r,n}[\#\cI_s] -  n^{-1}\log (2\E^{\rm{perm}}_{R,n}[\#\cI_s])\\
&=& f(r,s) - f(R,s) - o_n(1) = f(r,s) - o_n(1) - O\left(\frac{\log(r)\log\log(r)}{r}\right).
\end{eqnarray*}
\end{proof}

\subsection{Bounding clusters of independent sets}\label{sec:clusters0}

For $(\s,W) \in \Hom(\F_r,\sym(n))\times 2^n$, let $\Cluster_{s,\eps}(\s,W) \subset 2^n$ be the collection of independent subsets $W'$ of the graph $G(\s)$ such that $|\#W'/n-s|\le \eps$ and $|W\cap W'| \ge (s/2)n$. Informally, $\Cluster_{s,\eps}(\s,W)$ represents the cluster containing $W$ in $\bigcup_{s' \in (s-\eps,s+\eps)} \cI_{s'}(\s)$.

\begin{thm}\label{thm:indep}
Let $\bs \in (\frac{2 + \sqrt{2}}{4},1)$ and $s:= \bs \frac{ \log(2r)}{r}$. Then there exists $\g=\g(r,\bs)>0$, $\eps=\eps(r,\bs)>0$ and $0<b_1<b_2<s$, $b_2>s/2$, such that the following holds. Let $(\bfsig_n,\bfW_n)$ be random with law $\P^{\rm{unif}}_{r,s,n}$. Then with probability tending to 1 as $n\to\infty$,
\begin{enumerate}
\item there does not exist an independent set $W'$ of $G(\bfsig_n)$ such that $\#W'/n \in (s-\eps,s+\eps)$ and
$b_1n \le |\bfW_n \cap W'| \le b_2 n,$ 
\item 
$n^{-1} \log \#\Cluster_{s,\eps}(\bfsig_n,\bfW_n) \le f(r,s) - \g.$
\end{enumerate}
\end{thm}



\begin{proof}

The first claim was proven by Gamarnik-Sudan \cite[Theorem 2.6]{MR3359490} with the configuration model in place of the permutation model. This uses the hypothesis $\bs>\frac{2 + \sqrt{2}}{4}$. Since the two models are contiguous by \cite{MR1909503}, this implies the first claim. The proof of the second claim given below is modeled after the proof of an analogous result for sparse Erd\"os-Renyi graphs in \cite{MR3385742}. 

For $\s \in \Hom(\F_r,\sym(n))$, let $\cI_{s;\eps}(\s)$ be the set of all pairs $(W,W')$ satisfying: 
\begin{enumerate}
\item both $W,W'$ are independent subsets of $G_\s$,
\item $\#W = \lfloor sn\rfloor$, $|\#W'/n-s|\le \eps$,
\item $|W\cap W'| \ge (s/2)n$.
\end{enumerate}
The definition of the planted model implies:
\begin{eqnarray}\label{eqn:littlebear}
\E^{\rm{plant}}_{r,s,n}[\#\Cluster_{s,\eps}] = \frac{\E^{\rm{perm}}_{r,n}[\#\cI_{s;\eps}] }{\E^{\rm{perm}}_{r,n}[\#\cI_{s}]}.
\end{eqnarray}
By Theorem \ref{thm:F} and (\ref{eqn:littlebear}),
$$\lim_{n\to\infty} n^{-1} \log \E^{\rm{plant}}_{r,s,n}[\#\Cluster_{s,\eps}] = \max \{f(r,s,s',t) - f(r,s):~ s' \in [s-\eps,s+\eps], s/2\le t \le s\}.$$
By Theorem \ref{thm:continuity}, $f$ is uniformly continuous. So there exists $\eps=\eps(r,\bs)>0$ such that
\begin{eqnarray*}
&&\max_{ s' \in [s-\eps,s+\eps], s/2\le t \le s} f(r,s,s',t)  -  \max_{s/2\le t \le s} f(r,s,t)  \le \frac{\log(r)\log\log(r)}{r}.
\end{eqnarray*}
Assume from now on that $|s-s'|<\eps$. Then 
$$\lim_{n\to\infty} n^{-1} \log \E^{\rm{plant}}_{r,s,n}[\#\Cluster_{s,\eps}] = \max \{f(r,s,t) - f(r,s):~ s/2\le t \le s\} + O\left( \frac{\log(r)\log\log(r)}{r} \right).$$
By Corollary \ref{cor:maximization}, 
\begin{eqnarray}\label{eqn:thing10}
\lim_{n\to\infty} n^{-1} \log \E^{\rm{plant}}_{r,s,n}[\#\Cluster_{s,\eps}]  = f(r,s) -\bs(1-\bs)\frac{\log^2(r)}{r} + O\left( \frac{\log(r)\log\log(r)}{r} \right).
\end{eqnarray}

Next these estimates are transferred from the planted model to the uniform model. To simplify notation, let $X_n=X_{R,r,s,n}$ be as in Theorem \ref{thm:planted} where $R$ is the largest integer such that $\a(R)\ge s$. By Theorem \ref{thm:planted},
\begin{eqnarray*}
\E^{\rm{unif}}_{r,s,n}[\#\Cluster_{s,\eps}|X_n] &=& \P^{\rm{perm}}_{r,n}(X_n)^{-1}\E^{\rm{unif}}_{r,s,n}[\#\Cluster_{s,\eps}1_{X_n}] \\
&\le&   \P^{\rm{perm}}_{r,n}({X_n})^{-1}\E^{\rm{plant}}_{r,s,n}[\#\Cluster_{s,\eps}1_{X_n}] \times 2 \E^{\rm{perm}}_{R,n}(\#\cI_s)\\
&\le&   \P^{\rm{perm}}_{r,n}({X_n})^{-1}\E^{\rm{plant}}_{r,s,n}[\#\Cluster_{s,\eps}] \times 2 \E^{\rm{perm}}_{R,n}(\#\cI_s).
\end{eqnarray*}
Recall that $f(R,s) = \lim_{n\to\infty} n^{-1} \log \E^{\rm{perm}}_{R,n}(\#\cI_s)$. Combined with (\ref{eqn:thing10}), the fact that $ \P^{\rm{perm}}_{r,n}({X_n}) \to 1$ as $n \to \infty$ and equation (\ref{R}), this implies
\begin{eqnarray}\label{eqn:thing12}
\limsup_{n\to\infty} n^{-1} \log \E^{\rm{unif}}_{r,s,n}[\#\Cluster_{s,\eps}|{X_n}]  \le f(r,s) -\bs(1-\bs)\frac{\log^2(r)}{r} + O\left(\frac{\log(r)\log\log(r)}{r}\right).
\end{eqnarray}
By Markov's inequality, this implies the theorem with any choice of $\g=\g(r,\bs)$ satisfying
$$0<\g(r,\bs) < \bs(1-\bs) \frac{\log^2(r)}{r} + O\left(\frac{\log(r)\log\log(r)}{r}\right).$$

\end{proof}

\subsection{A variational principle}\label{sec:variational}

The next result will be used in the proof of Theorem \ref{thm:indep0} to obtain an invariant measure that is, in some sense, a subsequential limit of the uniform models. Its proof uses the same ideas that are behind Kerr-Li's proof of the Variational Principle for sofic entropy \cite{kerr-li-variational}.

\begin{prop}\label{prop:variational}
Let $\Si=\{\s_n\}_{n=1}^\infty$ be a sofic approximation to $\G$, $\cX$ a finite set and for each $n\in \N$, let $\sW_n \subset \cX^{V_n}$ be given. Then there exists an invariant measure $\mu \in \Prob_\G(\cX^\G)$ satisfying
\begin{eqnarray*}
h_\Si(\mu)&\ge& \inf_{\cO \ni \mu} \limsup_{n \to \infty} |V_n|^{-1}\log \# (\Omega(\cO,\s_n) \cap \sW_n )=  \limsup_{n \to \infty} |V_n|^{-1}\log \# \sW_n.
\end{eqnarray*}
\end{prop}

\begin{proof}
The inequality is trivial, so it suffices to prove the equality. 

The space $\Prob(\cX^\G)$ is compact and metrizable in the weak* topology. So fix a metric on $\Prob(\cX^\G)$ with diameter $\le 1$. Choose a sequence  $\sU^{(j)}$ ($j=1,2,\ldots$) of finite open covers of $\Prob_\G(\cX^\G) \subset \Prob(\cX^\G)$ by open balls of radius $2^{-j}$. For convenience, set  $\sU^{(0)}=\{\Prob(\cX^\G)\}$. 

We will inductively construct a sequence  $\{\cO^{(j)}\}_{j=0}^\infty$ of open sets $\cO^{(j)} \in \sU^{(j)}$ and positive constants $\{m^{(j)}\}_{j=1}^\infty$ satisfying: for every $J \in \N$,
\begin{enumerate}
\item $\cO^{(j)} \cap \cO^{(j+1)} \ne \emptyset$ for all $0\le j <J$,
\item for every $0\le j \le J$, there exist infinitely many $n$ such that
$$\#(\sW_{n} \cap \Omega(\cO^{(j)},\s_{n})) \ge \#\sW_{n}/m^{(j)}.$$
\end{enumerate}
For the base case, set $\cO^{(0)}= \Prob(\cX^\G)$ and $m^{(0)}=1$. For induction, suppose there is some $J \ge 1$, open subsets $\cO^{(1)}, \ldots, \cO^{(J)}$ and constants $\{m^{(j)}\}_{j=1}^J$ satisfying the criteria above. 

For each $n$, let $\cO^{(J+1)}_{n} \in \sU^{(J+1)}$ be an open set such that $\cO^{(J+1)}_{n}  \cap  \cO^{(J)} \ne \emptyset$  and 
$$\#(\sW_{n} \cap \Omega(\cO^{(J+1)}_n,\s_{n})) \ge \#(\sW_{n} \cap \Omega(\cO',\s_{n}))$$
for all $\cO' \in \sU^{(J+1)}$ such that $\cO' \cap  \cO^{(J)} \ne \emptyset$.  Let $\cO^{(J+1)} \in \sU^{(J+1)}$ be an open subset such that there exists infinitely many $n$ satisfying:
\begin{eqnarray*}
\cO^{(J+1)} &=& \cO^{(J+1)}_{n}\\
\#(\sW_{n} \cap \Omega(\cO^{(J)},\s_{n})) &\ge& \#\sW_{n}/m^{(J)}.
\end{eqnarray*}
Let
 $$m^{(J+1)}  := m^{(J)} \left|\sU^{(J+1)}\right|.$$
Since the number of open sets $\cO' \in \sU^{(J+1)}$ that intersect $\cO^{(J)}$ nontrivially is at most $\left|\sU^{(J+1)}\right|$ and 
$$\Omega\left(\cO^{(J)},\s_{n})\right) \subset \bigcup\left\{ \Omega(\cO',\s_{n}):~ \cO' \cap \cO^{(J)}  \ne \emptyset\right\},$$
it follows that for infinitely many $n$,
\begin{eqnarray*}
\#(\sW_{n} \cap \Omega(\cO^{(J+1)},\s_{n})) &\ge& |\sU^{(J+1)}|^{-1} \#\left(\sW_{n} \cap \Omega\left(\cO^{(J)},\s_{n}\right)\right) \\
&\ge& \frac{1}{ m^{(J)} |\sU^{(J+1)}| } \#\sW_{n}.
\end{eqnarray*}
 This proves the inductive step and the claim.
 
 Since each $\sU^{(j)}$ is a covering by balls of radius $2^{-j}$, if $\mu_j \in \cO^{(j)}$ is arbitrary, then $\{\mu_j\}_j$ is a Cauchy sequence. Let $\mu = \lim_{j\to\infty} \mu_j$.  If $\cO \subset \Prob_\G(\cX^\G)$ is any open subset containing $\mu$ then $\cO$ contains $\cO^{(j)}$ for some $j$.  Thus
  \begin{eqnarray*}
&& \limsup_{n \to \infty} |V_n|^{-1}\log \# \sW_n  \ge \inf_{\cO \ni \mu} \limsup_{n \to \infty} |V_n|^{-1}\log \# (\Omega(\cO,\s_n) \cap \sW_n )\\
 &\ge &  \limsup_{n \to \infty} |V_n|^{-1}\log \# (\Omega(\cO^{(j)},\s_n) \cap \sW_n )\ge   \limsup_{n \to \infty} |V_n|^{-1}\log \frac{\# \sW_n}{m^{(j)}} \\
        &=&  \limsup_{n \to \infty} |V_n|^{-1}\log \# \sW_n.    
  \end{eqnarray*}
  This proves the equality.




\end{proof}

\subsection{Proof of Theorem \ref{thm:indep0}}\label{sec:final}

\begin{proof}[Proof of Theorem \ref{thm:indep0}]


Choose constants $\bs,s,r,b_1,b_2,\eps,\g$ satisfying Theorem \ref{thm:indep}. By choosing $\eps>0$ smaller if necessary it may be assumed that 
$\eps/5< b_2-b_1 \textrm{ and } (\eps/10)\log(2) + H(\eps/10, 1-\eps/10)  < \g/6$
where $H(a,b) = -a\log(a) - b\log(b)$ for any $a, b>0$. 

Given a homomorphism $\s:\F_r \to \sym(n)$, let $\sW(\s) \subset \cI_s(\s)$ be the collection of independent subsets $W$ satisfying
\begin{enumerate}
\item there does not exist an independent set $W'$ of $G(\s_n)$ such that $\#W'/n \in (s-\eps,s+\eps)$ and $b_1n \le |W \cap W'| \le b_2 n$,
\item 
\begin{eqnarray}\label{cluster}
n^{-1} \log(\#\Cluster_{s,\eps}(\s_n,W)) \le f(r,s) - \g/2. &&
\end{eqnarray}
\end{enumerate}
An independent set $W$ of $G(\s_n)$ is identified with its indicator function $1_W \in \{0,1\}^n$. So by abuse of notation one may consider $\sW(\s)$ to be a subset of $\{0,1\}^n$.

By Corollaries \ref{thm:sofic-random}, \ref{cor:planted} and Theorem \ref{thm:indep} there exists a sofic approximation $\Si=\{\s_n\}_{n \in \N}$ with $\s_n:\F_r \to \sym(n)$ such that if $\sW_n:=\sW(\s_n)$ then 
$$\limsup_{n\to\infty} n^{-1} \log(\#\sW_n) \ge f(r,s).$$
By Proposition \ref{prop:variational}, there exists an invariant measure $\mu \in \{0,1\}^{\F_r}$ such that 
\begin{eqnarray*}
h_\Si(\mu)&\ge& \inf_{\cO \ni \mu} \limsup_{n \to \infty} n^{-1}\log \# (\Omega(\cO,\s_n) \cap \sW_n )\\
&=&  \limsup_{n \to \infty} n^{-1}\log \# \sW_n \ge f(r,s).
\end{eqnarray*}

For $g \in \F_r$, let $\bfpi_g:\{0,1\}^{\F_r} \to \{0,1\}$ be the coordinate projection. Let $\cO_2$ be the set of measures $\nu \in \Prob(\{0,1\}^{\F_r})$ satisfying
\begin{enumerate}
\item $\nu( \bfpi_e = 1) \in (s-\eps/2, s + \eps/2)$
\item $\nu( \textrm{either } (\bfpi_e, \bfpi_{a_i}) = (1,1) \textrm{ or } (\bfpi_e, \bfpi_{a^{-1}_i}) = (1,1) \textrm{ for some $i$}) < \eps/10$.
\end{enumerate}
Let $0<\k_2$ be a constant with $b_2 - \k_2 - \eps/5 > b_1$. 

It suffices to show that if $0<\k_1\le \k_2$ is any constant and $\cO_1 \subset \Prob(\{0,1\}^{\F_r})$ is any open neighborhood of $\mu$ with $\cO_1 \subset \cO_2$ then 
$$\limsup_n n^{-1} \log \dim_\Q( H_0(\cO_1,\cO_2,\k_1,\k_2,\s_n) \otimes_\Z \Q) \ge \gamma/3.$$
In the notation above, the superscript $L$ is omitted because all $0$-cycles are finite sums of length one $0$-cycles. So the parameter $L$ is irrelevant to studying $0$-dimensional homology.

For $x \in \Omega(\cO_1,\s_n) \cap \sW_n$, let $C(x)$ be the set of all $y \in \Omega(\cO_1,\s_n)$ such that
$$x-y \in B_0(\cO_2,\k_2,\s_n).$$
In other words, $y \in C(x)$ if and only if there exists a path $x=x_0,x_1,\ldots, x_k = y$ such that $x_i \in \Omega(\cO_2,\s_n)$ for all $i$ and $d(x_i,x_{i+1}) < \k_2$. Observe that $\Omega(\cO_1,\s_n) \cap \sW_n$ is the disjoint union sets of the form $C(x) \cap \sW_n$. 

We will estimate the cardinality $\#C(x)$ by showing that $C(x)$ is contained in the $\eps/10$-neighborhood of $\Cluster_{s,\eps}(\s_n, W_1)$. So suppose $y \in C(x)$ and let $x_0,x_1,\ldots, x_k=y$ be a path as above. Because $x_i \in \Omega(\cO_2,\s_n)$, its empirical measure is in $\cO_2$, which implies that the ``bad'' set
$$B(x_i):=\{v \in [n]:~x(v)=1 \textrm{ and either } x(\s(a_i)v) =1 \textrm{ or } x(s(a_i^{-1})v) = 1 \textrm{ for some } i\}$$
has cardinality at most $(\eps/10)n$. Let $W_i = x_i^{-1}(1) \setminus B(x_i)$. Then $W_i$ is an independent subset of $G(\s_n)$ with density $|W_i|/n \in (s-\eps,s+\eps)$.  Observe that $x^{-1}(1) = W_1$ since $x \in \sW_n$ is an independent set. 

Let $d_{\textrm{Hamm}}$ denote the normalized  Hamming metric on subsets of $[n]$. So
$$d_{\textrm{Hamm}}(W,W') = |W \vartriangle W'|/n.$$
Then
$$d_{\textrm{Hamm}}(W_i,W_{i+1}) \le d(x_i,x_{i+1}) + |B(x_i)|/n + |B(x_{i+1})|/n < \k_2 + \eps/5.$$
for all $i$. 

Let $\Cluster(x)$ be the collection of all independent sets $W$ of $G(\s_n)$ such that $|W|/n \in (s-\eps,s+\eps)$ and $|W \cap W_1|/n \ge b_2$. I claim that $W_k \in \Cluster(x)$. If not, then there exists a smallest number $j \ge 2$ such that $W_j \notin \Cluster(x)$. Since $W_{j-1} \in  \Cluster(x)$, $|W_{j-1} \cap W_1|/n \ge b_2$. Thus 
$$|W_j \cap W_1|/n \ge |W_{j-1} \cap W_1|/n - d_{\textrm{Hamm}}(W_{j-1},W_j) >  b_2 - \k_2 - \eps/5 > b_1.$$
 Since $x \in \sW_n$, the definition of $\sW_n$ implies $|W_j \cap W_1|/n \ge b_2$ which implies $W_j \in \Cluster(x)$. This contradiction proves the claim.

Since $W_k \in \Cluster(x)\subset \Cluster_{s,\eps}(\s_n,W_1)$ and $|y^{-1}(1) \vartriangle W_k|\le (\eps/10)n$, it follows that $C(x)$ is contained in the $\eps/10$-neighborhood of $\Cluster_{s,\eps}(\s_n, W_1)$. Thus 
$$\#C(x) \le 2^{(\eps/10)n}{n \choose \lceil \eps n/10 \rceil}  \#\Cluster_{s,\eps}(\s_n,W_1).$$
Combine the previous inequality with (\ref{cluster}) to obtain
$$n^{-1} \log \#C(x) \le (\eps/10)\log(2) + H(\eps/10, 1-\eps/10)  + f(r,s) - \g/2 + o_n(1) < f(r,s)- \g/3 + o_n(1).$$

Since $\Omega(\cO_1,\s_n) \cap \sW_n$ is a disjoint union of sets of the form $C(x) \cap \sW_n$ and 
$$\limsup_{n\to\infty} n^{-1} \log \# (\Omega(\cO_1,\s_n) \cap \sW_n) \ge f(r,s),$$
it follows that the number of different subsets of the form $C(x)$ for $x \in \Omega(\cO_1,\s_n) \cap \sW_n$ is at least $e^{ (\g/3)n}$ up to subexponential factors. However each subset of the form $C(x)$ contributes a dimension to the homology group $H_0(\cO_1,\cO_2,\k_1,\k_2,\s_n) \otimes_\Z \Q$. Thus
$$\limsup_{n\to\infty} n^{-1}\log \dim_\Q (H_0(\cO_1,\cO_2,\k_1,\k_2,\s_n) \otimes_\Z \Q) \ge \g/3>0.$$

\end{proof}

\section{Questions}\label{sec:questions}


\begin{enumerate}

\item Does there exist an action whose sofic homology does not vanish in dimension $1$ or in some higher dimension? Does there exist such an example which is a Markov chain over a free group? 

\item How does the $d$-dimensional sofic homology change under standard operations or perturbations of group actions, such as taking a direct product, passing to a subgroup of the acting group, coinducing from an action of a subgroup, ergodic decomposition, direct limits, inverse limits, taking a weak* limit of invariant measures or a $d$-bar limit?  

\item Is $b_{\Si,0}(\mu)$ a continuous or semi-continuous function of $\mu \in \Prob(\cX^\G)$ if $\cX$ is finite? Given a positive number $0<t<\log |\cX|$ does there exist an invariant measure $\mu \in \Prob_\G(\cX^\G)$ with $b_{\Si,0}(\mu)=t$? 

\item It is well-known that sofic entropy can increase under a factor map. To correct for this, several authors have defined the sofic entropy of a factor relative to the source. This notion was inspired by Kerr's approach to sofic entropy in \cite{kerr-partitions} and has been variously called outer sofic entropy, extension entropy or entropy in the presence  \cite{MR3993930, MR3635672, hayes-relative-entropy, MR3959054, seward-weak-containment}. It seems likely that there should be an analogous definition of outer sofic homology. If so, this might be useful for defining relative sofic homology.

\item Let $\K$ be a compact abelian group and $\G$ a countable group. Then $\K$ is identified with the subgroup of constants of $\K^\G$. This subgroup is $\G$-invariant and therefore $\K^\G/\K$ is a compact abelian group on which $\G$ acts by automorphisms. The action $\G \cc \K^\G/\K$ is called a {\bf Popa factor}. In \cite{MR3543677} Tim Austin proved that if $\G$ has property (T), is residually finite and $\K=\R/\Z$ then there is a sofic approximation $\Si$ relative to which the Popa factor does not have connected model spaces. This means that its $0$-dimensional sofic homology is not trivial. In spite of this, I conjecture that $b_{\Si,0}(\G \cc \K^\G/\K)=0$ for any sofic approximation $\Si$.  In  fact, my guess is that the sofic homology groups satisfy a bound of the form:
$$\dim_\Q(H_0(\cO_1,\cO_2,\k_1,\k_2,\s_n) \otimes_\Z \Q)\le B$$
where $B$ does not depend on $n$ (but is allowed to depend on everything else). 

\item I conjecture that any strongly ergodic distal action $\G \cc (X,\mu)$ has the property that its $0$-dimensional sofic homology is not trivial. These actions have zero $\Si$-entropy and therefore have $0$-th exponential Betti number equal to zero. My guess is that, like the Popa factors, such actions satisfy a constant bound on the growth of their homology.

 \item Fix a sofic approximation $\Si=\{\s_i:\G\to \sym(V_i)\}$. Let $F_i:\Abel\to \R$ be the function $F_i(G) = \frac{\log\dim_\Q(G \otimes_\Z \Q)}{\log \#V_i}$. The invariant $F_{\Si,d}(\mu)$ defined in Corollary \ref{cor:monotone} is the polynomial growth rate of the $d$-dimensional sofic homology. Given $t>0$ and a dimension $d\ge 0$ does there exist an invariant measure $\mu$ with  $F_{\Si,d}(\mu)=t$?

\item The present paper shows that free groups of sufficiently large rank admit actions without the Weak Pinsker Property. Does the same result hold for all non-amenable sofic groups?

\item If the $0$-dimensional sofic homology of an ergodic action vanishes, then does the action have the Weak Pinsker Property?

\item Does there exist an ergodic action with positive entropy that has no nontrivial direct Bernoulli factors? I conjecture that the frozen model introduced in \cite{MR3689942} has this property.


\end{enumerate}

\appendix





\bibliography{biblio}
\bibliographystyle{abbrv.bst}

\Addresses

\end{document}